\documentclass[11pt]{amsart}
\usepackage{amsmath,amsthm,amssymb,amsfonts}
\ifx\pdfoutput\undefined \usepackage[hypertex]{hyperref} \else \usepackage[pdftex,pdfstartview=FitH]{hyperref} \fi
\numberwithin{equation}{section}
\makeindex
\newcommand{\cov}{\mathrm{cov}}
\newcommand{\SO}{\mathrm{SO}}
\newcommand{\Trace}{\mathrm{Trace}}
\newcommand{\compactumone}{F}
\newcommand{\Sym}{\mathrm{Sym}}
\newcommand{\Sball}{B_S}
\renewcommand{\clubsuit}{F}
\newcommand{\compactumfortyone}{K}

\newcommand{\Xgood}{X_{\mathrm{good}}}
\newcommand{\isogeny}{\Phi}
\newcommand{\wk}{\mathrm{wk}}
\newcommand{\uglytempered}{q}
\newcommand{\Omegaone}{\Omega}
\newcommand{\Omegatwo}{\Omega'}
\newcommand{\Omegathree}{\Omega''}

\newcommand{\tempered}{p_G}
\newcommand{\Xcompact}{X_{\mathrm{cpct}}}

\newcommand{\supp}{\mathrm{supp}}
\newcommand{\Omicron}{\Upsilon}

\newcommand\xappa\kappa
\newcommand\yota\iota
\newcounter{consta}
\renewcommand{\theconsta}{{\xappa_{\arabic{consta}}}}
\newcounter{constb}
\renewcommand{\theconstb}{{\yota_{\arabic{constb}}}}
\newcounter{constc}[section]
\renewcommand{\theconstc}{{c_{\arabic{constc}}}}
\newcommand{\consta}{\refstepcounter{consta}\theconsta}
\newcommand{\constb}{\refstepcounter{constb}\theconstb}
\newcommand{\constc}{\refstepcounter{constc}\theconstc}

\newcommand{\Tr}{\mathrm{Tr}}

\newcommand{\mysymbol}{\zeta}

\newcommand{\Proj}{\operatorname{Pr}}
\newcommand{\Scomp}{\mathfrak{r}}

\newcommand{\Sob}{\mathcal{S}}

\newcommand{\Ad}{\mathrm{Ad}}

\newcommand{\disc}{\mathrm{disc}}
\newcommand{\h}{\mathfrak{h}}

\newcommand{\Lie}{\mathrm{Lie}}

\newcommand{\g}{\mathfrak{g}}
\newcommand{\C}{\mathbb{C}}

\newcommand{\Siegel}{\Sieg}

\newcommand{\vol}{\operatorname{vol}}

\newcommand{\G}{\mathbf{G}}

\newcommand{\Sieg}{\mathfrak{S}}
\newcommand{\height}{\mathrm{ht}}

\newcommand{\GL}{\mathrm{GL}}

\newcommand{\PGL}{\mathrm{PGL}}

\newcommand{\SL}{\mathrm{SL}}

\newcommand{\R}{\mathbb{R}}

\newcommand{\Z}{\mathbb{Z}}
\newcommand{\Q}{\mathbb{Q}}

\newcommand{\adele}{\mathbb{A}}

\newcommand{\dist}{\mathrm{dist}}
\newcommand{\propGname}{Proposition C1 }
\newcommand{\propHname}{Proposition C2 }
\DeclareFontFamily{OT1}{rsfs}{}
\DeclareFontShape{OT1}{rsfs}{n}{it}{<-> rsfs10}{}
\DeclareMathAlphabet{\mathscr}{OT1}{rsfs}{n}{it}
\swapnumbers
\newtheorem{thm}[subsection]{Theorem}
\newtheorem{lem}[subsubsection]{Lemma}
\newtheorem*{sublem}{Sublemma}
\newtheorem{prop}[subsection]{Proposition}
\newtheorem*{lem*}{Lemma}
\newtheorem*{prop*}{Proposition}
\newtheorem*{propA}{Proposition A}
\newtheorem*{propB}{Proposition B}
\newtheorem*{propC}{Proposition C}
\newtheorem*{propD}{Proposition D}
\newtheorem*{propE}{Proposition E}
\newtheorem*{propF}{Proposition F}
\newtheorem*{propG}{Proposition C1}
\newtheorem*{propH}{Proposition C2}

\newcommand{\myparagraph} {\refstepcounter{subsection}\subsection*{\textnormal{\thesubsection}}}

\newtheorem{cor}{Corollary}[subsection]

\begin{document}

\title[Effective equidistribution for closed orbits]{Effective equidistribution
for closed orbits of semisimple groups on homogeneous spaces.}
\author{M. Einsiedler, G. Margulis, A. Venkatesh}

\begin{abstract} We prove effective equidistribution, with polynomial rate, for large closed orbits
of semisimple groups on homogeneous spaces, under certain technical restrictions (notably,
the acting group should have finite centralizer in the ambient group). 
The proofs make extensive use
of spectral gaps, and also of a closing lemma for such actions.  \end{abstract}

\maketitle

\tableofcontents
 
 \newpage

\section{Introduction.}

\subsection{General introduction.} \label{gi}
Let $G$ be a real Lie group, and let $\Gamma$ be a
lattice in G. Let $U$ be a connected unipotent subgroup of $G$. A theorem of
Ratner \cite{Ratner-measure-rigidity}, which proved a conjecture of S.~G.~Dani,  classifies the
$U$-invariant ergodic probability measures $\mu$ on $X=\Gamma\backslash G$:
namely, they are $S$-invariant probability measures
supported on closed $S$-orbits, where $S$ is a closed subgroup containing $U$.

This theorem is a fundamental result with numerous applications. In particular,
in combination with results on nondivergence of unipotent orbits, and with
the linearization technique, Mozes and Shah \cite{Mozes-Shah} proved that the nonzero weak$^*$-limits
of $U$-invariant ergodic probability measures are again ergodic, and
therefore are Haar measures on closed orbits.

Let $H \subset G$ be a semisimple subgroup generated by unipotent elements.  A special case of the results of \cite{Mozes-Shah}
asserts that:
\begin{multline}\label{MS} \mbox{The only nonzero weak$^*$-limits of normalized Haar measures on
closed }
\\ \mbox{ $H$-orbits are Haar measures on 
closed orbits of a closed subgroup $S \supset H$}.\end{multline}

In the present paper we present a {\em polynomially effective} version of 
\eqref{MS}
under the assumptions that $G$ is semisimple, $\Gamma$ is an
arithmetic lattice and $\Lie(H)$ is centralizer-free inside
$\Lie(G)$. In explicit terms, we
 show that the Haar measure on a closed $H$-orbit is
close to the Haar measure on a closed $S$-orbit
for a closed subgroup $S \supset H$.  Moreover, the notions of ``close''
improves as a power of the volume of the relevant orbits.

We observe that, often, ``direct'' effectivizations of ergodic proofs
only give logarithmically poor (or worse) rates of convergence. 
The content of our result is, therefore, in the
polynomial rate. 

A key input is used from the theory of automorphic forms: namely, the $H$-action
on $L^2(\mu)$ has a {\em spectral gap} (if $\Gamma$ is as in (1) below,
and $\mu$ is the $H$-invariant probability measure on a closed $H$-orbit). 
This implies an effective version of the ergodic theorem for $H$-actions.

We say that a unitary representation of a 
(connected Lie) group $G$
possesses a spectral gap if there exists a compactly supported probability measure on $G$,
so that convolution with this measure has operator norm strictly less than $1$. 
For a detailed discussion of this concept for semisimple real Lie groups, 
see \S \ref{spectralgap}.

\subsection{Some technical assumptions.} \label{technical}
We shall assume that:
 \begin{enumerate}
 \item There is a semisimple $\Q$-group $\G$ so that $G=\G(\R)$ and $\Gamma$ is a congruence subgroup of $\G(\Q)$;
 \item $H^{+} = H$, i.e. $H$ has no compact factors and is connected;
 \item  The centralizer of $\mathfrak{h} = \mathrm{Lie}(H)$ in $\mathfrak{g}
= \mathrm{Lie}(G)$ is trivial.   \end{enumerate}

The third assumption means that closed $H$-orbits are ``rigid,'' i.e.,
do not come in continuous families. Moreover, it implies (Lemma \ref{assumption-2}) that
there are only finitely many intermediate subgroups $H \subseteq S \subseteq G$,
each of which is semisimple (i.e. the Lie algebra $\mathfrak{s}$ is a semisimple real Lie
algebra).

See \S \ref{generalization} for a discussion concerning lifting these assumptions. In particular, ``congruence'' may be weakened to ``arithmetic'',
at the cost of allowing the exponents in the theorem to depend on the lattice. 
The  assumption of finite centralizer is primarily to ``rigidify''
the situation; while it does not seem essential to the method, it
is clear that many new complications arise when it is removed. 

Examples:
\begin{enumerate}
\item $G= \mathrm{SL}_n(\R)$, $H=\mathrm{SL}_2(\R)$ embedded via the irreducible representation.
If $n \geq 4$, there exists intermediate subgroups
$H \subsetneq S \subsetneq G$: for instance, if $n$ is even,
the $H$-action on $\R^n$ preserves a symplectic form, and one may take for $S$
its stabilizer.

\item $G = \mathrm{SL}_n(\R)$, $H = \mathrm{SO}(q)^{(0)}$ for some indefinite
quadratic form $q$ in $n$ variables.  In this case, there do not exist
intermediate subgroups besides those with connected component $H$. (See \S \ref{arithappl} for an arithmetic application of the theorem in this setting). 

 \item  $G = \mathrm{SL}_2(\R)^n$, $H = \mathrm{SL}_2(\R)$ embedded diagonally.
Here there are many intermediate subgroups as long as $n \geq 3$.
\item $G= \SL_7(\R)$, $H= G_2(\R)$ the group of $\R$-points of the $\R$-split form of $G_2$ embedded into $G$. Here $\SO_7$ is an intermediate subgroup. (See \S \ref{arithappl} for an arithmetic application of the theorem in this setting). 

\end{enumerate}

%               THEOREM

\begin{thm} \label{thm:main} \hypertarget{thm:main}{}
Let $\Gamma,H\subset G$ be as above satisfying the assumptions in \S \ref{technical}.
Let $\mu$ be the $H$-invariant probability measure on a closed $H$-orbit
$x_0H$ inside $X=\Gamma\backslash G$.

There exists $\delta, d > 0$ depending only on $G,H$
and $V_0 > 0$ depending only on $\Gamma, G, H$ with the following properties. 
For any $V \geq V_0$ there exists
an intermediate subgroup $H \subseteq S
\subseteq G$ for which
$ x_0 S$ is a closed $S$-orbit with volume $\leq V$
and such that $\mu$ is $V^{-\delta}$-close to $\mu_{x_0 S}$, i.e.\ for any $f\in
C^{\infty}_c(X)$ we have
\begin{equation} \label{thmequation}\biggl|\int f
\operatorname{d}\!\mu-\int f \operatorname{d}\!\mu_{x_0 S}\biggr|< V^{-\delta} \Sob_d(f),
\end{equation}
where $\Sob_d(f)$ denotes an $L^2$-Sobolev norm of degree $d$
(see \eqref{sobnormdef}).
\end{thm}

Let us make certain comments on this result.
\begin{itemize}
\item The Sobolev norms $\Sob_d(f)$ measure the $L^2$-norm of
the derivatives of $f$ up to the $d$-th order w.r.t.\ a smooth measure on $X$. If $X$ is compact,
this defines them up to bounded multiples.

An equivalent way of formulating \eqref{thmequation}, without the use of Sobolev norms, would be in terms of the measure of small balls. For simplicity, consider the case when $X$ is compact
and fix a Riemannian metric on $X$. There exists $\delta_1, V_1>  0$ so that,
for any $V \geq V_1$ and
for any Riemannian ball $B(r)$ of radius $r \leq 1$,
we have
$$\left| \mu(B(r)) - \mu_{x_0 S}(B(r))  \right| \leq V^{-\delta_1}$$
Such an assertion is equivalent to \eqref{thmequation}, up to modification of constants.

\item In interpreting the theorem -- which at first sight is rather confusing
owing to the parameter $V$ --  it may be helpful to understand how it implies the
{\em non-effective} variant,  viz. theorem of Mozes and Shah \eqref{MS}.

We take for granted the following fact, for any intermediate subgroup $H\subset S \subset G$:
\begin{equation} \label{finiteness}
\mbox{There are finitely many closed $S$-orbits with volume $\leq V$.}
\end{equation}
 An effective version of \eqref{finiteness}, with polynomial bounds,
 is given in \S \ref{sec:quantiso}, but the qualitative statement is simple to establish.

Let $X=\Gamma\backslash G$ be as in Theorem \ref{thm:main}. Suppose we are given a sequence $x_i H$ of closed $H$-orbits
so that the sequence $\mu_{x_i H}$ converges to a weak$^*$-limit $\nu_\infty$.  

Fix, for a moment,
an intermediate subgroup $S \supset H$. 
Let us agree to set $\vol(x_i S) = \infty$ if $x_i S$ is not closed. 
Passing to a subsequence, we may assume that either there
exists some $M > 0$ such that the volume of $x_i S$ is less than $M$ for all $i$, or
the volume of $x_i S$ approaches infinity as $i\to\infty$. 

Since there are finitely many intermediate subgroups, we may assume, passing
to a further subsequence, that one or the other option holds for all such $S$.  That is to say: there is
a subset $\mathscr{S}$ of the set of intermediate subgroups $H \subset S \subset G$
so that $x_i S$ is closed and of volume $\leq M$ when $S \in \mathscr{S}$; 
and of volume approaching $\infty$, when $S \notin \mathscr{S}$.

Passing to a further subsequence, and applying \eqref{finiteness}, we may suppose
that, for any $S\in\mathscr{S}$, there exists $y_S$ so that $x_i S = y_S S$ for all $i$.
Apply the theorem with fixed $V >M$. It follows that $\nu_{\infty}$
is $V^{-\delta}$-close to a measure of the form $\mu_{y_S S}$ for some $S \in \mathscr{S}$.
Since $V$ was arbitrary, our conclusion follows.

\item Our proof gives the same result for any $H$-invariant measure $\mu$
so that the $H$-action on $L^2(\mu)$ possesses a spectral gap. In other terms,
the true content of the theorem is the ``effective classification of {\em strongly} ergodic $H$-invariant measures.'' The only naturally arising
context, however, seems to be the case described in the theorem.

\item {\em In principle}, the exponent $\delta$ is computable; however, it would be rather painful to compute and extremely small in general. 

 In the case when $H$ is maximal, it is likely
not too difficult to compute and not too small. 
However, our goal has been only to achieve polynomial dependence. 
This is of qualitative significance, for ``abstract'' effectivizations
of results in ergodic theory yield, in general, far poorer bounds. 

\item The theorem also has a purely topological version that we now enunciate.
%For simplicity, we state it in the case when $X$ is compact. 
%\marginpar{Rephrase in terms of balls on $S$??}
\end{itemize}

\begin{thm}\label{mcor1}
Notations as in Theorem \ref{thm:main}, let $\compactumone \subset X$ be compact,
and fix a Riemannian metric on $X$. There exists $\rho > 0$, depending on $G$, $H$,
and $V_1$ depending on $\Gamma$, $G$, $H$, $\compactumone$, and the choice of Riemannian metric,
with the following property:

For any $V \geq V_1$, there exists $S$ so that $x_0 S$ is a closed orbit of volume $\leq V$,
and any point in $x_0 S \cap \compactumone$ is at distance $\leq V^{-\rho}$ from $x_0 H$, w.r.t.\
 the Riemannian metric induced on $x_0 S$
from $X$.
\end{thm}

By the {\em Riemannian metric induced on $x_0 S$}, 
we mean: endow $x_0 S$ with the Riemannian structure induced
as a submanifold of $X$, and compute the corresponding metric.

Let us note in particular the import of Theorem \ref{mcor1} in the
case when $H$ is a {\em maximal} subgroup in $G$. Then
 there exists $\ell \geq 1$ so that a closed $H$-orbit
of volume $\geq \varepsilon^{-\ell}$ has to intersect any $\varepsilon$-ball in $X$. 
In particular, we also prove that the number of closed $H$-orbits 
that fail to intersect some $\varepsilon$-ball is bounded by a polynomial in $\varepsilon^{-1}$.

\subsection{Outline of the proof.  } \label{outline}

First of all, let us consider the {\em non-effective} version of Theorem \ref{thm:main}, i.e. \eqref{MS}. 
This special case of measure rigidity (when the acting group is semisimple)
admits a simple proof.
We present it in \S \ref{semisimple-measures}. One may then
see the proof of Theorem \ref{thm:main} as an effective form of that proof.\footnote{In this case, the process of {\em effectivization} causes proofs to increase
in length approximately twenty-fold!} We suggest the reader
begin by reading this non-effective proof.

In \S \ref{proof}, we present a detailed outline of the proof of the main Theorem, comparing
each portion of the proof to the corresponding non-effective statement in \S \ref{semisimple-measures}.
The reader may wish to read \S \ref{proof} after \S \ref{semisimple-measures}. 

For now we present a very brief high-level outline of the proof strategy:
We show for increasingly large intermediate subgroups $H \subseteq S \subseteq G$, 
 {\em either} $\mu$ is supported on an $S$-orbit of small volume, {\em or} $\mu$ is almost invariant by a strictly larger subgroup $S' \supset S$.

Indeed, take a one-parameter
unipotent subgroup $U \subset H$ and say a point $x \in X$ is
generic if the measure along long segments
of $x.U$ approximates well the measure $\mu$. The proof of the dichotomy
uses the following principles:
\begin{enumerate}
\item (Ergodic theorem). Most points $x \in X$ are generic. 
\item (Nearby generic points give additional invariance).
If $x,x' \in X$ are two generic points which are very close, i.e.
$x' = x g$ for some $g$ near the identity, and $g \notin S$,
then $\mu$ is almost invariant by a larger subgroup $S' \subset S$. 
\item (Dichotomy). If we {\em cannot} find $x,x'$
as above, then necessarily $\mu$ was supported on a 
closed $S$-orbit of small volume. 
\end{enumerate}
Step (2) is easily effectivizable, relying only on 
polynomial behavior of unipotent trajectories (see \S \ref{existing-work:HD}). 

 ``Good'' effective forms of both (1) and (3), however, make use of spectral gaps. 
Specifically, we use the fact that, for any intermediate subgroup $S \supset H$,
the action of $S$ on $L_0^2(x S)$, where $x S$ varies through all
closed $S$-orbits, has a uniform spectral gap. (Here $L_0^2$ denotes the orthogonal complement of locally constant functions.)
The crucial ingredient in providing the spectral gaps we need is the work of Clozel on ``property $\tau$.''
In effect, this guarantees\footnote{The terminology ``property $\tau$'' is perhaps confusing. In general, it asserts that group $Q$ has a uniform spectral gap in its action on some natural family of unitary representation. Therefore, unlike property $T$,  which is an intrinsic property of a group $G$, 
property $\tau$ only makes sense by reference to an implicit choice of such a family.}
  a uniform spectral gap for $Q$ acting on $L^2(\Lambda \backslash Q)$,
when $Q$ is a semisimple group and $\Lambda$ varies through {\em congruence} lattices within $Q$. 
(See also \S \ref{sarithmetic}.)

We also make use of several techniques that have become, to some extent,
standard in the theory of homogeneous dynamics. Most notably we use (as noted) polynomial
divergence for unipotent actions and the related linearization argument, which we will
discuss in \S \ref{existing-work:HD}.  These ideas will enter both in steps (2) and (3). 

Finally, the process of ``effectivizing'' results concerning invariant measures, inevitably,
leads to questions about efficient generation of Lie groups; in handling these
(technical) issues, we shall make use of some simple arguments in algebraic
 geometry and Diophantine geometry, as well as of
the Lojasiewicz inequality.

\subsection{Organization of the paper.}
The paper is organized as follows:

We start in \S \ref{new-sec-3} with the non effective version of Theorem \ref{thm:main}
and the simplified proof of Ratner's theorem on measure rigidity in the case of semisimple acting groups.

In \S \ref{notation}, we set up notation. Some of this may be of independent interest;
our use of Sobolev norms and relative traces, inspired by a paper of Bernstein and Reznikov,
seems a natural way of handling certain analytic questions on homogeneous spaces.

In \S \ref{proof}, we give a detailed sketch of the proof of the Theorem.  We give precise statements of
the main constituent results that go into the Theorem, and references to their
proofs in the text. 
We do not give proofs, but we indicate the main idea in each case.

In \S \ref{sobnormproofs}, \ref{spectralgap}, \ref{effectivegen}, \ref{Appendix-A} respectively,
we establish or recall some basic properties pertaining
to, respectively, Sobolev norms, the spectral gap, generation of Lie algebras,
and almost invariance of measures. The reader might wish to skip these sections
and refer to them as necessary while reading the rest of the text.

In the sections \S \ref{firststep} -- \S \ref{closing}, we prove the 
central statements required for the Theorem. 
%-- \S \ref{firststep} proves Proposition \ref{lemma3},
% \S \ref{subsec:production} proves Proposition \ref{lem:Moreinvariance},
% \S \ref{closing} proves Proposition \ref{prop:dichotomy},
%\S \ref{sec:quantiso} proves Proposition \ref{lem:Quantiso}.

\S \ref{finalproof} gives the proof of Theorem \ref{thm:main} and its corollaries. 

\S \ref{arithappl} gives a simple application of the Theorem to a number-theory problem. 

The Appendices recall and prove certain basic results that we need.

\subsection{Connection to existing work}\label{existing-work}
\subsubsection{Dynamics of unipotent flows}\label{existing-work:HD}

Given a one parameter subgroup $g_t \in G$,
and two ``nearby'' points $x_1, x_2 \in X$,
satisfying $x_2 = x_1 \exp(v)$, for $v \in \mathrm{Lie}(G)$,
we have $$x_2 g_t = x_1 g_t \exp(\Ad(g_{-t}) v).$$
The behavior of $t \mapsto \Ad(g_{-t} )v$ controls the relative
behavior of the orbits of $x_1$ and $x_2$ under $g_t$.
In general, $t \mapsto \Ad(g_{-t}) v$ may grow exponentially. However,
if $g_t$ happens to be a {\em unipotent} subgroup, then
 $t \mapsto \Ad(g_{-t}) v$ is a polynomial.  This feature is the
{\em polynomial behavior} or {\em polynomial divergence} of unipotent flows.

The polynomial divergence of orbits of unipotent flows was one of the
main motivations for M.S. Raghunathan when he formulated, in the mid-seventies,
his conjecture about ``orbit closure rigidity.''
He hoped that unipotent flows are likely to have ``manageable behavior''
because of the slow divergence of orbits of unipotent one-parameter
subgroups (in contrast to the exponential divergence of orbits
of diagonalizable subgroups).

The polynomial behavior of unipotent orbits is very important in the present
 paper, and we have used many techniques developed over the past four decades.
We make particular note of the following three ideas we use;
this is not intended as either a survey or a history of the field.

1. Nondivergence of unipotent flows. One of the first times
the polynomial behavior was used was in the proof of non-divergence
of unipotent flows by G.M. \cite{M-non-divergence}, where
it was indeed one of the basic ingredients.
 This phenomenon of non-divergence
was quantified and further refined by Dani, Kleinbock, and G.M.
\cite{Dani-unipotent,DM-1,Kleinbock-Margulis-Ann}. We actually make use of
\cite{Kleinbock-Margulis-Ann} to control how much mass of a closed
 $H$-orbit can be close to infinity, see Lemma \ref{measureoutsidecompact}.

2. "Nearby generic points give additional invariance."
A topological incarnation of this principle was utilized in the work of G.M.
 and G.M.-Dani \cite{DM-Opp-1, DM-Opp-2, DM, M-formes, M-Selberg, M-indefinite}, and a measure-theoretic incarnation was
utilized in the work of M. Ratner in \cite{Ratner-Acta, Ratner-Inventiones, Ratner-Duke, Ratner-measure-rigidity}; the former uses Chevalley's theorem in order to control relative behavior of orbits of unipotent flows, and the arguments of the latter papers utilize the``R-principle,''
generalizing the ``H-principle'' used earlier by Ratner \cite{ Ratner-rigidity,Ratner-factors,Ratner-joinings,
time-horocycle} 
and by Witte \cite{WitteBAMS, WitteInventiones}. 

%In the series of papers \cite{Ratner-Acta, Ratner-Inventiones, Ratner-Duke, Ratner-measure-rigidity}
%M. Ratner introduced and applied the R-principle which generalizes the H-principle that was
%used earlier by Ratner \cite{ Ratner-rigidity,Ratner-factors,Ratner-joinings,time-hororcycle} 
%and by Witte \cite{WitteBAMS, WitteInventiones}. 

To be a little more precise,
the arguments of \cite{DM-Opp-1, DM-Opp-2, DM, M-formes, M-Selberg, M-indefinite}
 show that, if $U$ is a unipotent subgroup of $G$,
and $Y$ is a minimal closed $U$-invariant subset of $\Gamma \backslash G$,
 then either:
\begin{enumerate}
\item  $Y$ is an orbit of $U$, or
\item  $Y$ is
invariant under a bigger connected subgroup contained in the normalizer of
$U$ (additional invariance). \end{enumerate}
Arguments contained in \cite{Ratner-Acta, Ratner-Inventiones, Ratner-Duke, Ratner-measure-rigidity}
show that, if $\mu$ is an ergodic $U$-invariant measure,  then either:
\begin{enumerate}
\item $\mu$ is the Haar measure on a
closed $U$-orbit, or
\item  $\mu$ is invariant under a bigger connected subgroup
contained in the normalizer of $U$ (additional invariance).
\end{enumerate}

The proof of these results may be understood as topological 
and measure-theoretic versions of the principle enunciated 
above;  the notion of {\em minimal set} is the topological
 analog of the notion of {\em ergodic measure}.  Indeed,
 the proof of the topological statement is based on studying
 the orbits of a sequence of points that converge to a
 minimal closed invariant set;
the proof of the measure-theoretic statement is based on studying the orbits
of nearby points that are generic in the sense of Birkhoff's ergodic theorem.

Let us draw attention to one important difference
between the topological and the measure-theoretic setting: an
arbitrary $U$-invariant measure may be decomposed
into ergodic components; the analog of this statement
for a $U$-invariant closed set is not true.
 
  Our use of an effective ergodic theorem, together with  polynomial
divergence, is motivated by this principle; many of our arguments are closely related to
(ineffective) arguments from the papers mentioned
above. These ideas enter, in particular, into steps (1) and (2) of the outline of the proof presented in \S \ref{outline}.

3. The ``linearization" technique. 
By this, we shall mean a collection of
methods which allow one to control the time an orbit of a unipotent subgroup U
spends near closed orbits of certain subgroups $S\subset G$ containing
$U$. One method to do this is by realizing the action of $U$ on $G/S$ as the linear action of $U$ on a
certain subset of a linear space (which also makes use of Chevalley's
theorem). The origins of this method can be traced back to the work of
S.G.Dani and Smillie \cite{Dani-Smillie} in the context of Fuchsian groups.
General results and techniques are developed in \cite{DM-Opp-1, DM-Opp-2, DM, M-formes, M-Selberg, M-indefinite, 
GMTS,
Ratner-Acta, Ratner-Inventiones, Ratner-Duke, Ratner-measure-rigidity}. 

We make use of a version of this technique to control the time spent near
a closed orbit of a semisimple subgroup, see Proposition~\ref{lem:Quantiso};
this enters into step (3) of the outline of \S \ref{outline}.

 \subsubsection{Effective equidistribution.}\label{effeqdist}
There are two general cases where it is known one may establish
effective results about the distribution of $H$-orbits on $\Gamma \backslash G$:
\begin{enumerate}
\item  When $H$ is a {\em horospherical} subgroup, i.e. the Lie algebra of $\mathfrak{h}$
consists of all contracted directions for the adjoint action of a semisimple element $s \in G$. In that case,
one may use the mixing properties of the $s$-action on $\Gamma \backslash G$.
\item When $H$ is ``large'' inside $G$,
 one may sometimes analyze effectively the distribution
of closed $H$-orbits inside $G$ via representation theory or automorphic forms.
``Large'' usually means, {\em at the very least}, that $H$ should act with an open orbit on the flag variety $G/P_0$,
 where $P_0$ is a minimal parabolic of $G$.
 \end{enumerate}

In the case (1), the approach using mixing properties
of the $s$-action on $\Gamma \backslash G$ can be traced back to the thesis
of G.M. \cite{Margulis-thesis}, where it was used in the context
of Anosov flows. (This thesis was written in 1970, but published more than thirty years later). Another approach to
effective equidistribution for long horocycles on $\operatorname{SL}_2(\R)$ is implicit
in the work of Ratner \cite{time-horocycle, Ratner-sl2}, where the effective ergodicity of the horocycle flow is used.
Both of these ideas would suffice to prove effective equidistribution of orbits of horospherical subgroups in any rank,
as would the work of Burger \cite{Burger}.
More detailed analysis of the quantitative equidistribution of the horocycle flow for $\operatorname{SL}_2(\R)$
may be found in \cite{flaminio-forni,Strombergsson}; quantitative mixing rates are discussed in
\cite{M-K-log}, and analysis of equidistribution
of closed horospheres and its relevance to the theory of automorphic forms may be found in \cite{sparse-equi}.

A typical instance of the second type of result is the equidistribution of
Hecke points, \cite{EO2} or \cite{COU}; a ``twisted'' version
is the work of Li, Jiang and Zhang, \cite{LJZ}.  These both
correspond to the special case of the main result when $H$ is diagonally
embedded in $G = H \times H$; in the ``Hecke point'' cases,
one restricts the possibilities for the $\Q$-form underlying the
 closed $H$-orbit, whereas the latter case restricts to $H = \PGL_2$ over a 
totally real number field.  

An illuminating context where the two cases overlap 
 is the question of {\em equidistribution of translates}
for a closed orbit of a symmetric subgroup: see \cite{Duke-Rudnick-Sarnak}
 for an effective
treatment by the second method, and \cite{Benoist-Oh}
 for an effective treatment in a more general setting, following the strategy of \cite{EM}, by a method very closely related to the first. 
 (The ``wavefront lemma" assures that, in the limit, such translates
 acquire local invariance by a horospherical subgroup).

In both cases above, the analysis 
usually makes use of spectral gaps; for instance,
the first method makes use of quantitative mixing, which is substantially equivalent to a spectral gap. Spectral gaps are also used in our argument, but
in an essentially different way; most importantly,
we do not use a spectral gap for the ambient group $G$ acting
on the ambient space, but rather for the acting group $H \subset G$
acting on the invariant measure.  The idea that this could
be used to give a simple proof -- in certain cases -- of 
the fact that ``limits of ergodic measures remain ergodic'' was used in a paper
by J. Ellenberg and A.V. \cite{EV}, based on prior discussions
with M. E. and Elon Lindenstrauss.

In any case, the methods in both cases are fundamentally limited
even if $H$ is a maximal subgroup of $G$. Moreover,
these methods do not
detect closed orbits of
intermediate subgroups $H \subset S \subset G$, and so there appears
to be little hope of generalizing them significantly. Indeed, the result
of this paper appears to be the first that produces quantitative results
when $G$ is semisimple and $H$ is not one of the subgroups
mentioned in (1) or (2) above,
and, in particular, in the case when $H$ is ``far from maximal.''

\subsubsection{A history of this paper.} \label{history}
A topological version of Theorem \ref{thm:main} was proven by G.M., in an unpublished manuscript,  in the case
 when $H = \mathrm{SO}(2,1)$ and $G = \mathrm{SL}_3(\R)$ (and implicitly
in the case when $H$ is a maximal semisimple subgroup of $G$). 

The present paper uses a different technique than the one used in that proof. 
In the case of $H = \mathrm{SO}(2,1)$, $G = \mathrm{SL}_3(\R)$, that proof uses mixing properties of
   the action of a torus within $H$, whereas the proof presented in this paper uses ergodic properties of
   a one-parameter unipotent flow within $H$. A discussion of the relationship of these techniques is presented in \S \ref{semisimple-measures}. 

   The idea of using effective ergodic theorems for the unipotent flow was noted
   independently by M.E. and A.V., and by G.M.,
motivated by various methods in the theory of unipotent flows. 

We note that the general case is very considerably more involved than the 
case of $\mathrm{SO}(2,1) \subset \mathrm{SL}_3(\R)$, owing to the possibility of intermediate subgroups.  In fact, the technically most complicated arguments,
e.g.\ the effective closing lemma, are necessary only if $H$ is not maximal.

\subsection{Remarks on generalizations.} \label{generalization}
\subsubsection{The congruence subgroup assumption.}\index{congruence}

One may replace the assumption that $\Gamma$ is congruence
by the assumption that $\Gamma$ is arithmetic. We indicate here how this is done. 

Suppose that $\Gamma$ is an arithmetic, but
not necessarily congruence, subgroup of $\G(\Q)$
and $x H$ is a closed $H$-orbit on $X = \Gamma
 \backslash G$. Let $\Lambda$ be a congruence subgroup of $\G(\Q)$.
Replacing $\Gamma$ with $\Gamma \cap \Lambda$ and
$x H$ by the orbit $x'H$ of a preimage $x'$ of $x$ in
$\Gamma \cap \Lambda\backslash G$, we may
assume that $\Gamma \subset \Lambda$. Let $\bar{X} = \Lambda \backslash G$,
so we have a natural projection $\pi: X \rightarrow \bar{X}$.

Let $\mu$ resp. $\overline{\mu}$ be the $H$-invariant measures
on $x H$ resp. their projections to $\bar{X}$. Our main theorem
gives an effective result about the distribution of $\overline{\mu}$.
Suppose, for simplicity, we are in the situation where $\overline{\mu}$
is close to the $G$-invariant probability measure $\mu_{\bar{X}}$ on $\bar{X}$.
Here the error of approximation is quantified by Theorem \ref{thm:main}.

Let $f \in C^{\infty}_c(X)$ be so that $\mu_X(f) = 0$. Let $u(t)$ be a one-parameter subgroup
in $H$, and let $f_T := \frac{1}{T}\int_{t=0}^{T} u(t)f$.

For a function $f$ on $X$, we define $\pi_*f$ on $\bar{X}$
via $\pi_*f(\bar{x}) = \sum_{\pi(x) = \bar{x}} f(x)$.

Then
\begin{multline*} %\label{noncong}
|\mu(f)|^2 = |\mu(f_T)|^2 \leq  \mu(|f_T|^2) \leq \overline{\mu}(\pi_* |f_T|^2)
\sim \\ \mu_{\bar{X}} (\pi_* |f_T|^2)  = [\Lambda:\Gamma]\mu_X(|f_T|^2)
\end{multline*}
where the rate of approximation depends on the quality of approximation of $\mu_{\bar{X}}$ by $\overline{\mu}$
and on the Sobolev norm of $\pi_* |f_T|^2$ and so also on $T$.
The right-hand quantity is bounded by quantitative mixing of the $G$-action
on $X$. Optimizing for $T$ shows that $\mu_X(f) = 0$ implies that $|\mu(f)|$
is ``small'', which is the same as $\mu \sim \mu_X$.

 The above argument is entirely quantitative.  The quality of the bound depends on the spectral gap for $G$ acting on $L^2(X)$. 
 
 Observe an important difference with Theorem \ref{thm:main}:
 in the case when $\Gamma$ is assumed only to be arithmetic, the quality of the bound
 depends on the lattice $\Gamma$, and not only on $G,H$. Presumably, such 
 a dependence is not avoidable, as one may see by considering the example
 of the horocycle flow on arithmetic quotients of $\SL_2(\R)$.

\subsubsection{$S$-arithmetic generalization. Translates of a fixed orbit.} \label{sarithmetic}
One may envisage a version of the theorem that
concerns closed orbits of a semisimple $S$-arithmetic group $H$
acting on an $S$-arithmetic quotient $\Gamma \backslash G$.

This extension has more than simple formal significance. We indicate
three applications which do not on their face involve $S$-arithmetic groups.
\begin{enumerate}
\item It is possible to give an independent proof of property $(\tau)$ (first established
by Clozel \cite{LC}) using the $S$-arithmetic extension of our theorem.\footnote{Note that, in the present paper, we use property $(\tau)$ as an input. However, it is possible to avoid it by more elaborate arguments, and indeed to {\em derive} it in general by our methods, at least for groups of absolute rank $\geq 2$.
The groups of absolute rank $1$ require an alternative treatment; this is also true of Clozel's approach.}

\item Notations as in our main theorem, let $x_0 H$ be a closed $H$-orbit.
One may ask about the distribution of {\em varying translates} of $H$,
i.e. $x_0 H g$ when $g \rightarrow \infty$. If $H$ is a symmetric subgroup,
effective distribution results may be given using the wavefront lemma, cf. \cite{EM}.
We anticipate that the $S$-arithmetic extension of our result
will allow us to treat this question in the more general setting of
\S \ref{technical}. 

\item In an arithmetic direction,
 \cite{EV} uses
the case when $H$, $G$ are $p$-adic orthogonal groups to prove local-global principles for representations of quadratic forms. To establish an effective version of that result would imply new bounds for the Fourier coefficients of Siegel modular forms.  Results in this direction would require removal of the centralizer assumption from \S \ref{technical}.
\end{enumerate}
As is usual in such matters, the existence of small additive subgroups of $\Q_p$ will cause further complications.
\subsubsection{The centralizer assumption}
There exist many natural settings where the centralizer assumption
of \S \ref{technical}, assumed in Theorem \ref{thm:main}, is too restrictive. Indeed, it does not seem to be truly essential to our method; the key part 
of our method is the existence and exploitation of spectral gaps. 

 However, many technical complications seem to arise when it is removed. We hope to discuss this elsewhere. 

\subsection{Acknowledgements.}
This work was initiated at the Institute of Advanced
Study during the academic year 2005-2006; in fact, all authors were
visiting there for various parts of that year. We would like to
thank the IAS for providing excellent working conditions.\footnote{The first and the third named authors would like to make
particular note of the excellence of the cookies.}

This research has been supported by the Clay Mathematics Institute
(M.E. by a Clay Research Scholarship, A.V. by a fellowship);
 and by the NSF (DMS grants 0622397 (M.E.), 02045606 (A.V.), 
(0244406) (D.M.), and an FRG collaborative grant 0554373 (M.E. \& A.V.))

We would like to express our appreciation of conversations with
E.\ Lapid (concerning the proof of \eqref{temp})
and H.\ Oh and Y.\ Shalom (concerning \S \ref{arithappl} and \S \ref{GWsec}). 

We would also like to express our gratitude to Elon Lindenstrauss,
who both contributed greatly to early discussions about this paper,
and encouraged its completion.
%He declined to be a co-author, but, nonetheless, our debt to him
%is very considerable. 
%Moreover, the TeX code for this paper uses some useful ideas of Elon to help deal
%with the mass of constants. 
It is a great pleasure to thank him for his generosity with ideas and time.

%%%%%%%%%%%%%%%%%%%%%%%%%%%%%%%%%%%%%%%%%%%%%%%%%%%%%%%%%%%%%%%%%%%%%%%%%%%%%%%%%%%%%%%

\section{A proof of measure classification for semisimple groups.} \label{new-sec-3} \label{semisimple-measures}
\subsection{Introduction.}In the present section, we give a short proof
of the non-effective version of Theorem \ref{thm:main}, 
the result \eqref{MS} of Mozes and Shah. 
Notation as in the theorem, this proof has three (independent)
constituents:
\begin{enumerate}
\item Proposition \ref{MC} classifies
$H$-ergodic measures on $\Gamma \backslash G$;
\item
Proposition \ref{PT} shows that, under a certain spectral
gap assumption, any limit of $H$-ergodic probability measures on $\Gamma \backslash G$
is itself ergodic.
\item  It is established in Proposition \ref{automorphic}
that this spectral gap assumption is valid in the setting of \eqref{MS}. 
\end{enumerate}

Therefore, taken in combination, these results establish \eqref{MS}. 
\footnote{To be precise, in the case where $\Gamma\backslash G$ is noncompact,
there may be ``escape of mass,'' i.e., it is not a priori obvious that
the limit of such a sequence is a probability measure.  However, in the case considered in Theorem \ref{thm:main} -- $\mathfrak{h}$ 
has trivial centralizer -- this does not occur: Lemma \ref{measureoutsidecompact}. }

While the proofs of Proposition \ref{MC} and Proposition \ref{PT}
are elementary, Proposition \ref{automorphic} is a deep result
from the theory of automorphic forms. 
The key virtue of this proof, however,
is amenable to effectivization.  Indeed, the contents
of this section
may be viewed as the non-effective counterpart of the proof of Theorem \ref{thm:main}.
\S \ref{proof} gives an outline of the effective proof of Theorem~\ref{thm:main}, together
with comparisons to the non-effective proof presented here.

\begin{prop}  \label{MC}
Let $G$ be a semisimple real Lie group, let $\Gamma$ be a discrete subgroup in $G$, and let $H$ be a
connected semisimple subgroup of $G$ without compact factors.  Let $\mu$ be an
$H$-invariant and ergodic probability measure on $X=\Gamma\backslash G$. Then $\mu$ is the
$S$-invariant probability measure on a closed $S$-orbit of an intermediate subgroup $H \subseteq S \subseteq G$.
\end{prop}

This is due to M. Ratner  \cite{Ratner-measure-rigidity}. In the setting considered here, where $H$ is semisimple, the proof simplifies significantly.
This has been noted already in \cite{Einsiedler} due to M.E. 
We shall present a simplified proof along
these lines, differing somewhat from \cite{Einsiedler}. See also \cite[p244]{Ratner-Duke}
where similar arguments are presented. 

Just as the outline presented in \S \ref{outline},  the proof of Proposition \ref{MC}
has three distinct steps:
\begin{enumerate} \item {\em Ergodic theorem:}
\S \ref{subsec:A}.
\item {\em Nearby generic points give additional invariance:} 
\S \ref{subsec:B}. 
\item {\em Dichotomy:} Either one can find nearby generic points,
or $\mu$ is supported on a closed orbit. \S \ref{subsec:C}. 
\end{enumerate}

\subsection{Setup for the proof.} \label{subsec:Setup}

We now indicate the proof of Proposition \ref{MC}. 
Let $G$, $H$, $\Gamma$, $X=\Gamma\backslash G$, and $\mu$ be as in Proposition \ref{MC}.

Let $S=\{g \in G : g$ preserves $\mu\}$. The subgroup $S$ is closed and
contains $H$. $H$ being semisimple, the restriction to $H$ of the adjoint
representation of $G$ is completely reducible. Hence there exists an
$\mathrm{Ad}(H)$-invariant complement $\mathfrak r$ to the Lie algebra $\mathfrak s$ of $S$,
within $\mathfrak{g}$. 
 It is the precisely the existence of this $\mathrm{Ad}(H)$-invariant complement
which simplifies the proof of the measure classification result.

\subsection{Ergodic theorem} \label{subsec:A}
Let $U=\{u(t):t\in\R\}$ be a one-parameter unipotent subgroup of $H$ which projects nontrivially on each
simple factor of $H$. Then, according to a theorem of Moore  (Mautner phenomenon) the measure
$\mu$ is $U$-ergodic.

According to Birkhoff's individual ergodic theorem, for $\mu$-almost $x\in X$
and any continuous compactly supported function $f$ on $X$,
\begin{equation}\label{compact-Birkhoff}
  \frac{1}{T} \int_0^T f(xu_{-t})\operatorname{d}\! t\rightarrow\int f \operatorname{d}\!\mu.
\end{equation}
A point $x\in X$ satisfying \eqref{compact-Birkhoff} for all $f\in C_c(X)$ is called {\em generic}.

We say that $E\subset X$ is a set of uniform convergence if for any
compactly supported function $f$ the above convergence is uniform with respect to $x\in E$.
In view of Egoroff's theorem, there exist sets of uniform convergence of measure arbitrarily close to $1$.  
Fix such a set $E$ of uniform  convergence of $\mu$-measure $>9/10$.

\subsection{Nearby generic points give additional invariance.} \label{subsec:B}

First let us make the following remark, which is a quantification of ``polynomial behavior''
discussed in \S \ref{existing-work:HD}:
\begin{enumerate}
 \item[(i)] For any element $g$ of $G$, the matrix coefficients of $\mathrm{Ad}(u_t gu_{-t})$
are polynomials in $t$ of degree not greater than $\dim(G)$.
\end{enumerate}

Let $B$ be an open bounded subset of $G$ containing the identity. For $g\in B$, we set
\begin{eqnarray*}
 T_g &=& \sup\bigl\{ T : u_t g u_{-t} \in B \mbox{
whenever } t \in (0,T)\bigr\}, \\
 g^* &=& u_{T_g} g u_{-T_g}.
\end{eqnarray*}
It follows from (i) that $T_g$ is finite if $g$ does not belong
to the centralizer $C_G(U)$ of $U$ in $G$. Moreover, in that case
$g^*$ is well defined and belongs to the boundary of $B$.

 For $g\in B\setminus C(U)$, we define a map
$q_g:[0,1]\to G$ by
\[
            q_g(s/T_g)=u_s gu_{-s}, s \in [0,T_g].
\]
Since $B$ is bounded, the family $\{q_g\}$ is uniformly bounded. Therefore, the adjoint
actions of these elements form an equicontinuous family of polynomial maps on the Lie algebra.
This implies the following statement, for a metric $d$ on $X$ that is obtained
from a left-invariant metric on $G$:
\begin{enumerate}
 \item[(ii)] For every $\epsilon>0$ one can find $\delta>0$ such that if
 $x\in X$, $g\in B\setminus C_G(U)$, and $t\in [(1-\delta)T_g,T_g]$,
then the distance
\begin{equation}\label{slow-divergence}
        d\bigl(xg u_{-t} ,xu_{-t} g^* \bigr) \leq d( u_t g u_{-t}, g^*) < \epsilon.
\end{equation}
\end{enumerate}

\begin{lem} \label{easy=proof}
(a) Let $g$ be an element of $C_G(U)$. Suppose that we can find $x,y\in E$ such that $y=xg$.
Then the measure $\mu$ is invariant under $g$.

(b) Suppose that we can find a sequence $g_n\in G$ and a sequence $x_n \in E$ such that
\begin{enumerate}
    \item $g_n \rightarrow 1$;
    \item $g_n \notin C_G(U)$ for all $n$,
    \item $x_ng_n\in E$ for all $n$.
\end{enumerate}
Then $\mu$ is invariant under any limit point of the $g_n^*$.
\end{lem}

\begin{proof} The first part follows from the
genericity of the points.

Consider the second assertion.  Let $g^*$ be a limit point of $g_n^*$, which exists by compactness of the closure of $B$. Let $f \in C_c(X)$.
Observe that $T_{g_n} \rightarrow \infty$. 

Given $\varepsilon > 0$ and $\delta>0$  we may apply the definition of $E$ (for
$T=T_{g_n}$ and $T=\delta T_{g_n}$), that $x_n g_n \in E$, and that $x_n \in E$
to obtain for all sufficiently large $n$ that
\begin{equation*}\begin{aligned}
\left|\frac{1}{\delta T_{g_n}} \int_{(1-\delta)T_{g_n}}^{T_{g_n}} f(x_ng_n u_{-t})\operatorname{d}\! t
-  \int f\operatorname{d}\!\mu \right| < \varepsilon', \\
\left|\frac{1}{\delta T_{g_n}} \int_{(1-\delta)T_{g_n}}^{T_{g_n}} f(x_n u_{-t} g^*)\operatorname{d}\! t
-  \int f(x g^*) \operatorname{d}\!\mu(x) \right| < \varepsilon'
\end{aligned} 
\end{equation*}
Choosing $\delta>0$ as in (ii), depending on $\varepsilon>0$ and the uniform
continuity of $f$, we conclude in the limit 
 that the $\mu$-integral of $x \mapsto f(x) $ and $x \mapsto f(x g^*)$ coincide. \end{proof}

\subsection{Dichotomy} \label{subsec:C}

\begin{lem} \label{step=two}
 Assume that $\mu(xS)=0$ for all $x\in X$.
 Then there exists a sequence $x_n\in E$ and a sequence $g_n\in\exp (\mathfrak r\setminus\{0\})$
 converging to $1$ such that $x_n g_n\in E$.
\end{lem}

\begin{proof}[Sketch of the proof]
 Let $O$ be a bounded open neighborhood of the identity in $S$, and
let $E_1$ be the set of all points $x$ in $E$ such that the relative measure in
$O$ of the set $\{s \in O: xs \in E\}$ is greater than $8/10$. It follows
from the Fubini theorem that $\mu(E_1)>1/2$.

For each $n > 1$
 take two points
$y_n,z_n \in E_1$ that satisfy $d(y_n, z_n) < 1/n$, but do not satisfy $y_n= z_n s$ for any $s \in S$ near the identity. A density computation shows that there exists $s_n, s_n' \in O$
so that $y_n s_n \in E, z_n s_n' \in E$, and $z_n s_n' = y_n s_n \exp(r_n)$
for some $r_n \in \mathfrak{r}$. 
This is quite standard, see  e.g.\ \cite[Lemma 4.4, 4.5]{Einsiedler} for details. 

%For each $s \in O$, there 
%The set
%\[
%\bigl\{s \in O: z_ns\in E\mbox{ and } z_n s\exp(\mathfrak{r})\cap y_nO\cap E\neq\emptyset\bigr\}
%\]
%has relative measure in $O$ greater than $7/10$, as follows from the
%fact that the map $(X \in \mathfrak{r}, s \in S) \mapsto \exp(X) s$
%is a diffeomorphism onto $G$ near the identity. Choosing such an $s$, 
Take $x_n = y_n s_n$, and take $g_n =\exp(r_n)$.  The resulting sequences have the desired properties. 
\end{proof}

\subsection{Conclusion of the proof} \label{subsec:D}
Suppose that $\mu(x S) = 0$ for all $x$. By using  Lemma~\ref{easy=proof} and~Lemma \ref{step=two} we show that $\mu$
is $g$-invariant, for some $g \in \overline{B} \cap \exp(\mathfrak r\setminus\{0\})$.
By choosing $B$ sufficiently small, we may assume that $ \overline{B} \cap \exp(\mathfrak r\setminus\{0\})$ is disjoint from $S$.
Contradiction.

Thus $\mu(xS)>0$ for some $x\in X$. The stabilizer
$S_x=\{s\in S: xs=x\}$ of $x$ is a lattice in $S$. Consequently, $S$ is unimodular and $x S$ is closed.
We are done, for $\mu$ is ergodic.

\qed

\subsection{A remark about the two approaches in \S \ref{history}}
We may phrase the proof of Proposition \ref{MC} in qualitative forms as follows:
\begin{multline}\label{approx} \mbox{The measure along the trajectory
$\{x u(t): 0 \leq t \leq T\}$}\\
\mbox{approximates $\mu$ for ``most'' $x$.}\end{multline} One then finds two
such points $x_1$, $x_2$ and studies the relation between their trajectories in order
to obtain additional invariance. 

\eqref{approx} is a special case
of a more general fact, not specific to unipotent trajectories:
if $\nu$ is a measure on $H$ with ``spread-out support,'' and we let $\nu_x$ be the push-forward of $\nu$ by the map $h\to xh$,
then $\nu_x$ will approximate $\mu$ for ``most'' $x$; the previous remark  is the special
case when $\nu$ is localized along a long trajectory of $u(t)$.

In \S \ref{history}, we remarked on another method of treating the case 
of $H \subset G$ maximal, using mixing properties of semisimple elements.
We are now in a position to briefly describe it.

In that approach, one considers instead the measures $\nu$ that are obtained by translating a fixed,
compactly supported density on $H$ by a large semisimple element.
In particular, these measures $\nu$ resemble large pieces of a coset of a horospherical subgroup of $H$.
As in the proof with unipotent flows given here, one studies the relation between $\nu_{x_1}$, $\nu_{x_2}$ for nearby $x_1$ and $x_2$
to conclude additional invariance.

\subsection{Ergodicity of limit measures in presence of a spectral gap} \label{GWsec}

 \begin{prop} \label{PT}
Let a $\sigma$-compact metric group $H$ act continuously on a
$\sigma$-compact metric space $X$. Let $\mu_n$ be a sequence of
$H$-invariant and ergodic measures so that the $H$-action on
$$
\bigl\{ f \in L^2(\mu_n): \int f \operatorname{d}\!\mu_n= 0\bigr\}
$$
possesses a spectral gap which is uniform in $n$. Then any weak$^*$-limit
of the $\mu_n$ is ergodic under $H$.
\end{prop}

This result is due to E.~Glasner and B.~Weiss, in the case when $H$ has property (T).  More along the lines of the current paper: it is used to give a ``cheap'' proof of a special case
of a Mozes-Shah type result, in a $p$-adic setting, in \cite{EV}. 

Note that we do not say whether the weak$^*$ limit is still a probability measure.

 For our purposes, it will be important that, in a suitable
sense, the proof of Proposition \ref{PT} is effective. Indeed, fixing a compact generating
set $K \subset H$, there exists $\delta > 0$ so that:
$$\sup_{k \in K} \mu_{\infty}(k Z - Z) \geq \delta \mu_{\infty}(Z)$$ for any Borel set $Z \subset X$ with $\mu(Z) \leq 1/2$, where $\mu_\infty$ is any weak$^*$ limit of the sequence $\mu_n$. 

In other terms, there are no {\em almost} invariant subsets. 

\proof
Let $\nu$ be a compactly supported probability measure on $H$, so that for every $n$:
\begin{equation} \label{SG}
 \bigl\| \nu \star f - \int f \operatorname{d}\!\mu_n \bigr\|_{L^2(\mu_n)} \leq \frac{1}{2} \|f\|_{L^2(\mu_n)}
\end{equation}
for all $f\in L^2(\mu_n)$ (or equivalently for all $f\in C_c(X)$).
Here $\nu\star f$ denotes convolution of $f$ under $\nu$ w.r.t.\ the
action. The measure $\nu$ exists by virtue of the assumption on the
uniform spectral gap. Passing to the limit, one easily verifies that
\eqref{SG} remains valid with $\mu_n$ replaced by any weak$^*$-limit
thereof. This implies that such a weak$^*$-limit must be
$H$-ergodic.\qed

%%%%%%%%%%%%%%%%%%%%%%%%%%%%%%%%%%%%%%%%%%%%%%%%%%%%%%%%%%%%%%%%%%%%%%%%%%%%%%%%%%%%%%%%%%%%%%%%%%%%
\section{Notation and first facts.}
  \label{notation}

 Fix, first of all,
a semisimple $\Q$-group $\G$ and a Euclidean norm $\|\cdot\|$ on the Lie algebra $\mathfrak{g}$
 of $G=\G(\R)$  in such a way that $\|[u,v]\| \leq \|u\| \|v\|$ for all $u,v\in\mathfrak{g}$. We fix also an orthonormal basis for $\mathfrak{g}$ w.r.t.\ $\|\cdot\|$.

  We fix a semisimple subgroup $H \subset G$ that
satisfies the conditions indicated in \S \ref{technical}.

Fix also an embedding $\rho: \G \rightarrow \GL_N$. We shall assume that
the adjoint representation of $\G$ occurs as an irreducible subrepresentation
of that defined by $\rho$.

We take a {\em congruence lattice} $\Gamma \leqslant G$.\index{congruence} That
$\Gamma$ is congruence means that $\rho(\Gamma)$ contains a congruence subgroup
of $\mathrm{GL}_N(\Z)\cap\rho(G)$.

Set $X = \Gamma \backslash G$. 

The Lie algebra $\g$ possesses a natural $\Q$-structure. We choose a rational
$\Gamma$-stable lattice
$\g_{\Z} \subset \g$ satisfying $[\g_\Z,\g_\Z]\subset\g_\Z$. 
This is always possible, assuming only that $\Gamma$ is arithmetic.

We fix throughout this section, and, indeed, throughout the entire paper,
a homomorphism $\phi: \SL_2(\R) \rightarrow H$ that
projects nontrivially to each simple factor of $H^+$.
This determines a unipotent one-parameter subgroup:
\begin{equation}\label{udef} u(t):= \phi\left( \begin{array}{cc} 1 & t \\ 0 & 1
\end{array}\right).\end{equation}

 The adjoint representation of this $\SL_2(\R)$,
i.e.\ $\mathrm{Ad} \circ \phi$, decomposes $\mathfrak{g}$ into a
direct sum of irreducible representations. Let $\mathfrak{g}_0$ be
the sum of all the highest weight spaces in all these irreducible
representations, with respect to the diagonal torus of $\SL_2(\R)$
(equivalently: the $\{u(t)\}$-fixed subspace) and let
$\mathfrak{g}_1$ be the sum of all remaining weight spaces. Thus
\begin{equation}\label{gdecomp}\mathfrak{g} = \mathfrak{g}_0 \oplus \mathfrak{g}_1.\end{equation}
For $r \in \mathfrak{g}$, we shall write $r = r_0 + r_1$ according to \eqref{gdecomp}.

\subsection{Concerning constants and their implicit dependencies}
The notation $\xappa,\consta\label{k-sample2}, \dots$
\index{xappa1@$\xappa$, $\ref{k-sample2}$, sample constants}will always denote positive constants that depend
{\em only on the isomorphism class of $(H,G)$.}  
Because, in the circumstance of the main Theorem, there exist only finitely many possibilities
for this isomorphism class if one fixes $\dim(G)$ (cf. Lemma \ref{finite-embeddings}), it is equivalent to say that these constants depend only on $\dim(G)$.
Moreover, these constants
are all indexed, and they
come with hyperlinks to the point where they are defined. We hope the latter feature will be useful for readers using suitable viewing software.

 An important note
is that, because there are only finitely many intermediate subgroups $S$ between $H$ and $G$, any constant that depends only
on $S$ or just on its isomorphism class also can be chosen so that it depends only on $H$, and $G$.
We shall use this observation several times without explicit mention.

The notation $\yota, \constb\label{i-sample2}, \ldots$ will denote positive constants
that depend only\footnote{In particular, constants of the form $\iota_*$ are allowed to depend on the $\Q$-structure on $\G$, by contrast with $\kappa_*$.} on $\G, H, \rho$, the norm $\|\cdot\|$, and on the given orthogonal basis
of $\mathfrak g$.\index{yota1@$\yota$, $\ref{i-sample2}$, sample constants}

The notation $\constc, \constc, \ldots$ will denote positive constants
that may depend on $\G, H, \rho, \|\cdot\|$, the lattice $\Gamma$, and on the lattice $\mathfrak g_\Z\subset\mathfrak g$
discussed in \S \ref{ss-noncompact}.
The numbering of these constants resets by subsection.

If either of these constants depends additionally on other
parameters we make this explicit by writing them in parenthesis,
e.g.\ $\constb(d)\label{bbb-d}$\index{yota11@$\ref{bbb-d}(d)$, sample depending also on $d$} is a constant
depending on $\G, H, \rho$, the
norm $\|\cdot\|$, the orthogonal basis of $\mathfrak g$, and the
parameter $d$.

As a rule of thumb we shall strive to ensure that the {\em exponents} in our results
depend only on $G,H$; on the other hand, we shall not strive for such minimal dependency
in other constants. For instance, in the statement of the Theorem, the exponent $\delta$
depends only $G, H$; whereas the constant $V_0$ is permitted to depend on $\Gamma$. 

\subsection{The $\ll$-notation and the $\star$-notation.}
So that the notation of this paper does not become overwhelming, we shall introduce certain notational conventions.

We shall use the expression ``$A \ll B$,'' for two positive quantities $A,B$, to mean that $A \leq \constc(d) B$,
notations as above. (It will happen in a large fraction of this paper that
a parameter $d$ will be present, measuring the index of a suitable Sobolev norm. Therefore,
we make the convention that implicit constants are allowed to depend on the symbol ``$d$'',
or whatever its value in the present context is. If no symbol $d$ is present, $A \ll B$
simply means $A \leq \constc B$. In \S \ref{sobolevnormsubsec} we give an example to clarify this notation.)

We shall use $A\asymp B$ to mean that both $A \ll B$ and $B \ll A$.
\index{$\ll$-notation}\index{$\asymp$-notation}

Suppose $A$ is a quantity taking values in $(0,\infty)$.  We shall
write $A_{\uparrow}$ for any quantity $f(A)$ that is defined for
sufficiently large $A$ in a fashion that depends only on $ G,
H$, and so that $ f(A)\rightarrow \infty$ as $A \rightarrow \infty$.
We write $A_{\downarrow}$ for any quantity $g(A)$ defined for
sufficiently small $A$ in a fashion that depends only on $ G,
H$ and so that $g(A) \rightarrow 0$ as $A \rightarrow 0$.

For instance, the function $ A^{\dim (H)} - \dim(G) A^{\dim(H)-1} $ could be abbreviated
as $A_{\uparrow}$.

We write $B = A^\star$ if $B = \constc A^{\consta\label{star-not}}$.
We write $B= A^{-\star}$ if $B = \constc A^{-\ref{star-not}}$.
In a similar fashion, we define $B \leq A^{\star}$, $B \geq A^{\star}$, etc.
\index{xappa11@$\ref{star-not}$, sample for star notation}\index{$A^\star$-notation}

For instance, the function $\vol(\Gamma \backslash G) A^{\dim G}$
could be abbreviated as $A^\star$.

 \subsection{Metrics, measures, Lie algebra}

The Euclidean norm on $\mathfrak{g}$ defines a left-invariant Riemannian metric on $G$, which
descends to a metric on $X=\Gamma\backslash G$.

For $g_1, g_2 \in
G$, we write $g_1 \stackrel{\epsilon}{\sim} g_2$ if the distance
between
 $g_1$ and $g_2$ is $\leq  \epsilon$.
Similarly, we use the same notation $x_1 \stackrel{\epsilon}{\sim} x_2$ for $x_1, x_2 \in X$.
We say a finite subset of a metric space is $\delta$-separated if the points all are at mutual distance
$\geq \delta$.

 The Riemannian metric on $G$ also gives a Haar measure on any subgroup of $G$, in particular,
all the intermediate subgroups between $H$ and $G$. We shall
denote these measures by $\operatorname{d}\!\vol$.   If $Q$ is a
subgroup of $G$, we may use this Haar measure to speak of the {\em
volume} $\vol(x_0 Q)$ \index{volume of orbits}of any closed $Q$-orbit on $X$.
By contrast, we shall use either the letters $\mu$ or $\nu$ to denote the $Q$-invariant probability measure on $x_0 Q$.

We set $\|g\| := \max_{ij} \bigl(|\rho(g)|_{ij}, |\rho(g^{-1})|_{ij} \bigr)$.

We note that, for some constant $\constb \label{distortconstant}$ \index{yota111@$\ref{distortconstant}$, operatornorm
for adjoint action}
\begin{equation} \label{distortion}
\|g^{-1}\| = \|g\|,\ \|g_1 g_2\|\leq N \|g_1\| \|g_2\|,\
\| \mathrm{Ad}(g) \|_{\operatorname{op}} \leq \ref{distortconstant}\|g\| 
\end{equation}
 Here $\| \cdot \|_{\operatorname{op}}$ denotes the
operator norm w.r.t.\ the chosen Euclidean norm on $\mathfrak{g}$ and the last inequality
follows since we assume that the adjoint representation occurs in $\rho$. A
consequence of \eqref{distortion} is
\begin{equation} \label{distortion2}  d(xg, yg) \leq \ref{distortconstant}\|g\| d(x,y)  \end{equation}
either for $x,y \in X$ or $x,y \in G$.

\subsection{Intermediate subgroups.}\label{sec: intermediate}
For each intermediate subgroup $ H \subset S \subset G$, we let
$S^0$ be the connected component of the identity in $S$, and let
$\tilde{S}$ be the normalizer of the Lie algebra $\mathfrak{s}$;
thus $S^{0} \subset S \subset \tilde{S}$; moreover, the index
$[\tilde{S}: S^{0}]$ is finite by virtue of the assumption that $\mathfrak{h}$ has
trivial centralizer.    Also, $\tilde{S}$ is ``a real algebraic group'',
i.e.\ consists of the real points of an algebraic subgroup of
$\mathbf{G}$ defined over $\mathbb{R}$. The following will be proved in \S \ref{sec-ass-2}.

\begin{lem}\label{assumption-2}
 Suppose $H\subset G$ is semisimple and that $\h$ has trivial centralizer in
 $\g$. Then $H$ is not contained in any proper parabolic subgroup of
 $G$. Moreover, there are only finitely many intermediate subgroups
 $H\subset S\subset G$. Each such $S$ is semisimple without compact factors. 
\end{lem}

Although we shall not explicitly introduce notation for it, it is convenient
to regard as fixed a choice of maximal compact subgroup of $S$, for every
intermediate subgroup $H \subset S \subset G$. 

\subsection{Balls in intermediate subgroups} \label{intermediateballs}
For each connected intermediate subgroup $S \supset H$,
we shall define a certain increasing family of balls $\Sball(T)$
(sometimes denoted $\Sball^T$). 
These balls have the following properties:

\begin{enumerate}
\item $B_S(T) \subset B_S(T')$ if $T \leq T'$:
\item $B_S(T) \subset \{s \in S: \|s\| \leq T\}$. 
\item \begin{enumerate}\item  If $\Omega \subset S$ is compact, there
exists $c = \constc(\Omega)$ so that
$\Omega.  B_S(T) . \Omega \subset B_S(c T)$;
\item As $\Omega \downarrow 1$, we may take $c \rightarrow 1$. 
\end{enumerate}
\item The volume $B_S(T)$ grows as a positive power of $T$ (up to logarithms),
see \eqref{volest} for the precise statement. 
\end{enumerate}

The construction is slightly elaborate; taking
simply $B_S(T) =\{ s \in S: \|s\| \leq T\}$ is fine when $S$ is almost simple,
but gives rise to ``hyperbola-like'' balls in the semisimple case,
so a slight modification is needed.  The reader need not pay to much attention to the precise details of the construction. 

Fix for each such $S$ an isogeny
\begin{equation}\label{isogeny}\isogeny: \prod_{i=1}^{I} S_i \rightarrow S\end{equation}
from a product of connected, almost-simple groups onto $S$. 
We put, for $S$ connected,
$$B_S(T) := \isogeny \bigl\{ (s_1, \dots, s_I):  \|\isogeny(s_i)\| \leq N^{-1} T^{1/I}
\mbox{ for }i=1,\ldots,I\bigr\}.$$

It would be possible to define corresponding sets in the disconnected case,
but we do not need this. 
%Now, fix a set of representatives $\sigma_1=1, \dots, \sigma_r$
%for the connected components of $\tilde{S}$. Put
%$$B_{\tilde{S}}(T) = B_S(T)\cup\bigcup_{2 \leq i \leq r} \sigma_i B_S(
%N^{-1} \|\sigm%a_i\|^{-1}  T).$$
%If $S \subset S' \subset \tilde{S}$, put $B_{S'}(T) = B_{\tilde{S}}(T) \cap
%S'$. 

\subsection{Noncompactness} \label{ss-noncompact}

For $x \in X$ we set:
\begin{eqnarray*}
 \height(x) &:= &\sup \bigl\{\| \mathrm{Ad}(g^{-1})v \|^{-1} :  \ \Gamma g = x, v \in \g_{\Z}\setminus\{0\} \bigr\} \\
 \Siegel(R) &:= &\{x \in X: \height(x) \leq R\}.
\end{eqnarray*}
In words $\height(x)^{-1}$ is the size of the smallest nonzero vector
in the lattice $\mathrm{Ad}(g^{-1})\g_\Z\subset \g$ corresponding to
$x=\Gamma g$. Note that we have $\height(x)\gg 1$ for all $x\in X$.\index{$\Siegel(R)$, compact subsets of $X$}

As follows from reduction theory the set $\Siegel(R)$ is a compact subset of $X$.
Indeed,
every $x \in \Siegel(R)$ may be expressed as $x = \Gamma g$, where $\mathrm{Ad}(g)$
has operator norm $\|\mathrm{Ad}(g)\|_{\operatorname{op}}\ll R$ w.r.t.\ the norm on $\g$.

We note that
\begin{equation} \label{heightdistort}
\height(x g) \leq  \ref{distortconstant} \|g\| \height(x).
\end{equation}

Moreover, there are constants $\constb \label{hc1}$ and $\consta \label{hc2}$ so that
\index{xappa1112@$\ref{hc2}$, injectivity radius}\index{yota12@$\ref{hc1}$, injectivity radius}
\begin{equation} \label{injfact} \mbox{$g \mapsto x g$ from $G$ to $X$
is injective for $d(g,1) \leq \ref{hc1}
 \height(x)^{-\ref{hc2}}$.}
\end{equation}
That statement follows by considering a basis for the lattice
$\mathrm{Ad}(g^{-1})\g_\Z$ corresponding to $x=\Gamma g$. 
We shall refer to $\ref{hc1} \height(x)^{-\ref{hc2}}$ as the {\em injectivity radius} at $x$. 

We shall require the following lemma, which relies on the mentioned linearization technique \cite{Kleinbock-Margulis-Ann} and
on our technical assumption in \S \ref{technical} that the
centralizer of $\mathfrak h$ is trivial (in the form of Lemma \ref{assumption-2}).
The proof is given in Appendix \ref{manfredsec}.

\begin{lem} \label{measureoutsidecompact}
\begin{enumerate}
\item There exists $R_1 \geq 1$ so that
  any $H$-orbit on $X$ intersects $\Siegel(R_1)$;
\item There are constants $\constc\label{m1constant}$ and $\consta \label{m2constant}$ so that, for any $H$-invariant measure $\mu$,
$$\mu(\{x \notin \Siegel(R)\}) \leq \ref{m1constant} R^{-\ref{m2constant}}$$
\index{xappa1113@$\ref{m2constant}$, decay of $\mu(X\setminus\Siegel(R))$}
\end{enumerate}
\end{lem}

In particular, we may choose $R_0$ so large that, for any
$H$-invariant measure $\mu$, we have $\mu(X\setminus\Siegel(R_0/2)) \leq {10}^{-11}$. The set $\Siegel(R_0)$ will occur
throughout our paper as a convenient choice of compact. We shall
refer to it as $\Xcompact$. \index{$\Xcompact$, a large compact set}

\subsection{Sobolev norms}\label{sobolevnormsubsec}
In dealing with analytic questions on $X = \Gamma \backslash G$, 
we shall make systematic use of a certain family of Sobolev norms. 
While there is overhead in developing their basic properties, 
they offer a clean formalism for analysis on homogeneous spaces. 

Recall that we have chosen an orthonormal basis for $\mathfrak{g}$.
This defines for every $d\geq 0$ a system of $L^{2}$-Sobolev norms
$\Sob_d$ on $C^{\infty}_c(X)$:
\begin{equation} \label{sobnormdef} \Sob_d(f)^2 :=
\sum_{\mathcal{D}}  \|  \height(x)^{d} \mathcal{ D}
f\|_{L^{2}}^2,\end{equation} where the sum is taken over
$\mathcal{D} \in U(\mathfrak{g})$, the universal enveloping algebra
of $\mathfrak{g}$, which are monomials in the chosen basis of degree
$\leq d$.

Although defined {\bf for all $d$}, we use only for $d$ sufficiently large
that $\Sob_d$ majorizes pointwise values of derivatives; see important comment below.

We note that $\Sob_d$ defines a {\em Hermitian} norm on
$C^{\infty}_c(X)$.
 It has the following basic properties: there exists some
 $\consta\label{linftyconst}$, which we may assume to be larger than $\dim(G)$ for technical convenience later,  so that
 \index{xappa112@$\ref{linftyconst}$, $\Sob_d$ dominates $L^\infty$}
\begin{equation} \label{linftyfact}
 \|f\|_{L^{\infty}} \ll_d \Sob_d(f),\mbox{ for } d \geq \ref{linftyconst}\mbox{ and }
 \ \ f \in C^{\infty}_c(X)
 \end{equation}
and for every $d$ there exists\footnote{In this discussion of traces, we draw heavily on the work \cite{BR}
of Bernstein and Reznikov.} an integer $d' > d$
\begin{equation} \label{relativetrace} \mathrm{Trace}(\Sob_{d}^2 |
 \Sob_{d'}^2) < \infty\end{equation}
This means: if we complete $C^{\infty}_c(X)$ in the norm defined by
$\Sob_{d'}$ to obtain a Hilbert space
$V_{d'}$, the form $\Sob_d(f)^2=\langle A_df,f\rangle$ is
defined by a positive, compact operator $A_d: V_{d'}
\rightarrow V_{d'}$ with finite trace. In alternate
terms, there exists an orthonormal basis $\{e_i\}$ w.r.t.\
$\Sob_{d'}$ so that $\sum_{i} \Sob_{d} (e_i)^2 <
\infty$.

To establish  \eqref{linftyfact} and \eqref{relativetrace}, we use \eqref{injfact}
and reduce to $\R^{\dim G}$ via coordinate patches. We give details in
\S \ref{sobnormproofs}.

 We note that there is $\consta \label{sdconstant}$\index{xappa121@$\ref{sdconstant}$, distortion of Sobolev norm}
and $\constb \label{sd2constant}(d)$\index{yota121@$\ref{sd2constant}$, distortion of Sobolev norm}
(where the latter depends also on $d$) so that
\begin{equation} \label{sobolevdistort}
\Sob_d(g. f) \leq \ref{sd2constant}(d) \|g\|^{{\ref{sdconstant}} d }  \Sob_d(f) \mbox{ for any }
  f \in C^{\infty}_c(X)\mbox{ and } g \in G
\end{equation}
Here $g\in G$ acts on a function $f$ defined on $X$ by $(g.f)
(x)=f(xg)$.  To see this, combine \eqref{distortion}, \eqref{heightdistort},  \eqref{sobnormdef}.  
%
%For this note first that the derivative of $f^g$ along
%some element $u\in\g$ equals the derivative of $f$ along
%$\mathrm{Ad}(g^{-1})(v)$. Therefore, for any $\mathcal D$ as in
%\eqref{sobnormdef} the estimate in \eqref{distortion} shows that
%$\mathcal{D}(f^g)$ at $x$ is bounded in terms of $\mathcal{D}f$ at
%$xg$ and $\|g\|$. Moreover, \eqref{heightdistort} controls the
%change of the factor $\height(x)$ between $x$ and $xg$.

If $f_1, f_2 \in C^{\infty}_c(X)$, then, 
\begin{equation} \label{product}
\Sob_d(f_1 f_2) \ll_d \Sob_{d+\ref{linftyconst}}(f_1) \Sob_{d+
\ref{linftyconst}}(f_2),
\end{equation}
as may be deduced from the definition and \eqref{linftyfact}. 

Once $d\geq \ref{linftyconst}+1$ we have
\begin{equation}\label{simple-Sobolev-estimate}
 \| f - g.f\|_{L^{\infty}} \ll_d d(e,g) \Sob_d(f)
\end{equation}
This follows from \eqref{linftyfact}
applied to the partial derivatives of $f$
and by integrating the directional derivative of $f$ along a
geodesic curve connecting $x$ and $xg$.

\subsection{Notational conventions concerning Sobolev norms.} We observe two very important notational conventions concerning
Sobolev norms.

Throughout this paper, when we write $\Sob_d$,
we {\em always} assume that $d \geq \ref{linftyconst}+1$.

Moreover, in any statement that makes reference to a Sobolev
norm $\Sob_d$, the implicit constants in the symbol $\ll$ are permitted to depend on $d$.

As an example of these conventions: \eqref{product} could be
 legitimately abbreviated by
$\Sob_d(f_1 f_2) \ll \Sob_{d+\ref{linftyconst}}(f_1) \Sob_{d+
\ref{linftyconst}}(f_2)$.

\subsection{$L^2$-spaces}
If $\nu$ is a measure, we denote by  $L^2(\nu)$ the associated
$L^2$-Hilbert space and by $L_0^2(\nu)$ the orthogonal complement of
the constant function.

In some cases, it will be more natural to use $L_0^2$ to denote the orthogonal complement of {\em locally} constant functions: for instance, 
if the support of $\nu$ is disconnected.  We will always indicate clearly when this is the case.

\subsection{Almost invariant measures} \label{def:almostinvariant} \hypertarget{def:almostinvariant}{}
For any $g \in G$ and any measure $\nu$ on $X$, we denote by
$\nu^{g}$ the translated measure that is defined via $\nu^{g}(f) =
\int f(xg) d\nu(x)$ for $f\in C_c(X)$ or equivalently
$\nu^{g}(B)=\nu(Bg^{-1})$ for any measurable $B\subset X$.

In what follows, we shall define several notions of a measure being {\em almost invariant}. More precisely, these notions will be {\em almost
invariant w.r.t.\ to a Sobolev norm $\mathcal{S}=\mathcal{S}_d$.} We will sometimes omit reference to the Sobolev norm if it is clear from context.

We say that $\mu$ is $\epsilon$-almost invariant under $g \in G$ if
$$\left| \mu^g (f) - \mu(f) \right| \leq \epsilon \Sob(f), \ \ \ \ \ f \in C^{\infty}_c(X)$$
We say that it is $\epsilon$-almost invariant under a subgroup $Q$,
if it is $\epsilon$-almost invariant under every $q \in Q$ with $\|q\| \leq 2$. 
We say that it is $\epsilon$-almost invariant under $Z \in \mathfrak{g}$
if it is $\epsilon$-almost invariant under every $\exp(t Z)$ with $|t| \leq 2$.

%
%We say that $\mu$ is $\epsilon$-almost invariant under $Q$ if
%$$\left| \mu^q (f) - \mu(f) \right| \leq \epsilon \Sob(f), \ \ \ \ \ f \in C^{\infty}_c(X), q \in Q, \|q\| \leq 2$$

%
%If $Z \in \mathfrak{g}$ is an element  of the Lie algebra, we say
%that $\mu$ is $\epsilon$-almost invariant under $Z$ if
%$$\left| \mu^{\exp(t Z)} (f) - \mu(f) \right| \leq \epsilon \Sob(f), \ \ \ \ \ f \in C^{\infty}_c(X), |t|  \leq 2$$

These notions satisfy all the expected properties: e.g. if a measure is $\epsilon$-almost invariant under
$Z_1, Z_2\in\g$ with $\|Z_1\|,\|Z_2\|\leq 2$,
it is $\epsilon^\star$-almost invariant under $Z_1 + Z_2$.
%Also if $\mu$ is $\epsilon$-almost invariant under $Q$, then $\mu$ is also almost
%invariant under an $\epsilon^{-\star}$-cube in $Q$. The opposite implication
%holds as well. 

A list of such properties is given, with proof, in \S \ref{Appendix-A}.

\subsection{Spectral gap.}
There exists an integer $\tempered$, depending only on $G$, with the following property. 

Let $H \subset S \subset G$ be as before, and let $x_0 \in X$ be so that $x_0 S$ is a closed 
connected orbit.
Let $\nu$ be the $S$-invariant probability measure on $x_0 S$, and $L_0^2(\nu)$ the orthogonal complement of the constant function.
Then $L_0^2(\nu)^{\otimes (\tempered-1))}$ is {\em tempered} as an $S$-representation; for a review of this notion, see \S \ref{tempered}.
This applies, in particular, to the measure $\mu$ as in the statement of Theorem \ref{thm:main}.
We discuss in \S \ref{spectralgap} the origins of that statement and its consequences that we will need.

%%%%%%%%%%%%%%%%%%%%%%%%%%%%%%%%%%%%%%%%%%%%%%%%%%%%%%%%%%%%%%%%%%%%%%%%%%%%%%%%%%%%%%%%%%%%

\section{Outline of the proof of Theorem \ref{thm:main}.} \label{proof}
This section is intended as a summary of the paper.
We give a detailed outline of the proof of Theorem \ref{thm:main}, 
giving references and comparisons to the non-effective proof in \S \ref{semisimple-measures}. 
The actual proof of the Theorem is given, along the lines indicated here, in \S \ref{finalproof}. 

While this outline is intended to be a reasonable r{\'e}sum{\'e} of the proof,
we nonetheless have not indicated all the technicalities that are involved
in the complete proof.

We remark that the presentation of this section is designed so as to match \S \ref{semisimple-measures}.
The order in which results are discussed here, therefore, does not always match
the order in the body of the paper. \footnote{The primary reason for this
is that a certain group of results, which are used at different points in the argument,
have their roots in the polynomial divergence properties of unipotent flows.  Therefore,
we treat them together in the body of the paper. }

\subsection{Setup} 
Let, for the whole section, $G$, $H$, $\Gamma$, $X=\Gamma\backslash
G$, and $\mu$ be as in the statement of Theorem \ref{thm:main}. Let $S \supset H$ be a closed
subgroup.  As in the proof in \S \ref{semisimple-measures} we let
$\mathfrak{r}$ be an $\Ad(H)$-invariant complement of the Lie
algebra $\mathfrak s$ of $S$ in $\g$.
Recall that we have fixed a one-parameter unipotent
subgroup $u(t)$ of $H$ that projects nontrivially to each simple factor.

We shall
assume that $\mu$ is almost invariant under $S$, and will demonstrate
that {\em either} it is almost invariant under a larger subgroup; or 
$\mu$ is supported on a closed $S$-orbit of small volume. 

\subsection{Outline of the steps.} \label{outline2}
As discussed in \S \ref{outline} one can regard the proof in three steps:

\begin{enumerate}
\item {\em Ergodic theorem.}  The assertion that ``most points are generic''
is used in \S \ref{subsec:A} in the ineffective proof, and appears 
as Proposition A (\S \ref{subsec:AE}). 

\item {\em Nearby generic points give additional invariance.}
A form of this statement appears in \S \ref{subsec:B} in the ineffective proof;
an effective form will be given here in Proposition B (\S \ref{subsec:BE}). 

\item {\em Dichotomy:} 
If we cannot find two nearby generic points, 
as above, then necessarily $\mu$ was supported on a 
closed $S$-orbit of small volume.  This appears in \S \ref{subsec:C}
of the ineffective proof, and an effective form will appear here as Proposition C (\S \ref{subsec:CE}). 
 The proof of this is the most involved argument of the paper, involving
a closing result for actions of semisimple groups: an ``almost-closed'' orbit is near a closed orbit. 
\end{enumerate}

The conclusion of the non-effective proof, \S \ref{subsec:D}, makes use of some seemingly trivial principles: If $\mu$ is invariant under $S$ and under $g^{*} \notin S$, then it is invariant
by a strictly larger subgroup $S_* \supset S$; and, if $\mu(xS) > 0$,
then $\mu$ is the $S$-invariant measure on the closed orbit $xS$. 
To get the main Theorem from Propositions A, B and C requires that these facts be
made effective. 
We discuss these issues in \S \ref{almostsubalgebracopy}
and conclude the sketch of proof of the theorem in \S \ref{ABC}.

\subsection{Effective ergodic theorem (\S \ref{subsec:A}).}  \label{subsec:AE}

Fixing a suitable big integer $M$ -- the choice depends only on
$G$ --  we say that a point $x \in X$ is $[T_0, T_1]$-generic
w.r.t.\ $\mu$ if the natural probability measure on $\{u(t) x\}_{T^M
\leq t \leq (T+1)^M}$ approximates $\mu$ to within an error of about
$T^{-1}$, whenever $T \in [T_0, T_1]$ is an integer. The choice of
$M$ and a precise formulation of this is to be found in the
discussion around \eqref{fTdef}.

\begin{propA} (Proposition \ref{lemma3})\label{lemma3-copy}
Let $H \subset S \subset G$, $S$ connected. Suppose that $\mu$ is $\epsilon$-almost invariant
under $S$ w.r.t.\ $\Sob_{d}$, for $d \geq \ref{linftyconst}+1$.
Then there exists $\beta \in (0,1/2)$, $d' > d$, depending only on $G$, $H$, and $d$,
so that:

Whenever $R \leq \epsilon^{-\beta}$ and $T_0>0$, the fraction of points 
$(x,s) \in X \times \Sball(R)$ (w.r.t.\ $\mu$ resp.\ the Haar measure on $S$)
for which $x.s$ is not $[T_0,\epsilon^{-\beta}]$-generic w.r.t.\ $\Sob_{d'}$ 
is $\ll_d T_0^{-1}$.
\end{propA}

This assertion is the effective replacement of the discussion of \S \ref{subsec:A}. 

Let us contrast more carefully the above statement with that of \S \ref{subsec:A}. 
The Birkhoff ergodic theorem says, roughly speaking, that
the measure of $[T_0, \infty)$-generic points approaches $1$, as $T_0 \rightarrow \infty$ (but
doesn't give an error either in the rate of genericity nor in the measure of the set). 
\S \ref{subsec:A} isolates a large subset $E$ of such points. 
In the context
of \S \ref{subsec:A}, because the measure $\mu$ was {\em exactly $S$-invariant},
 the set of pairs $(x \in X ,s \in S)$
so that $xs \in E$ has large measure. 

In the almost-invariant context, 
 one obtains only that the set of pairs $(x \in X, s \in S)$ so that $xs$ is {\em very close to} $E$
 has large measure. In particular, for such pairs $(x,s)$, the point $xs$ is
 not $[T_0,\infty)$-generic, but only $[T_0, T_1]$-generic,
 where the size of $T_1$ depends on the strength of the notion of {\em very close}. 
 
 This discussion accounts for the difference in formulation between
 Proposition A and \S \ref{subsec:A}. 
 
 Proposition A is proved using the quantitative information about decay
of correlations provided by the spectral gap. It also makes use of the trace estimate
in \eqref{relativetrace}.

\subsection{Nearby generic points give additional invariance (\S \ref{subsec:B}).}  \label{subsec:BE}
We refer to \eqref{gdecomp} for the notation $r_1$, for $r \in \mathfrak{r}$. 

 \begin{propB} (Proposition
 \ref{MaL})\label{MaL-copy}
 Let $d\geq\ref{linftyconst}+1$.
    There exists constants $\consta\label{eipconstant}>0$ and $\consta\label{gammaconstant}>\consta \label{deltaconstant}>0$ so that:
 \index{xappa19@$\ref{eipconstant}$, exponent in additional invariance}
 \index{xappa2@$\ref{gammaconstant}$, genericity needed for additional invariance}
 \index{xappa21@$\ref{deltaconstant}$, genericity needed for additional invariance}

    Suppose $\mu_1, \mu_2$ are $H$-invariant measures,
 that $x_1, x_2 \in X$ satisfy $x_2 = x_1 \exp(r)$ for some nonzero $r \in \Scomp$, and that 
$x_i$ is $[\|r\|^{-\ref{gammaconstant}}, \|r_1\|^{-\ref{deltaconstant}}]$-generic w.r.t.\ $\mu_i$ 
and a Sobolev norm $\Sob_d$ (for $i=1,2$). 
    Then there is a polynomial $q: \R \rightarrow \mathfrak{r}$ of degree $\leq
\dim(\mathfrak{g})$,
so that:     
     $$\left| \mu_2^{\exp{q(s)}}(f) - \mu_1(f)\right| \ll_d \|r_1\|^{\ref{eipconstant}} \Sob_d(f), 1 \leq s \leq 2^{1/M}$$
If $r_1 \neq 0$, then $\max_{s \in [0,2]}\|q(s)\|$ can be set to be one.\footnote{However, we will need to set this maximum equal to some constant that depends only on $G$ and $H$.}
Moreover, if $\mu_1=\mu_2$ is $\epsilon$-almost invariant under $S$, then $\mu_1$ is $\ll_d \max(\epsilon,\|r\|^{\ref{eipconstant}})^{1/2}$-almost invariant under some $Z\in\mathfrak{r}$ with $\|Z\|=1$.
 \end{propB}

This Proposition, based on polynomial divergence, constitutes an effective version of the discussion from
\S \ref{subsec:B}. More specifically, when applied with $\mu_1 = \mu_2$, it yields an effective version of Lemma \ref{easy=proof}. \footnote{\label{f-pol}See \S \ref{existing-work} for a
discussion of the origins of that argument.}  The generalization to two distinct measures $\mu_1\neq \mu_2$ requires no effort, and 
is technically convenient for certain other applications (e.g., when studying
two distinct closed $S$-orbits). 

 The proof of this Proposition is an obvious quantification
 of \S \ref{subsec:B}.  The Proposition and related ideas have several corollaries
that will be used later in the proof.  In particular, it implies a quantitative version of the isolation
of distinct closed $S$-orbits.

%However, generating a
%bigger subgroup with almost invariance is not a triviality as it is
%when one works with invariant measures. This effective generation of
%a bigger subgroup is done in Proposition \ref{any-delta-copy} (i.e.
%Proposition \ref{any-delta}) which is used in combination with the
%generic points of Proposition \ref{lemma3-copy} (i.e.\ Proposition
%\ref{lemma3}) to gain additional invariance under favorable
%circumstances.

\subsection{Dichotomy (\S \ref{subsec:C}).} \label{subsec:CE}

 \begin{propC}\label{prop:dichotomy-copy} (Proposition \ref{prop:dichotomy})
Given a connected intermediate subgroup $H\subseteq S\subseteq G$,
$\mysymbol \in(0,1)$, and $d_S \geq 1$, there exists $\xi$ and $d'$ 
depending only on $d_S$, $\mysymbol$, $G$, $H$, and $\epsilon_0
\ll_{d_S} 1$, so that:

 Suppose for some $\epsilon \leq \epsilon_0$ that
\begin{enumerate}
\item  $\mu$ is $\epsilon$-almost invariant under $S$, with respect to the Sobolev norm $\Sob_{d_S}$, and that 
\item  $\mu(x S) = 0$ for all closed $S$-orbits of volume $\leq \epsilon^{-\mysymbol}$.
\end{enumerate}
Then there exists $x_1, x_2$ so that $x_2 = x_1 \exp(r)$,
$r \in \mathfrak{r}$, $\|r\| \leq \epsilon^{\xi}$, and
 $x_1,x_2$ are both
$[\|r\|^{-\ref{gammaconstant}},
\|r_1\|^{-\ref{deltaconstant}}]$-generic w.r.t.\ $\Sob_{d'}$. 
\end{propC}

This Proposition is an effective analog of \S \ref{subsec:C}, in particular, an effective analog of Lemma \ref{step=two}. 
However, the proof is considerably more involved.   It necessitates, in particular, an ``effective closing result'': the assertion that if an orbit $x S$ is ``almost closed'', there exists $x'$ near $x$ 
that is closed.

We devote most of the present subsection to explaining the proof of Proposition C. 
We begin by enunciating two results which are, in spirit, quite close to Proposition B (= Proposition \ref{MaL}); for, like that Proposition, 
they are based in concept on the polynomial divergence properties of unipotent flows. 
For that reason, in the text, they are proved immediately after Proposition~\ref{MaL}. 

\begin{lem} (Lemma \ref{iso}) \label{iso-copy} (Quantitative isolation of periodic orbits for semisimple groups)
There are constants $\consta \label{separationheightconstant}$
and $\consta \label{separationvolumeconstant}$ with the following property.
\index{xappa212@$\ref{separationheightconstant}$, exponent of height in isolation}
\index{xappa2121@$\ref{separationvolumeconstant}$, exponent of volume in isolation}

Suppose $H \subset S \subset G$. 
Let $x_1, x_2\in X$ be so that $x_i S$ are closed orbits with volume $\leq V$.
Then either:
\begin{enumerate}
\item  $x_1$ and $x_2$ are on the local $S$-orbit, i.e.\ there exists some $s\in S$ with $d(s,1)\leq 1$
and $x_2=x_1 s$, or
\item $d(x_1, x_2) \gg \min(\height(x_1),\height(x_2))^{-\ref{separationheightconstant}} V^{-\ref{separationvolumeconstant}}$.
\end{enumerate}
\end{lem}

The Lemma in combination with Lemma \ref{measureoutsidecompact} implies, in particular, that the {\em total} number of closed $S$-orbits with volume $\leq V$
is bounded by a polynomial in $V$.  

\begin{propG} (Proposition \ref{lem:Quantiso})\label{lem:Quantiso-copy}
Let $S \supset H$. 
There exists some $V_0$ depending on $\Gamma$, $G$, and $H$ and some
$\consta>0\label{expQuantiso}$ with the following
property.\index{xappa22@$\ref{expQuantiso}$, exponent in isolation}
Let $V\geq V_0$ and suppose $\mu(Y) = 0$ if $Y$ is any closed
$S$-orbit of volume $\leq V$. Then:
\begin{equation}\label{freddy-copy}
\mu\bigl(\bigl\{x  \in X: \mbox { there exists }x'
\stackrel{V^{-\ref{expQuantiso}}}{\sim} x \mbox{ with }\vol(x' S)
\leq V\bigr\}\bigr) \leq1/2\end{equation}
\end{propG}
This Proposition is again based on polynomial divergence. %$^{\ref{f-pol}}$
For each closed orbit $x' S$ as in
\eqref{freddy-copy}, the $\mu$-mass of a
$V^{-\ref{expQuantiso}}$-neighbourhood of $x'S$ is quite small. The
idea of this is that one can linearize the flow in the neighbourhood
of $x' S$ and thereby understand it completely.
Taken in combination with Lemma \ref{iso-copy} (\ref{iso}), this yields 
the claimed result.

\begin{propH} (Proposition \ref{CaseIIlemma})\label{CaseIIlemma-copy}
Let $H \subset S \subset G$, with $S$ connected.
Let $\delta\leq 1\leq N$. There exists some $T_0 = T_0(\Gamma, G, H, N)$ with the following
property:

 Let $T\geq T_0$ and let $v = \vol \Sball(T)$. Suppose that
$\{s_1,\ldots,s_k\}\subset \Sball(T)$ is $\frac1{10}$-separated, that
$k\geq v^{1-\delta}$, and that there exists $x \in \Xcompact$ so
that $x s_i \stackrel{T^{-N}}{\sim} x s_j$.
Then there is
$x' \stackrel{T^{-N_\uparrow}}{\sim} x$ so that $x' S$ is a closed orbit
of volume $\leq T^{\delta_{\downarrow}}$.
\end{propH}

This result constitutes an effective ``closing Lemma'' for actions of semisimple groups on homogeneous spaces.  

{\em Idea of proof.} 
 Choosing a representative $g_0 \in G$ for $x$, the fact that $x s_i$ is close to $x s_j$
means that there is an element $\gamma \in \Gamma$ that moves $g_0 s_i$ near to $g_0 s_j$. 
Thereby, one constructs a whole collection of elements of $\Gamma$ {\em near to}
$g_0 S g_0^{-1}$.

We first establish that all these elements must
themselves lie on a different conjugate $g_0' S g_0'^{-1}$. Here $g_0' \in G$
is very close to $g_0$.  This uses, in particular, the arithmetic nature of $\Gamma$. 
A baby version of this argument is the fact that, if three points in $\Z^2 \cap [-N, N]^2$
all lie within $\frac{1}{100 N}$ of a certain line in the plane, then they 
must {\em exactly} lie on a nearby line. 

Next, we show that the existence of so many elements in $\Gamma$ that lie inside $g_0' S g_0'^{-1}$
force the subgroup $g_0' S g_0'^{-1} \cap \Gamma$ to be a lattice.  
This step uses a spectral gap. 
%
%The use of the spectral gap to give an effective closing lemma is closely
%related to its use in lattice count problems, see, e.g. \cite{Leuzinger} and \cite{Quint}.
A simple statement with the same general flavor as what we need is the following: Suppose 
we are given a discrete subgroup $\Lambda \subset \SL_3(\R)$ so that
\begin{equation} \label{toymodel}\lim_{T \rightarrow \infty}  \frac{
\#\{m \in \Lambda,\|m\| := \sqrt{\mathrm{Trace}(m. m^t)}
\leq T\}}{T^{5.9}} = \infty\end{equation}
Then $\Lambda$ is a lattice. 
(For comparison, the asymptotic
for $\SL_3(\Z)$ is $T^6$.)

This statement may be derived from Property (T). (The fact that $\Lambda$ has many elements translates
to the existence of a subrepresentation of $L^2(\Lambda \backslash \SL_3(\R))$ near
to the trivial representation.) The precise form of what
we use is somewhat different and is stated in Proposition \ref{sl32}, part (2).

Therefore, we have established that $x' := \Gamma g_0'$ is close to $x=\Gamma g_0$,
and $x' S$ is a closed $S$-orbit.  A more detailed study shows that the volume of $x' S$ is also small. 
\qed

Now, we sketch the proof of Proposition C. 

By Proposition A, for ``most'' $x \in X$, the measure of 
$\mathcal{B}_x := \{s \in S: xs \mbox { generic}\}$ is ``large.''  (Here ``generic'' means the notion specified in the statement of Proposition C.)

More is true: let $O$ be a small neighbourhood of the identity in $S$.
Given $x \in X$, let 
$\mathcal{B}_x'$ consist of $s \in S$ so that $\vol(s O \cap \mathcal{B}_x) > 0.99 \vol(O)$. 
By Fubini's theorem,
the volume of $\mathcal{B}_x'$ is large for ``most'' $x$. 
(Compare the usage of Fubini's theorem in \S \ref{subsec:C}.)

Supposing, for simplicity, that $X$ were compact, we cut $X$ into boxes of size $\epsilon^{\star}$.
Let us observe that if one such $\epsilon^{\star}$ box contained a large fraction of the volume of $x . \mathcal{B}_x'$,
then \propHname would show that there is $x'$, very near to $x$, so that $x' S$ is closed. 
\propGname shows that this cannot happen, at least for ``most'' $x$. 
Therefore, \propGname shows that, for at least one $x \in X$,
no single $\epsilon^{\star}$-box contains a large fraction of the volume of $\mathcal{B}_x$. 

Consequently, for at least one $x \in X$, 
there exist two distinct but close $\epsilon^{\star}$-boxes, each of which contain a
point of the form $x s, s \in \mathcal{B}_x'$. It is easy to see --
by ``adjusting along the $S$-direction'' -- 
that this implies the conclusion of Proposition C.

\subsection{Effective version of \S \ref{subsec:D}} \label{effgen1}
In this section, we discuss some of the results which effectivize
the rather trivial-seeming principles used implicitly in \S \ref{subsec:D}. 
After this, we present the proof of Theorem \ref{thm:main}.

\begin{propD} \label{almostsubalgebracopy} (Proposition \ref{almostsubalgebra})
Fix a Sobolev norm $\Sob_d$; all notions 
of almost invariance are taken with respect to this.  Suppose that $\mu$ is $\epsilon$-almost
invariant under a subgroup $S$ and also under $Z \in \mathfrak{g}$
such that $Z \perp \mathfrak{s}$ and $\|Z\|=1$ where $\mathfrak{s}$
is the Lie algebra of $S$.

Then there is a constant $\consta
\label{almostinvconstant}$\index{xappa23@$\ref{almostinvconstant}$,
generating a bigger subgroup with almost invariance} so that  $\mu$
is also $\constc(d) \epsilon^{\ref{almostinvconstant}}$- almost
invariant under some subgroup $S_*$ with $\dim(S_*) > \dim(S)$.  If
$H \subseteq S$, we may also assume that $H \subseteq S_*$.
\end{propD}

This is proved in \S \ref{Appendix-A}, and effectivizes
the statement ``If $\mu$ is invariant under $S$ and under $g^{*} \notin S$, then it is invariant
by a strictly larger subgroup $S_* \supset S$'' used in \S \ref{subsec:D}. 
The difficulty may be seen in the following case: let $S_1$ be an
intermediate subgroup, and suppose that $Z$ lies {\em almost} in the
Lie algebra of $S_1$. Then, in the statement of the proposition, one
wants to take $S_* = S_1$ rather than the group $S_2$ generated by
$H$ and $\exp(Z)$. Indeed, because $Z$ lies almost in $S_1$, the
subgroup $H$ and the element $\exp(Z)$ generate $S_2$ very
``inefficiently.'' 

We deduce it from the corresponding Lie algebra statement.
Given $T = \{t_1, \dots \} \subset \mathfrak{g}$,
we set  $T^{(k)}$ to be the set of all possible iterated Lie
brackets of the $t_i$s of depth $\leq k$.

\begin{propE}\label{any-delta-copy}(Proposition \ref{any-delta})
There exists an integer $k$ and some $c>0$ depending only on $G$
with the following property. For any orthonormal subset $T = \{t_1,
\dots \}$ of $\mathfrak{g}$, and $0 < \delta <  1$ there exists a
subalgebra $\mathfrak{w} \subset \mathfrak{g}$ with orthonormal
basis $w_1, w_2, \dots$ satisfying:
\begin{enumerate}
\item
For each $w_i$, there exists a linear combination $w_i' = \sum_{t
\in T^{(k)}} c_t t$, with $c_t \in \R$ satisfying $|c_t| \leq
\delta^{-k}$, so that $\|w_i - w_i'\| \ll \delta^c$.
\item Each $t \in T$ is within $\delta$ of $\mathfrak{w}$ (i.e. $\min_{w \in \mathfrak{w}} \|w-t\| \leq \delta$).
\end{enumerate}
  If the linear span $\langle T \rangle$ of $T$ contains a subalgebra $\mathfrak{h}$,
  then there exists $k$ and $c$ depending on the pair $(\mathfrak{h} , \mathfrak{g})$ so that
  (1) and (2) hold and, in addition, $\mathfrak{h} \subset \mathfrak{w}$.
\end{propE}
In explicit terms, $\mathfrak{w}$ is a Lie subalgebra
 that is
efficiently generated by small perturbations of the elements of $T$.
The proof of this is (at least intuitively) clear, the $t_i$ span an ``almost-subalgebra''
in an efficient way, and because the space of subalgebras of $\mathfrak{g}$
is compact, this almost-subalgebra is near a genuine algebra.
We will use the Lojasiewisz inequality to carry that argument through in an effective way.

\begin{propF} (Proposition \ref{nearness})\label{nearness-copy}
Let $x_0 \in X$ be so that $x_0 S$ is a closed orbit of
volume $V$,  for some connected $S \supset H$. 
 Suppose $\mu$ is a probability measure on $x_0 S$ that is
$\epsilon$-invariant under $S$ w.r.t.\ a Sobolev norm $\Sob_d$.
Let $\nu$ be the $S$-invariant probability measure on $x_0 S$.

   Then there are $\consta\label{V-constant}$ and $\consta\label{Rconstant}$%
\index{xappa3@$\ref{V-constant}$, exponent of $V$ in
Lemma~\ref{nearness-copy}~(\ref{nearness})}\index{xappa31@$\ref{Rconstant}$,
exponent of $R$ in Lemma~\ref{nearness-copy}~(\ref{nearness})} so
that
\begin{equation}\label{c-smoothness}
  \left| \mu(f)  - \nu(f) \right| \ll_{d}
V^{\ref{V-constant}} \epsilon^{\ref{Rconstant}/d}\Sob_d(f)   \mbox{ for }
   f \in C_c^{\infty}(X). \end{equation}
   
In particular, there are constants $\consta \label{Aconstant}, \consta \label{Bconstant} > 0$%
\index{xappa32@$\ref{Aconstant}$, bound on $V$ in terms of $\epsilon$ in
Lemma~\ref{nearness-copy}~(\ref{nearness})}\index{xappa33@$\ref{Bconstant}$,
exponent for closeness in
Lemma~\ref{nearness-copy}~(\ref{nearness})}
 such that if $V\leq \epsilon^{-\ref{Aconstant}/d}$ then
 $\mu$ and $\nu$ are $\epsilon^{\ref{Bconstant}/d}$-close: 
 \begin{equation*}
  \left|\mu(f)  - \nu(f) \right| \ll_{d} \epsilon^{\ref{Bconstant}/d}\Sob_d(f)   \mbox{ for }
   f \in C_c^{\infty}(X).
 \end{equation*}
\end{propF}

This result is an effective version of the statement ``$\mu(xS) > 0$ implies
that $\mu$ is the $S$-invariant measure on the closed orbit $xS.$''  
 To prove it, let $\chi \in C_c(S)$ be a positive compactly
supported function on $S$ with integral one. For a smooth function
$f$ as in \eqref{c-smoothness}, the values $\mu(f)$ and $\mu(f \star
\chi)$ are approximately the same by almost invariance. However,
repeated applications of the convolution operator $f \mapsto f \star
\chi$ makes $f$ converge to a constant function, whence the result.

\subsection{The theorem is as easy as ABC}\label{ABC}

Say that $\mu$ is $[S, \epsilon , d_S]$-almost invariant if it is
$\epsilon$-almost invariant under a connected subgroup $S \supset H$, w.r.t.\ $d_S$. 
Propositions D and F, from the previous section, easily convert
Propositions A,B,C to the following dichotomy:

\begin{lem*} (\S \ref{completeness})
Suppose that $\mu$ is $[S,\epsilon, d_S]$-almost invariant.
There exists constants $\constc(d_S) \label{epsconst}$, 
$\consta(d_S) \label{volumeconst1}$,
$\consta(d_S) \label{volumeconst2}$,
 $\consta(d_S) \label{moreinvarianceconstant}$, and $d_S'$%
\index{xappa331@$\ref{volumeconst1}$, Last Lemma to iterate}%
\index{xappa332@$\ref{volumeconst2}$, Last Lemma to iterate}%
\index{xappa333@$\ref{moreinvarianceconstant}$, Last Lemma to iterate}%
so that for any $\epsilon\leq\ref{epsconst}$ either:
\begin{itemize} \item
$ |\mu(f) - \mu_{x_0 S}(f)| \leq \constc(d_S)
\epsilon^{\ref{volumeconst2}(d_S)} \Sob_{d_S}(f)$, 
for some closed orbit $x_0 S$ of 
volume $\leq \epsilon^{-\ref{volumeconst1}(d_S)}$.
\item The measure $\mu$ is 
$[S_* ,  \constc(d_S) \epsilon^{\ref{moreinvarianceconstant}(d_S)},
d_S']$
almost invariant, where the connected subgroup $S_* \supset H$ has larger dimension than $S$. 
\end{itemize}
\end{lem*}
Iterating this lemma yields Theorem \ref{thm:main}.

%
%To effective the section \S \ref{subsec:D} requires a little more effort. 

%In particular, one needs
%to establish:
%\begin{itemize}
%\item If $\mu$ is almost invariant under $S$ and under $g^* \notin S$, it is almost invariant under a larger subgroup $S_*$. This intuitively obvious result relies on an ``effective generation''
%result. See \S \ref{effgen1}.
%\item If $\mu$ is supported on a closed $S$-orbit of small volume, and almost invariant under $S$,
%then it is indeed close to the $S$-invariant measure on the closed orbit.  See roposition G
%of \S \ref{effgen1}.

%%%%%%%%%%%%%%%%%%%%%%%%%%%%%%%%%%%%%%%%%%%%%%%%%%%%%%%%%%%%%%%%%%%%%%%%%%%%%%%%
\section{Basic properties of Sobolev norms.} \label{sobnormproofs}

\subsection{Bounding the $L^\infty$-norm}
We extend the notion of Sobolev norm to 
$\R^n$.
We define for $f \in C^{\infty}(\R^n)$ and the open set $B_\epsilon=(-\epsilon,\epsilon)^n\subset \R^n$
the Sobolev norm
\[
 \Sob_{d, B_\epsilon}(f)^2 = \sum_{|\underline{\alpha}| \leq d} \|f^{(\underline{\alpha})}\|_{L^2(B_\epsilon)}^2,
\]
where $f^{(\underline{\alpha})}$ denotes the $\underline{\alpha}$-partial derivative, and
the sum is extended over all multi-indexes $\underline{\alpha}$ of degree $\leq d$.

The following establishes \eqref{linftyfact}.

\begin{lem} \label{upperboundlemma}
 Let $n\geq 1$. There exists a constant $c(n)$ for which
\begin{equation} \label{rn1}
 |f(x)| \leq c(n)  \varepsilon^{-n} \Sob_{n, B_{\varepsilon}}(f)
\end{equation}
for all $\varepsilon \leq 1$, $f\in C^\infty(\R^n)$, and all $x\in B_\epsilon$.

Moreover, there exists a constant $\constc(d)$ and $\ref{linftyconst}$
such that, whenever $\mathcal{D} \in U(\mathfrak{g})$ has order
$\leq r$, 
\begin{equation}\label{linftytechnicalestimate}|\height(x)^{r} \mathcal{D} f|\leq \constc(r, \mathcal{D}) \Sob_{d+r}(f)\end{equation}
for all $f\in C_c^\infty(X)$ whenever $d\geq \ref{linftyconst}$.
\end{lem}

\proof
 The statement in \eqref{rn1} follows for $n=1$ quite easily: If there
 is no $|f(y)|$ smaller than $\frac{1}{2}|f(x)|$ the statement is clear, otherwise the first
 derivative will have an integral $\geq\frac{1}{2}|f(x)|$. Using Fubini's theorem
 \eqref{rn1} follows by induction on $n$.

 Let now $x \in X$ and $f \in C^{\infty}(X)$. We consider the function
$g \mapsto f(xg)$ in a ball of the form $d(g,1) \leq \ref{hc1}
\height(x)^{-\ref{hc2}}$, which by \eqref{injfact} is an injective image of the
corresponding ball in $G$. 
 Choosing a coordinate chart for $G$
around the identity, we transfer the question to $\R^{\dim G}$ and apply
\eqref{rn1}. 

%First notice that here we have to switch from integration
%w.r.t.~the Haar measure to integration w.r.t.~the Lebesgue measure as well as
%from differentiation along the basis elements of $\g$ to differentiation along
%the coordinates in the chart. This only introduces some constants in the
%estimates. Second notice that we have to choose $\ref{linftyconst}\geq\dim G$
%so that we may apply \eqref{rn1} with $n=\dim G$ and so that the decay of the
%injectivity radius and the corresponding term $\varepsilon^{-n}$ in \eqref{rn1} is
%controlled by the power of $\height(x)$ in the definition of $\Sob_d$ for
%$d\geq\ref{linftyconst}$.  With these choices 
%\eqref{linftytechnicalestimaet}, and so also \eqref{linftyfact}, follows. 
\qed

\subsection{Trace estimates} \label{sobnormproofs-trace}

We are going to use the following basic properties of
the relative trace. See the paper of Bernstein-Reznikov \cite[App.~A]{BR}
for a more careful development of these ideas.  The relative trace of two Hermitian forms $A$ and $B$ (the latter being
positive definite) on a complex vector space $V$
is defined as
\[
 \Tr(A|B)=\sum_i\frac{A(e_i)}{B(e_i)}
\]
if $V$ is finite-dimensional and $e_1,\ldots$ is an orthogonal basis of $V$ with respect to $B$.
\footnote{This is well defined since it equals $\Tr(B^{-1}A)$ if we consider $A$ and $B$ as maps to the
Hermitian dual $V^+$ of $V$.}
In the infinite dimensional case we define
\[
 \Tr(A|B)=\sup_{W\subset V}\Tr(A_W|B_W)
\]
where the supremum goes over all finite dimensional subspaces $W$ of $V$
and $A_W$, $B_W$ are the restrictions of $A$ and $B$ to $W$.
Then
$$
 \Tr(A|B) \leq \Tr(A'|B)\mbox{ if }A \leq A' \mbox{ and }\Tr(A|B') \geq \Tr(A|B)\mbox{ if }B \geq B'
$$
for any Hermitian forms $A$, $A'$, $B$, and $B'$ on $V$. Here
$A\leq A'$ means that $A(v)\leq A'(v)$ for all $v\in V$. If $A\leq B$
then $\Tr(A|B)=\sum_i A(e_i)$ where $e_1,\ldots$ is
an orthonormal basis w.r.t.~$B$. This was already claimed in \S \ref{notation}
and follows from \cite[Prop~A.2]{BR}\footnote{In \cite[App.~A]{BR} continuous
Hermitian forms on a topological vector space were considered, but with $A\leq B$
both forms become continuous w.r.t.\ the norm derived from $B$.}. 

We will be using that
notion and these facts also in the case of a general, not necessarily positive definite, Hermitian form $B$,
but only in the case when $A$ is zero in the radical of $B$ (the latter being the subspace
where $B$ is zero). In this case, the trace is defined as the trace on the quotient of $V$ by the radical.

Although we do not explicitly use it, the following gives a helpful ``normalization'' of the idea of relative traces. Suppose $N_1, N_2, \dots, N_r$
are Hermitian norms on $V$, so that $N_i$ is bounded by a multiple of $N_{i+1}$;
and the relative trace $\Tr(N_i|N_{i+1})$ is finite.
Then there exists an orthogonal basis $e_1, \dots, e_k, \dots$
for the completion of $V$ with respect to $N_r$, so that $N_1$
is also diagonal with respect to this basis, and moreover:
\begin{equation} \label{iterationrelativetrace}
\sup_{j} j^{r-1} \frac{N_1(e_j)}{N_r(e_j)}  < \infty. \end{equation}

\subsection{Proof of \eqref{relativetrace}}

For $x \in X$, consider the linear form on $C^{\infty}_c(X)$ given by:
$$L_x: f \mapsto \height(x)^{r} \mathcal{D}f,$$
where $\mathcal{D}$ is a monomial in the chosen orthonormal basis of $\mathfrak{g}$ with degree $\leq r$. 
By Lemma \ref{upperboundlemma}, 
\begin{equation} \label{preliminarytraceestimate}
\Trace(|L_x|^2|S_{d}^2) <  c_2(r, \mathcal{D})^2\end{equation}
whenever $d \geq \ref{linftyconst} + r$. Integrating
\eqref{preliminarytraceestimate} over $x \in X$, and summing
over $\mathcal{D}$,  yields \eqref{relativetrace}
(cf. \cite[App. A, Prop. 1]{BR}).  \qed

%%%%%%%%%%%%%%%%%%%%%%%%%%%%%%%%%%%%%%%%%%%%%%%%%%%%%%%%%%%%%%%%%%%%%%%%%%%%%%%%

\section{Basic properties pertaining to spectral gap.} \label{spectralgap}

Our primary aim in this section is to prove the following result, as well
as establishing enough background on spectral gap to take advantage of it.

\begin{prop} \label{sectionslice}
Notation being as in \S \ref{gi},
let $H \subset S \subset G$. Let $x \in X$, and $S_x$ the stabilizer of $x$ in $S$.

Then the action of $S$ on the orthogonal complement to locally
constant functions in $L^2(S_x\backslash S)$
has a spectral gap. Moreover, that spectral gap depends only on $G,H$.
\end{prop}

We observe that, in the case when $H$ has property (T), this result is obvious.

\subsection{The spectral gap and tempered representations.} \index{tempered}\label{tempered}

For general reference on semisimple groups, see \cite{Wallach}.
We refer also to \cite{Nevo} where the relation of spectral gap and matrix coefficient decay is discussed and used. 
Let $S$ be a (not necessarily connected)  semisimple real group with finite center.

\subsubsection{Spectral gap: definition.}We say that a unitary representation $(\pi, V)$ of $S$ -- not necessarily irreducible -- possesses a spectral gap if
\footnote{It would be more usual to make this definition without the connected component restriction; the present definition is more convenient for our purposes.}
 there is a compactly supported probability measure $\nu$ on the {\em connected component of} $S$, and $\delta > 0$, so that:
$$\|\pi(\nu) v \| < (1-\delta) \|v\|, \ \ (v \in V).$$

It is equivalent to say that the ``irreducible constituents of $\pi$,''
as a subset of the unitary dual of $S^0$, are isolated from the trivial representation in the Fell topology.

\subsubsection{Tempered: definition.} We say that an irreducible, unitary representation $\pi$ of $S$ is {\em tempered} if
it is weakly contained in the regular representation $L^2(S)$.
 We say that a unitary representation $(\pi,V)$
is tempered if it may be disintegrated into tempered representations.

It is equivalent to ask \cite[Theorem 1]{CHH} that
there exist a dense subset $\mathcal{V} \subset V$ so that $s \mapsto \langle \pi(s) v, w \rangle$
belongs to $L^{2+\varepsilon}(S)$ for each $v,w \in \mathcal{V}$ and $\varepsilon > 0$.%
\footnote{More precisely, that Theorem asserts that, if $v \in V$
is so that the diagonal matrix coefficient $\langle \pi(s) v, v \rangle$
belongs to $L^{2+\varepsilon}$, then the representation of $S$
on $\overline{\langle S v \rangle}$ is weakly contained in the regular representation. 
This means that the diagonal matrix coefficient $\langle \pi(s) v, v \rangle$ is uniformly 
approximable, on compacta, by convex combinations of diagonal matrix coefficients associated 
to the regular representation. Clearly,
if this property is valid for a bounded sequence $v_i$, it is also valid
for any limit of the $v_i$s; whence the stated conclusion.}
It is also equivalent to ask that the restriction of $\pi$ to the identity component of $S$ be tempered; 
or that the pull-back of $\pi$ to any finite covering of $S$ be tempered.

If $\pi$ is an irreducible representation of $S$ with compact kernel, then
there exists an integer $m \geq  1$ so that $\pi^{\otimes m}$ is tempered.
This provides a useful measure how close to tempered a representation is:
We say a unitary representation $\pi$ is $1/m$-{\em tempered} if $\pi^{\otimes m}$ is tempered. (It is equivalent to say that it is $L^{2m+\varepsilon}$
for all $\varepsilon > 0$; the latter phrasing appears often in the literature.)

\subsubsection{Relations between spectral gap and temperedness.} \label{Rel:SGT}If $S$ is almost simple, 
\begin{multline}\label{temp} \mbox{If $(\pi, V)$ possesses a spectral gap,}\\
\mbox{then $\pi$ is $1/m$-tempered for some $m \geq 1$.}\end{multline}
This assertion is by now  a ``folk'' result, but is rather remarkable.
The property of having a spectral gap is ``local'' in nature: it depends
on the action of group elements near the identity. On the other hand,
the property of being $1/m$-tempered constrains the action of group
elements very far from the identity. 

The proof of this assertion may be obtained by combining
various results in the literature; see Appendix C.
Obviously this is a rather
unsatisfactory state of affairs; we don't know of any entirely conceptual proof.

This (\eqref{temp}) is false if $S$ is not almost simple. According to our definition, the representation $1 \otimes L^2(\SL_2(\R))$ of $\SL_2(\R) \otimes \SL_2(\R)$ possesses a spectral gap. However,
it is not $1/m$-tempered for any $m \geq 1$.

\subsection{Some facts on matrix coefficients}  \label{matrixcoeff}
We restrict, in the present section (\S \ref{matrixcoeff}), to the
case of $S$ connected.  \footnote{This is mainly because of the lack of suitable references in general.}

 $S$ admits  an Iwasawa decomposition $S = N . A . K$; here $K$ is a maximal compact subgroup of $S$ which is the fixed point set of a global Cartan involution $\Theta: S \rightarrow S$. 
%The group $K$ meets every connected component of $S$.
Accordingly, there is a projection $H_A:  S \rightarrow A$.
\footnote{
The notions we are about to define depend on the choice of $K$,
although this dependency is not very important. Moreover, we use the
results of this section for $S$ an intermediate group between $H$ and $G$;
for such we have fixed a choice of maximal compact after Lemma \ref{assumption-2}.}

We define the Harish-Chandra spherical function $\varphi_0$ via the rule:
\begin{equation} \label{HCdef}
\varphi_0(s) := \int_{k \in K} H_A(ks)^{\rho} dk\end{equation}
where $\rho:A \rightarrow \R_{+}$ is the half-sum of the roots (counted with multiplicities) of $A$ acting on $N$,
and $dk$ is the probability Haar measure on $K$.

$S$ also admits a Cartan decomposition $S = K A^{+} K$. The function $\varphi_0$ is bi-$K$-invariant
and belongs to $L^{2+\varepsilon}(S)$, for every $\varepsilon > 0$.   Moreover, in obvious notation (see \cite[Proposition 7.15(c)]{Knapp-red})
\begin{equation} \label{mcgeneral} 
\left| \varphi_0(k a_+ k')  \right| \ll_{\varepsilon} \rho(a_+)^{-1+\varepsilon },
\end{equation}
for any $\varepsilon > 0$; the implicit constant here depends
on the isomorphism class of $S$ only.

  Moreover, for any one-parameter unipotent subgroup $u: \R \rightarrow S$ we have the bound
\footnote{As can be deduced, e.g., by embedding $u$ in an $\SL_2(\R)$.}
\begin{equation} \label{mcunipotent} 
\left| \varphi_0(u(t)) \right| \ll_{\varepsilon} (1+|t|)^{-1+\varepsilon}.
\end{equation}

 It will be convenient to make a slight generalization of this definition
to handle some slight complications arising from groups with multiple simple factors. If $S$ is a direct product 
$S_1 \times \dots S_k$ of almost simple groups, we set:
\begin{equation}\label{weakdef}\varphi_0^{\wk}(s_1, \dots, s_k) = \max_{1 \leq i \leq k} \varphi_{0,S_i}(s_i).\end{equation}
Even if $S$ fails to be a direct product, i.e.
is isogenous to a direct product, there exists an isogeny
$\prod_{i=1}^{I} S_i \rightarrow S$; the function $\varphi^{\wk}_0$
 defined on $\prod_{i=1}^{I} S_i$ by \eqref{weakdef} is bi-invariant by a maximal
compact of $\prod_{i=1}^{I} S_i$, and in particular by the kernel
of the isogeny. It thereby descends to $S$. 
% nonetheless the group ``$A$''
%appearing above admits a direct product decomposition corresponding to the simple factors, and we define $\varphi_0^{\wk}$ on $A^{+}$ via \eqref{weakdef},
%and via $S = K. A^{+}. K$ in general. 

\begin{lem}
Let $\clubsuit \subset S$ be bi-$K$-invariant, i.e
$K. \clubsuit .K = \clubsuit$.  
Then
\begin{multline} \label{clubestimate}\vol(\clubsuit)^{-1} \int_{\clubsuit} \varphi_0(s) ds 
 \ll \vol(\clubsuit)^{-1/3}, \\ 
 \vol(\clubsuit)^{-2}\int_{s,s' \in \clubsuit} \varphi_0(s s'^{-1})^{1/p}
\ll \vol(\clubsuit)^{-\frac{2}{3p}}.\end{multline}
The implicit constants in these estimates depend only the isomorphism
class of $S$ and the choice of Haar measure. 
\end{lem}
\proof The first assertion follows
by the duality between $L^3$ and $L^{3/2}$. 
%$$\vol(\clubsuit)^{-1} \int_{\clubsuit} \varphi_0(s) ds \leq (\int_{\clubsuit} \varphi_0(s)^3)^{1/3} (\int_{\clubsuit} 1)^{2/3-1}  \ll \vol(\clubsuit)^{-1/3}.$$

Next, let $p \geq 1$, and consider $\int_{\clubsuit}
 \varphi_0(s s'^{-1})^{1/p} ds \, ds'$. Let $q$ satisfy $1/p +1/q=1$; then, by duality between $L^p$ and $L^q$, 
$$\vol(\clubsuit)^{-2}\int_{s,s' \in \clubsuit} \varphi_0(s s'^{-1})^{1/p}
\leq \left( \int_{s,s' \in \clubsuit} \varphi_0(s s'^{-1}) \right)^{1/p}
\vol(\clubsuit)^{-2/p}$$ 
Noting the identity $\int_{k \in K} \varphi_0(s_1 k s_2)
= \varphi_0(s_1) \varphi_0(s_2)$, cf. \cite[(7.45)]{Knapp-red}, we simplify this to
\eqref{clubestimate}. 
\qed

\subsubsection{Bounds for matrix coefficients of tempered and $1/m$-tempered representations.} If $(\pi, V)$ is a tempered
representation of $S$, it is known \cite{CHH} that:
\begin{equation} \label{mc1} 
\langle s.v , w \rangle \leq \varphi_0(s) . \|v\| \|w\|  (\dim Kv)^{1/2} . (\dim K w)^{1/2} 
\ \ \ (v,w \in V)
\end{equation}

Fix a basis for the Lie algebra $\mathfrak{s}$ of $S$. 
We may define a system of Sobolev norms $\Sob_d^V$ on the 
smooth subspace\footnote{Here $v \in V$ is smooth if the associated orbit map 
$G \stackrel{g \mapsto g. v}{\rightarrow} V$ is smooth,}
of any unitary representation $V$ via\footnote{It should be noted that, when
specialized to the case of $V = L^2_\mu(\Gamma \backslash G)$, this gives norms that do not coincide with the 
family of norms defined in \eqref{sobnormdef}. However, the latter majorizes the former if we allow different 
indices and use \eqref{linftyfact} for as many derivatives along $\mathfrak{s}$ as needed.}
$$\Sob_d^V(f) := \sup_{\mathcal{D}} \|\mathcal{D} f \|$$
where the supremum is taken over $\mathcal{D} \in U(\mathfrak{g})$, the universal enveloping algebra
of $\mathfrak{s}$, which are monomials in the chosen basis of degree $\leq d$. Then
\begin{equation} \label{mc2}
  \langle s.v , w \rangle \leq c_k \varphi_0(s) . \Sob_{k}^V(v) \Sob_k^V(w) \ \ \ (v,w \in V)\end{equation}
for $s \in S$, any integer $k > (\dim K)/2$ and a constant $c_k \geq 1$.

 Suppose, now, that $(\pi,V)$ is $1/m$-tempered; in that case we have:
\begin{equation} \label{mc3}
\langle s .v , w \rangle \leq c'_k \varphi_0(s)^{1/m} \Sob_k^V(v) \Sob_k^V(w) \ \ \ (v,w \in V)\end{equation}
for $s \in S$.

\subsubsection{Bounds for matrix coefficients in presence of a spectral gap.}
As we have discussed (end of \S \ref{Rel:SGT}) it is possible for a representation to possess
a spectral gap, but not to be $1/m$-tempered for any $m \geq 1$. 
In order to quantify the decay of matrix coefficients
in such cases,  we use the function $\varphi_0^{\wk}$
of \eqref{weakdef}.

If $(\pi,V)$ possesses a spectral gap,
there exists $\rho > 0$, depending only on this gap, such that
the following  majorization holds:
\begin{equation} \label{mcweak}
\langle s.v , w \rangle \leq c_k (\varphi_0^{\wk}(s))^{\rho} .
\Sob_{k}^V(v) \Sob_k^V(w) \ \ \ s \in S,, (v,w \in V),\end{equation}
for $ k > (\dim K)/2$. 

By unitary decomposition 
 it suffices to verify \eqref{mcweak} for every irreducible ``constituent'' of $V$,
i.e. every irreducible unitary representation that weakly occurs within $V$.

Take an isogeny $\prod_{i=1}^{I} S_i \rightarrow S$ with each $S_i$ connected 
almost simple. 
Any irreducible unitary representation of $\prod_{i=1}^{I} S_i$
factors as a tensor product $\otimes_{i=1}^{I} \sigma_i$,
where each $\sigma_i$ is irreducible.

The assumption of spectral gap
implies (cf. \eqref{temp}) that there exists an integer $m$
so that every irreducible ``constituent'', upon
pullback to $\prod_{i=1}^I S_i$, is of the form
$\sigma = \otimes_{i=1}^{I} \sigma_i$, where
there exists $1 \leq i \leq I$ so that $\sigma_i^{\otimes m}$
is tempered as an $S_i$-representation.  Assume $i=1$,
the argument being similar in general. 

For $v, w$ in the space of $\sigma$, put
$w' := (1, s_2, \dots, s_I)^{-1} w$. Then:
\begin{multline} \langle (s_1, \dots, s_I) v, w\rangle =
\langle (s_1, 1, \dots, 1) v, (1, s_2, \dots, s_I)^{-1} w \rangle
\\ \ll \varphi_0(s_1)^{1/m} S_k(v) S_k(w') 
\leq \varphi_0^{\wk}((s_1, \dots, s_k))^{1/m} S_k(v) S_k(w).
\end{multline} The very last step follows, because the Sobolev norms here
are taken on $\sigma$ considered as an $S_1$-representation,
and in particular commute with the action of $S_2 \times \dots \times S_I$. 
Our assertion follows.

\subsection{Some estimates of volumes and matrix coefficients} \label{volume}
In the present section, $H, G$ are as in \S \ref{gi}. 

Let $H \subset S \subset G$, with $S$ connected.  We claim that for $T \geq 2$,
and for any $\rho > 0$, 
\begin{multline} \label{volest} 
 \vol \Sball(T) \sim v_S T^{A_S}(\log T)^{\ell_S}, \\
\frac{1}{\vol \Sball(T)^2} \int_{g,g' \in \Sball(T)} 
\varphi_0^{\wk}(g g'^{-1})^{\rho} \ll 
T^{-\zeta_S \rho}\end{multline}
for suitable $v_S, A_S, \zeta_S>0$, and $\ell_S \geq 0$ depending on $S$. Here $f(t)\sim g(t)$ if $\frac{f(t)}{g(t)}\to 1$ as $t\to\infty$. 

%Indeed, we have a decomposition $S = K . A^+ . K$ as in the previous section.
%Let $A^{+}_T = A^+ \cap Sball(T)$. Then there exists $c > 1$ depending on $S$ so that
%$$K . A^{+}_{c^{-1} T} . K \subset Sball(T) \subset K. A^{+}_{c T} . K,$$
%as follows from \eqref{distortion}.  From this, we conclude that it suffices to verify the
%claims of \eqref{volest} with $Sball(T)$ replaced by $K . A^{+}_T . K$. 
%To this end, we use the expression for Haar measure in $KAK$ coordinates (cf. \cite[Chapter 12]{Wallach}).

The estimate on the volume follows from \cite[Thm.\ 2.7]{GW}.
(Actually, we do not need this level of precision, but it is convenient
to have an asymptotic.)

Indeed, notations as in \eqref{isogeny},  it suffices to compute the measure
of the inverse image of $\Sball(T)$ under $\isogeny$,
i.e. the measure of
\begin{equation} \label{preimageball}\ker(\isogeny) \cdot \prod_{i} \{s \in S_i: \|\isogeny(s_i)\| \leq N^{-1} T^{1/I}\}.\end{equation} We can regard this as a certain union of sets, parameterized by
the finite kernel $\ker(\isogeny)$. The results stated in \cite{GW}
implies an asymptotic for each intersection of these sets, whence also an asymptotic for their union.  \footnote{The reference \cite{GW} assumes that $S$ connected. It also
 proves ``Theorem 2.7'' only in a special case, deferring the general case to another paper. We note however that Benoist and Oh \cite{Benoist-Oh} have given a lovely conceptual argument for such volume asymptotics.  }

Now, let us indicate the proof of the second assertion of \eqref{volest}.
 Take the isogeny $\isogeny: \prod_{i=1}^{I} S_i \rightarrow S$
from \eqref{isogeny}.   Because $\max(a^{\rho},b^{\rho}) \ll a^{\rho}+b^{\rho}$, and the Haar measure
on $S$ pulls back to the Haar measure on $\prod_{i} S_i$, it is enough to check,
for each $1 \leq i \leq I$,
the average value of $\varphi_0(s_i s_i'^{-1})$ over \eqref{preimageball}
is bounded by a negative power of $T$.  The set
\eqref{preimageball} is contained in a set of the form
\begin{equation} \label{preimageball2} \prod_{i} \{s \in S_i: \|\isogeny(s_i)\| \leq R\}.\end{equation}
where $R \asymp T^{1/I}$; moreover, the sets \eqref{preimageball}
and \eqref{preimageball2} have comparable volumes, by what we have just discussed. 

We conclude that it is enough to check that, for each $1 \leq i \leq I$,
the average of $\varphi_0(s s'^{-1})^{\rho}$
over the set $\{(s,s') \in S_i, \|\isogeny(s)\|, \|\isogeny(s')\| \leq 
R\}$, is bounded above by a negative power of $R$. 

That assertion, however, is an almost-immediate consequence of \eqref{clubestimate};
we need only enlarge $\{s \in S_i: \|\isogeny(s)\| \leq R\}$ so that it is bi-invariant under a maximal compact of $S_i$. This increases the volume by at most a constant factor,
by property (1) of \eqref{intermediateballs} and the above volume estimate. 
%By making use of the identity $\max(a,b) \ll a+b$,   
%we reduce to the case when $S \subset G$ is almost simple. 
% By H{\"o}lder's inequality, 
%the left-hand side is bounded by a suitable power of the integral of $\varphi_0(g g'^{-1})$.
%By \eqref{HCdef}, for any $K$-bi-invariant function $f(s)$, we have:
%$$\int \varphi_0(s_1 s_2^{-1})  f(s_1) \overline{f(s_2)} =
%\left| \int \varphi_0(s) f(s) ds\right|^2$$
%Take for $f$ the characteristic function of $K. Sball(T) . K$. 
%Noting that, since $\varphi_0 \in L^3(S)$, the majorization $\int \varphi_0(s) f(s) \leq \left(\int |\varphi_0(s)|^3 \right)^{1/3}
%\left(\int f(s) ds \right)^{2/3} \ll T^{2 A_S/3}$, the result follows. 

\subsection{The representation of $\mathbf{S}(\R)$ on an algebraic homogeneous space}

 Having established basic definitions concerning tempered and
$\frac{1}{m}$-tempered representations, we now show that certain naturally occurring examples, viz.,
representations of the real points of an algebraic group,
 on the real points of an (algebraic) homogeneous spaces, have such properties.

We say that, if $V$ is a vector space over a field $k$, and $\G \subset \GL(V)$
an algebraic group, then $v \in V$ is {\em stable} if it is not contained
entirely in the non-negative weight spaces for any $k$-split torus $\G_m \subset \G$.

For example, if $\G= \SL_2$, the irreducible (algebraic) representations of $\G$ are precisely the representations on the space of degree $m$ 
bivariate homogeneous polynomials,
for $m \geq 1$. According to the definition above, 
a degree $m$ polynomial is then stable exactly 
when there is no root in $\mathbb{P}^1(k)$ of multiplicity $\geq m/2$. 

We will use this notion when $k = \R$.  In this case,
the set of stable points is topologically open.

\begin{lem}\label{properaction} Suppose $S$ acts {\em properly} on a real manifold $M$.
Then, if $\nu$ is an $S$-invariant Radon measure on $M$, then
$L^2(M,\nu)$ is tempered as an $S$-representation.
\end{lem}

Recall {\em properness} means that the map $S \times M \rightarrow M \times M$,
defined via $(s,m) \mapsto (sm, m)$ is proper.
\proof
Indeed, it suffices to verify that matrix coefficients $s \mapsto
\langle s f_1, f_2 \rangle$, when $f_1, f_2 \in C_c(M)$, are
in $L^{2+\varepsilon}(S)$. By properness, they are compactly supported.
\qed

\begin{lem} \label{algaction}
Suppose we are given a finite-dimensional algebraic representation
$(\rho, V)$ of $\SL_2$.  Let $M \subset V_{\R}$ be a $\SL_2(\R)$-stable submanifold
so that every $m \in M$ is stable.

Then the action of $\SL_2(\R)$ on $M$ is proper.
\end{lem}

\proof
Let $A = \{a_t\}$ be a diagonal torus within $\SL_2(\R)$. It suffices to verify
that the action of $A$ on $M$ is proper, in view of the Cartan decomposition.

Split $V = \oplus V_n$, where $\{a_t\}$ acts on $V_n$ by the character
$e^{nt}$.

Choose compacta $K_1, K_2 \subset M$. Suppose, to the contrary,
that there exists a sequence $u_i \in K_1,v_i \in K_2$ and an unbounded sequence $a_{t_i} \in A$
so that $v_i = a_{t_i} u_i$.  Let $u$ be a limit of the $u_i$ and $v$ a limit of the $v_i$.  Then neither $u$ nor $v$ are stable, a contradiction.  \qed

\begin{lem}\label{temp-3}
Suppose we are given a finite-dimensional algebraic representation
$(\rho, V)$ of $\SL_2$. Let $\nu$ be an $\SL_2(\R)$-invariant Radon
measure on $V_{\R}$ so that $\nu$-almost all vectors $v \in V_{\R}$
are not fixed by $\SL_2(\R)$. Then almost all vectors in $V_{\R} \oplus V_{\R}$, 
w.r.t.\ $\nu \times \nu$, are stable.
\end{lem}
\proof
If $v_1,v_2 \in V_{\R}$ are not fixed by $\SL_2(\R)$, then $v_1 \oplus g v_2$
is stable in $V \oplus V$ for almost all $g \in \SL_2(\R)$.
This, together with Fubini's theorem, implies the stated claim.
\qed

\begin{lem} \label{stabilitytempered}
Let $\mathbf{S}$ be a semisimple $\R$-algebraic group without anisotropic factors, and let $\mathbf{L} \subset \mathbf{S}$
be an algebraic subgroup of strictly lower dimension. 
Let $S$ resp.\ $L$ be the real points of $\mathbf{S}$ 
resp.\ $\mathbf{L}$. Then the right action of $S$
on $L^2(L\backslash S)$ has a spectral gap.
\end{lem}

Although this case does not arise in our application,
one should interpret $L^2(L \backslash S)$ as the ``unitary induction from $L$ to $S$ of the trivial representation,'' in the case where there fails to be
an $S$-invariant measure on $L \backslash S$. 

\proof
It is easy to see that one may assume that $\mathbf{L}$
has no characters over $\R$. Otherwise, replace $\mathbf{L}$
by the intersection $\mathbf{L}'$ of the kernels of all such characters.
Then $L^2(L' \backslash S)$ weakly contains $L^2(L \backslash S)$,
and it suffices to prove the theorem replacing $\mathbf{L}$ by $\mathbf{L}'$. 
Similarly, we may assume that $\mathbf{L}$ is connected. 

Fix a morphism from $\SL_2$ to $\mathbf{S}$ that projects
nontrivially to any almost-simple component of $\mathbf{S}$.
Then the measure of points in $L\backslash S$ fixed by $\SL_2(\R)$ is zero.
%Indeed, on each connected component, the set of $\SL_2(\R)$-fixed points
%is a real-analytic subvariety, and so either has measure $0$
%or is the entire connected component. 
%The latter case is excluded,
%for the normal closure of the image of $\SL_2$ must contain
%the connected component of $\mathbf{S}(\R)$. 

There exists a representation $(\rho, V)$ of
$\mathbf{S}$ 
and a rational vector $v_L \in V$ so that the stabilizer of the line $\R.v_L$
is precisely $\mathbf{L}$.  
By assumption, $\mathbf{L}$ is connected and has no $\R$-characters, 
so $\mathbf{L}$ is the stabilizer of $v_L$. 
%Let $\mathbf{L}_1$ be the stabilizer of $v_L$ in $\mathbf{L}$,
%and $L_1$ its real points. Then $L/L_1$ is abelian; in particular,
%$L^2(L\backslash S)$ is weakly contained, as an $S$-representation, in $L^2(L_1\backslash S)$.

Let $\mathbf{Y}$ be the orbit of $v_L$ under $\mathbf{S}$;
it is an algebraic subvariety of $V$.  Let $Y = \mathbf{Y}(\R)$.
Then $L\backslash S$ is identified with an open subset of $Y$.
Let $\nu$ be the measure on $Y$ corresponding to an $S$-invariant
measure on $L\backslash S$, which exists since $L$ is unimodular.

Consider $V \oplus V$ as an $\SL_2$-representation. Consider
$Y \times Y \subset V_{\R} \times V_{\R}$. The
set of stable points $M \subset Y \times Y$ is an open subset
of full measure (Lemma \ref{temp-3}). 
So $L^2(Y \times Y, \nu\times\nu) = L^2(M, \nu \times \nu)$.

We apply Lemma \ref{properaction}--\ref{temp-3} to conclude that the $\SL_2(\R)$-action
on the tensor product $L^2(L\backslash S)^{\otimes 2}$ is tempered.  Therefore, the $S$-action
on $L^2(L\backslash S)$ has a spectral gap.
\qed

\subsection{Property $\tau$ and its corollaries}
\begin{prop}\label{proptau} Let $\mathbf{S}$ be
an absolutely almost simple, simply connected group over a number field $F$.
Let $v$ be a place of $F$. There is an integer $\uglytempered_{\mathbf{S}(F_v)}$, which depends only on the isomorphism class of $\mathbf{S}(F_v)$,
so that the representation $\mathbf{S}(F_v)$
on $L^2_0(\mathbf{S}(F) \backslash \mathbf{S}(\adele_F))$ is $\frac{1}{\uglytempered_{\mathbf{S}(F_v)}-1}$-tempered. 
\end{prop}
In fact, the integer $\uglytempered$ may be taken to depend only on $\dim(\mathbf{S})$, but we do not need this.  We say ``$\frac{1}{\uglytempered-1}$--tempered'', rather than 
``$\frac{1}{\uglytempered}$--tempered'', to absorb some annoying factors of $\varepsilon$
at later stages in our proof (examples of such factors can be seen in \eqref{mcunipotent}).  

{\em Explicit} (and rather ``good'') values for $\uglytempered$ could be derived from the work of H. Oh
and Gorodnik-Maucourant-Oh: \cite{Oh-Duke} and \cite[Cor 3.26]{Oh}; these results
give explicit (and ``good'') rates of decay for matrix coefficients. Indeed, in the nonarchimedean case when the rank is $\geq 2$,
\cite{Oh-Duke} is apparently the only place in the literature where the existence of $\uglytempered$ is established.
	%
	%In the case of rank $\geq 2$, very explicit bounds for matrix coefficients were derived in \cite{Oh-Duke},
	%which allows one to give a fairly precise value for $\uglytempered$; \cite[Cor 3.26]{Oh}
	%ad
	%See also \cite[Cor 3.26]{Oh} for a quantification of the above Proposition. 
	%See also \cite{Oh}; while as stated \cite[Cor. 3.26]{Oh}
	%depends on the $\Q$-form, the proof gives
	%in fact an explicit value for $\uglytempered_{\mathbf{S}(F_v)}$. 

\proof
We confine ourself to the case of $v$ archimedean, which is the only case we use in the present paper (cf. remarks above).

This follows from the solution to property $\tau$ which was completed by L. Clozel \cite{LC};
this solution uses a variety of ingredients: the trace formula, prior ideas of Burger and Sarnak \cite{BS}, 
the ideas of Kazhdan on property $T$, and work of A. Selberg for groups of type $A_1$. 

  However, it does not follow from the main statement
of \cite{LC}, but rather from the proof. Indeed, the statements
of \cite{LC} {\em a priori} depend on the $F$-form $\G$, and not just
the group $G$; however, the proofs give stronger statements.

 If the real rank $\mathbf{S}(F_v)$
exceeds $1$, the assertion is proved explicitly in \cite[Theorem 2.4]{Cowling}.  If the real rank of $S := \mathbf{S}(F_v)$ equals $1$, it is proven in \cite{LC} that there exists a homomorphism $H \rightarrow S$,
where $H$ is a real Lie group locally isomorphic to
 $\mathrm{SU}(n,1)$, some $n \geq 1$, with the property
that the pull-back of $L^2_0(\mathbf{S}(F) \backslash \mathbf{S}(\adele_F))$
is isolated from the identity as an $H$-representation. 
Moreover, this notion of ``isolated'' is absolute, i.e.
independent of $\mathbf{S}$. 

However, there exist only finitely many conjugacy classes
of homomorphisms (of real Lie groups) from $\mathrm{SU}(n,1)$
to $S$; see Lemma \ref{finite-embeddings}.  It follows from that statement (cf. \eqref{temp}) that
the $S$-action on $L_0^2$ has a  spectral gap 
depending only the isomorphism class of the real Lie group $S$, 
and therefore there exists an integer $\uglytempered_S$ as asserted.

%Indeed, if $\S(F_v)$ has $F_v$-rank $\geq 2$, this assertion follows
%from the explicit bounds of \cite{Oh}.
%
%
%
%The necessary computations to apply it were carried out
%
%We need to show that, for every infinite place $w$ of $K_i$, the representation
%of $\mathbf{S}_i(K_{i,w})$ on $L_0^2(\mathbf{S}_i(K_i) \backslash \mathbf{S}_i(\adele_{K_i}))^{\otimes m}$ is tempered.  If the real rank of the group $\mathbf{S}_i(K_{i,w})$ is $\geq 2$, this follows from [[REF]]. We are left to consider the case
%where this group has real rank $1$. It is observed (without sufficient justification?) in \cite{Oh} that matrix coefficients on
%$L_0^2(\mathbf{S}_i(K_i) \backslash \mathbf{S}_i(\adele_{K_i}))^{\otimes m}$
%are majorized by an $L^p$-function ``$\xi_w$'' on $\mathbf{S}_i(K_{i,w})$
%that depends only on the isomorphism class of $\mathbf{S}_i(K_{i,w})$.
%Since there are only finitely many isomorphism classes possible for $\mathbf{S}_i(K_{i,w})$, depending only on $\dim G$, we deduce the conclusion.
\qed

\begin{lem} \label{automorphic}
Notation being as in \S \ref{gi},
let $H \subset S \subset G$; let $x S$ be a closed connected orbit of $S$ on $\Gamma \backslash G$.
Let $\nu$ be the $S$-invariant probability measure on $xS$.

There exists $p_{G} \geq 1$, which can be taken to depend only on $\dim G$, so
that $L^2_0(\nu)$ is $\frac{1}{\tempered-1}$-tempered as an $S$-representation.
\end{lem}

Note that this means that $L^2_0(\nu)$ is uniformly isolated (as one varies the closed orbit $xS$) from the identity as an $S$-representation
or as an $H$-representation.

\proof
Replace $S$ by $x S^0 x^{-1}$ and $\Gamma x S$ by $\Gamma x S^0 x^{-1}$.  It is easy
to see that it suffices to prove the statement in this setting.

By the Borel-Wang density theorem,
\cite[Chapter II, Corollary 4.4]{Margulis}
there exists a semisimple connected $\Q$-group $\mathbf{S}$ so that $S$ is a finite index subgroup
of $\mathbf{S}(\R)$. (Take the Zariski closure of $\Gamma \cap S$).

Set $\Gamma_S = \Gamma \cap \mathbf{S}(\R)$. It is a congruence subgroup of $\mathbf{S}(\Q)$.
It suffices to verify that the representation of $\mathbf{S}(\R)$ on $L^2_0(\Gamma_S \backslash \mathbf{S}(\R))^{\otimes m}$ is tempered.  (Here $L^2_0$ denotes orthogonal complement
of {\em locally} constant functions.)

$\Gamma_S$ being a congruence subgroup, it is enough to check that the representation
of $\mathbf{S}(\R)$ on $L^2_0(\mathbf{S}(\Q) \backslash \mathbf{S}(\adele))^{\otimes m}$ is tempered.
In this context, we understand $L_0^2$ as meaning the orthogonal complement of $\mathbf{S}(\R)^{0}$-invariant functions. 

There exists number fields $K_i$,
and absolutely almost simple, simply connected groups $\mathbf{S}_i$
over $K_i$, together with an isogeny $\varphi:\prod_{i=1}^{k} \mathrm{Res}_{K_i/\Q} \mathbf{S}_i \rightarrow \mathbf{S}$.
It is then enough to verify that the representation
of $\mathbf{S}_i(K_i \otimes \R)$ on $L^2_0(\mathbf{S}_i(K_i) \backslash \mathbf{S}_i(\adele_{K_i}))^{\otimes m}$
is tempered. Indeed, this guarantees that the representation of the identity component
of $\prod_{i} \mathbf{S}_i(K_i \otimes \R)$ on $L^2_0(\mathbf{S}(\Q) \backslash \mathbf{S}(\adele))^{\otimes m}$ is tempered. But the former group is a finite covering
of the identity component of $\mathbf{S}(\R)$.  We apply Proposition \ref{proptau} to conclude, taking for $\tempered$ the maximum of all the integers
$\uglytempered_{\mathbf{S}(\R)}$ associated to possible $\mathbf{S}$s. 
 \qed

\subsection{Proof of Proposition \ref{sectionslice}}

Recall that for a subgroup $S$ of $G$ we denote the normalizer
of the Lie algebra
of $S$ by $\tilde{S}$. 

It suffices to show that the action of $S$ on the orthogonal complement
to constant functions in $L^2(\tilde{S}_x\backslash\tilde{S})$
has a spectral gap; or, that the action of $\tilde{S}$
on the same space has a spectral gap.
Set $\mathbf{L}$ to be the Zariski closure of $\tilde{S}_x$ inside $G$;
then
let $L = \mathbf{L}(\R) \subset \tilde{S}$.
If $L=\tilde{S}$,
the $\tilde{S}$-orbit of $\Gamma g \in \Gamma \backslash G$ is closed,
and we apply Proposition~\ref{automorphic}.
Otherwise $L$ satisfies $\dim \, L < \dim \, S$.\footnote{ Note that the group $\mathbf{L}$
conjugated by an element of $G$ representing $x$
is defined over $\Q$ and is $\Q$-anisotropic since its integer points are Zariski-dense. In particular, its real points are unimodular.} 

The representation $L^2(\tilde{S}_x \backslash \tilde{S})$ may be regarded
as the induction from $L$ to $\tilde{S}$, of $L^2(\tilde{S}_x\backslash L)$. 
Now, if $G_1 \subset G_2$ are locally compact groups,
$V$ a unitary representation of $G_1$, and $\nu$ a probability measure on $G_2$,
then the operator norm of convolution with $\nu$ on $\mathrm{Ind}_{G_1}^{G_2} V$
is bounded by the corresponding norm on $L^2(G_1\backslash G_2)$. In particular,
if $L^2(G_1\backslash G_2)$ has a spectral gap, so also does $\mathrm{Ind}_{G_1}^{G_2} V$. 
We invoke Lemma \ref{stabilitytempered} to conclude. 
\qed

%%%%%%%%%%%%%%%%%%%%%%%%%%%%%%%%%%%%%%%%%%%%%%%%%%%%%%%%%%%%%%%%%%%%%%%%%%%%%

\section{Effective generation of Lie algebras.} \label{effectivegen}

Given a subset $T = \{t_1, \dots \}$ of a Lie algebra $\mathfrak{g}$, recall that
we denote by  $T^{(k)}$ the set of all possible iterated Lie
brackets of the $t_i$s of depth $\leq k$. 
In this section, we shall prove the following (``Proposition E''):

\begin{prop}\label{any-delta}
There exists an integer $k$ and some $c>0$ depending only on $\g$
with the following property. For any orthonormal subset $T = \{t_1,
\dots \}$ of $\mathfrak{g}$, and $0 < \delta <  1$ there exists a
subalgebra $\mathfrak{w} \subset \mathfrak{g}$ with orthonormal
basis $w_1, w_2, \dots$ satisfying:
\begin{enumerate}
\item
For each $w_i$, there exists a linear combination $w_i' = \sum_{t
\in T^{(k)}} c_t t$, with $c_t \in \R$ satisfying $|c_t| \leq
\delta^{-k}$, so that $\|w_i - w_i'\| \ll \delta^c$.
\item Each $t \in T$ is within $\delta$ of $\mathfrak{w}$ (i.e. $\min_{w \in \mathfrak{w}} \|w-t\| \leq \delta$).
\end{enumerate}
  If the linear span $\langle T \rangle$ of $T$ contains a subalgebra $\mathfrak{h}$,
  then there exists $k$ and $c$ depending on the pair $(\mathfrak{h} , \mathfrak{g})$ so that
  (1) and (2) hold and, in addition, $\mathfrak{h} \subset \mathfrak{w}$.
\end{prop}

Let us discuss the motivation for this. Given elements $e_1, \dots, e_r \in \mathfrak{g}$, one can speak of
the ``subalgebra $\mathfrak{w}$ spanned by $\mathfrak{g}$'', but
this notion is not robust under perturbation: this subalgebra might
drop drastically under a small modification of the $e_i$. This is
related to the fact that the $e_i$ might span $\mathfrak{w}$ very
{\em inefficiently.} In explicit terms, there may exist $w \in
\mathfrak{w}$ of small norm, which cannot be written as a linear
combination of Lie brackets of the $e_i$s with small coefficients.
Proposition \ref{any-delta} constructs an effective replacement:
We construct a subalgebra
$\mathfrak{w}$ which ``almost'' contains $T$ and so that elements of
$T^{(k)}$, for suitable small\footnote{Here the word `small' refers
to the fact that it only depends on $\g$.} $k$, ``almost span $\mathfrak{w}$ in
an efficient fashion.''

The proof, given in \S \ref{ADproof} will require some setup on Euclidean spaces (\S \ref{eucl})
and a recollection of the Lojasiewicz inequality (\S \ref{loj}).

\myparagraph \label{eucl} Suppose $V$ is a Euclidean vector space. We shall say
that two subspaces $U_1, U_2 \subset V$ are $\epsilon$-close
if there are
orthonormal bases $u_1, \ldots, u_r$ and $v_1, \ldots, v_r$ for
$U_1, U_2$ so that $\|u_i - v_i\| \leq \epsilon$. We say that $U_1$
is $\epsilon$-almost contained in $U_2$ if there exists $U_2'
\subset U_2$ so that $U \stackrel{\epsilon}{\sim} U_2'$.

Suppose that there exists
an orthonormal basis $u_1,\ldots,u_r$ of $U_1$
so that $\mathrm{dist}(u_i, U_2) \leq \epsilon$.
Then there exists a constant $C_V$, depending only on $\dim V$, so that
 $U_1$ is $C_V\epsilon$-almost contained in $U_2$.

Suppose $V,W$ are two Euclidean vector spaces and $f: V \rightarrow
W$ a linear map between them. 
We recall the singular value decomposition: let $w_1,\ldots, w_n$
be an orthonormal basis of $W$ consisting of eigenvectors of
$ff^t$ with decreasing eigenvalues where $f^t:W\to V$ denotes the 
transpose map. Then the eigenvalues $(ff^t
w_i,w_i)=\|f^t w_i\|^2$ are nonnegative, and we define
$v_i=\frac{1}{\|f^t w_i\|}f^t w_i$ whenever this is defined, say for
$i\leq k$, and extend it to an orthonormal basis of $V$. With
this choice we have $(f v_i,w_i)=\frac{1}{\|f^t
w_i\|}(ff^t w_i,w_i)=\|f^t w_i\|=\|fv_i\|=\sigma_i$ for $i\leq k$
and $fv_i=0$ otherwise.

Therefore, there are orthonormal bases $v_1, v_2,
\ldots,v_m$ for $V$  respectively  $w_1, w_2, \ldots,w_n$ for $W$ so that
$f(v_i) = \sigma_i w_i$ for some $\sigma_i\geq 0$; here, by definition,
$w_i = 0$ for $i > n$.
Let $W[\delta]$ be the space spanned by those $w_i$,
$i\leq k$, for which $|\sigma_i| \geq \delta$.

 Then, for any $v \in V$:
\begin{equation} \label{first}
 \mbox{the element $f(v)$ is within $\delta \|v\|$ of $W[\delta]$.}
\end{equation}
To see this, let $v=\sum_{i}t_i v_i$ and write $f(v)=\sum_{i:
|\sigma_i|\geq\delta}\sigma_i t_i w_i+\sum_{i:
|\sigma_i|<\delta}\sigma_i t_i w_i$. Then the second sum has norm
$\leq \delta \|v\|$, as claimed.

We will be using this construction for various choices of the
function $f$ and the domain $V$ to obtain various subspaces of
$W=\mathfrak{g}$, and the following remark will help to
compare their dimension. Let $V_1 \subset V$ be a subspace, then
\begin{equation}\label{var-comment}
 \|f( v)\| \geq \delta \|v\|\mbox{ for all }v \in V_1\mbox{ implies }\dim W[\delta] \geq \dim
V_1.
\end{equation}
Suppose, to the contrary, that $\dim W[\delta] < \dim
V_1$. Then there would exist $v \in V_1$ perpendicular to all the
$v_i$ with $|\sigma_i| \geq \delta$. This gives a contradiction to
\eqref{first} since $f(v)$ is perpendicular to $W[\delta]$ and has
length $\|f(v)\|\geq \delta\|v\|$ by assumption.

\subsection{Lojasiewicz inequality}
The {\em Lojasiewicz inequality} states that, for $U \subset \R^n$ open,
 $K \subset U$ a compact set,  and $f: U \rightarrow \R$ a real analytic function
with zero-set $Z_f$, there exist constants $c_1, c_2 > 0$ so that:
\begin{equation} \label{loj}|f(x)| \geq c_1 \dist(x, Z_f)^{c_2}, x \in K \end{equation}
Here $\dist$ refers to the Euclidean distance on $\R^n$.
For this, see \cite[Theorem 4.1]{Malgrange}.

\subsection{Proof of Prop. \ref{any-delta}} \label{ADproof}
 We define the map $f_m:\R^{T^{(m)}} \rightarrow \mathfrak{g}$
by sending $(v_t)_{t\in T^{(m)}}$ to $\sum_{t\in T^{(m)}}v_t t$.
Considering $\R^{T^{(m)}}$ as an Euclidean space in the usual way,
we obtain the subspace $W_m[\delta] \subset W=\mathfrak{g}$ by the
above construction.

From the definition it follows that the spaces $W_m[\delta]$
increase when $\delta$ decreases and $m$ is held fixed. Moreover,
their {\em dimensions} increase when $m$ increases and $\delta$ is
held fixed, which follows from \eqref{var-comment}. Finally, as
follows from \eqref{first}: if $m_1 \leq m_2$, then there is an
orthonormal basis for $W_{m_1}[\delta]$, all of whose elements are
within $ \epsilon \delta^{-1}$ of $W_{m_2}[\epsilon]$, i.e.\
$W_{m_1}[\delta]$ is ${\ll\epsilon \delta^{-1}}$-almost contained in
$W_{m_2}[\epsilon]$. In particular, every $t \in T$ is within $
\epsilon$ of $W_m[\epsilon]$ for each $m \geq 1$.

We claim there is a constant $\xappa$
depending only on $\dim(G)$, so
that, for any $\delta \in (0,1)$, we may find $\delta_1 \in [\delta^\xappa,
\delta]$ and $m \leq \xappa$ so that
$$\dim \, W_m[\delta_1] =  \dim \, W_{2m}[\delta_1^3]$$
In fact, this follows from the above remarks: set initially
$\delta_1=\delta$ and $m=1$, and consider the dimensions $\dim \,
W_m[\delta_1] \leq  \dim \, W_{2m}[\delta_1^3]$. If equality holds,
we are done. Otherwise set $\delta_1'=\delta_1^3$ and $m'=2m$. This
way we have increased the dimension of $W_{m'}[\delta_1']$. Since
this can only happen $\dim(G)$ often if we repeat the process, the
claim follows.

%The above also implies $W_m[\delta_1] \stackrel{\ll\delta_1}{\sim}
%W_{2m}[\delta_1^4]$.

Applying \eqref{first} once more we get that $W := W_m[\delta_1]$
has an orthonormal basis $w_1, \dots, w_r$ so that $[w_i, w_j]$ is
at distance $\ll \epsilon \delta_1^{-2}$ from $W_{2m}[\epsilon]$ for
any $\epsilon > 0$.
 In particular,
$[w_i, w_j]$ is at distance $\ll \delta$ from $W$.

It follows that there exists a Lie subalgebra $\mathfrak{w} \subset
\mathfrak{g}$ satisfying $\mathfrak{w} \stackrel{\ll\delta^c}{\sim}
W$, where $c$ depends only on $\mathfrak{g}$. To see this set $r = \dim W$.  Consider
the open subset $U \subset \mathfrak{g}^r$ consisting of vectors $(X_1, \dots, X_r)$
which are linearly independent. The inner product on $\mathfrak{g}$
induces one on all exterior powers. The real-analytic function
\begin{equation} \label{Ffunc} F := \sum_{i,j}
\| X_1 \wedge \dots \wedge X_r \wedge [X_i, X_j] \|^2\end{equation}
vanishes precisely when the span $\langle X_1, \dots, X_r\rangle$ is a subalgebra.
Clearly our assumptions imply that $|F(w_1, \dots, w_r)| \ll \delta^2$; applying
\eqref{loj} to a suitable compact inside $U$ yields the result.

% Reducing $\delta$ and increasing $m$ as
%necessary, we have completed the proof of the first two
%assertions.

For the final assertion, we apply the Lojasiewicz inequality to a
slightly different variety. In that case the subspace $W$ has the
following additional property: $\mathfrak{h}$ is $\ll\delta$-almost
contained in $W$. This means that $W$ almost belongs to the closed
subvariety of the Grassmannian defined by subspaces $\mathfrak{w}
\subset \mathfrak{g}$ which are Lie subalgebras and which contain
$\mathfrak{h}$. Invoking \eqref{loj} for a suitably modified version of
\eqref{Ffunc} concludes the proof.
\qed

%%%%%%%%%%%%%%%%%%%%%%%%%%%%%%%%%%%%%%%%%%%%%%%%%%%%%%%%%%%%%%%%%%%%%%%%%%%

 \section{Almost invariance of measures.}\label{Appendix-A}

 We begin with some reminders to the reader about our notations. 
These points were discussed in \S \ref{notation}, but they are particularly pertinent here. 
Firstly, we always assume that any Sobolev norm $\Sob_d$ we use satisfies $d \geq \ref{linftyconst} + 1$, 
i.e., involves enough derivatives that \eqref{simple-Sobolev-estimate} is valid.  Secondly, if $X \in \mathfrak{g}$,
the notions of ``almost invariant under $\exp(X)$'' and
``almost invariant under $X$'' do not coincide; the latter is stronger,
as it entails almost invariance under $\exp(t X), 0 \leq  t\leq 2$. 
Finally, implicit constants in the notation $\ll$ are permitted
to depend on $d$, if a Sobolev norm $\Sob_d$ is implicitly present. 

The primary purpose of this section is to prove the following (``Proposition D''
of \S \ref{proof}).

\begin{prop} \label{almostsubalgebra}
Fix a Sobolev norm $\Sob_d$; all notions 
of almost invariance are taken with respect to this.  
Suppose that $\mu$ is $\epsilon$-almost invariant under a subgroup $S$,
and also under $Z \in \exp(\mathfrak{r})$, where $\|Z\|=1$. 

Then there is a constant $\ref{almostinvconstant}$\index{xappa23@$\ref{almostinvconstant}$,
generating a bigger subgroup with almost invariance}
so that  $\mu$ is
also $\constc(d) \epsilon^{\ref{almostinvconstant}}$- almost
invariant under some subgroup $S_*$ with $\dim(S_*) > \dim(S)$.  If
$H \subseteq S$, we may also assume that $H \subseteq S_*$.
\end{prop}

The proof could be much simplified, in our present setting,
by using the assumption
that there are finitely many intermediate subgroups between $H$ and $G$.
However, we shall make use of our (quite general) results from \S \ref{effectivegen} on effective generation
of Lie algebras.

\subsection{Stability properties of almost invariance}
The notion of ``$\mu$ is almost invariant under $Z$,'' for $Z \in \mathfrak{g}$,
is almost stable under linear combinations, under commutators, and under passing from $Z$ to a nearby element $Z'$:

\begin{lem}\label{stabilityofalmost}
let $k \geq 1, 0 < \delta <1$. Let 
$T \subset \mathfrak{g}$ consist of unit vectors, i.e. $\|X\|=1$ for $X \in T$.
Let $Z \in \mathfrak{g}$ be a unit vector so that $\|Z - \sum_{t \in T^{(k)}} c_t t \| \leq \delta$. 

If $\mu$ is $\epsilon$-almost invariant under every $X \in T$
w.r.t.\ $\Sob_d$,
it is also $\ll  (\epsilon^{\consta(k)\label{stabilityofinvarianceconstant}}\sum_{t} |c_t| +\delta)$-almost invariant under $Z$\index{xappa4@$\ref{stabilityofinvarianceconstant}$, stability of almost invariance, Lemma \ref{stabilityofalmost}}
w.r.t.\ $\Sob_d$.
\end{lem}

\proof The proof is a tedious elaboration of more or less obvious properties.

Let us observe
\begin{itemize}
\item 
If $\mu$ is $\epsilon$-almost invariant under $Z\in\mathfrak{g}$ with $\|Z\|\leq1$,  
and $W\in\mathfrak{g}$ satisfies $\|W - Z\| \leq \delta$, then
$\mu$ is $\ll \max(\epsilon, \delta)$-almost invariant under $W$.
\end{itemize}
This is an easy consequence of
\eqref{simple-Sobolev-estimate}, 
\eqref{sobolevdistort} and the fact that $\exp(tZ)$ and $\exp(tW)$
are at distance $\ll \delta$ from each other, for $t \in [0,2]$. 
Next, 
\begin{itemize}
\item 
 If $\mu$ is $\epsilon$-almost invariant under $\exp(Z)\in\mathfrak{g}$,
and $1 \leq c \leq 2 \|Z\|^{-1}$, 
then $\mu$ is $\ll (c \epsilon+ \|Z\|)$-almost invariant under $cZ$. 
\end{itemize}
Indeed, it is evident that $\mu$ is $\ll n \epsilon$-almost invariant
under $\exp(n Z)$, when $n$ is integral and satisfies $n \leq 10 \|Z\|^{-1}$.  
This implies
the stated conclusion.

\begin{itemize}
\item 
 If $\mu$ is $\epsilon$-almost invariant under $Z \in\mathfrak{g}$,
and $1 \leq c \leq \|Z\|^{-1}$, 
then $\mu$ is $\ll c \epsilon$-almost invariant under $cZ$. 
\end{itemize}
This follows in a similar fashion to the previous statement. 

\begin{itemize}
\item 
Suppose $\mu$ is $\epsilon$-almost invariant under $Z_1, Z_2 \in
\mathfrak{g}$ with $\|Z_i\|\leq 1$. Take  $\alpha_1, \alpha_2$
so that $|\alpha_i| \leq 1$ for $i=1,2$. Then $\mu$ is
$\ll \epsilon^{1/2}$-almost invariant under $Z:= \alpha_1 Z_1 +\alpha_2 Z_2$. 
\end{itemize}
We observe that, for $n \geq 1$ integral,
$$\exp(\alpha_1 Z_1/n) \exp(\alpha_2
Z_2/n) = \exp(\frac{Z}{n}) \exp(W_1),  
$$where $\|W_1\| \ll 1/n^2$. 
As a consequence, it follows that $\mu$ is
$\ll (n^{-2} + \epsilon)$-almost invariant under $\exp(\frac{Z}{n})$. 
It follows by a prior assertion that $\mu$
is $\ll (n^{-1} + n \epsilon) $-almost invariant
under $Z$.  Take $n = \epsilon^{-1/2}$ to conclude. 

\begin{itemize}
\item  Suppose $\mu$ is $\epsilon$-almost invariant under
$Z_1, Z_2$ with $\|Z_i\| \leq 1$. Then 
 $\mu$ is $\ll \epsilon^{1/2}$-almost invariant under $[Z_1, Z_2]$.
\end{itemize}
Indeed, 
\begin{multline*}
 \exp(Z_1/n) \exp(Z_2/n) \exp(Z_1/n)^{-1} \exp(Z_2/n)^{-1} = \\
  \exp([Z_1, Z_2]/n^2) \exp(W_2),
\end{multline*}
where $\|W_2\| \ll n^{-3}$. 
 In view of that, 
$\mu$ is $\ll (\epsilon + n^{-3})$-almost invariant under
$\exp([Z_1, Z_2]/n)$.  (Observe that our choice of the norm on $\g$
is so that $\|[Z_1, Z_2]\| \leq 1$.) By a previously noted fact,
$\mu$ is $\ll (n \epsilon+n^{-2})$-almost
invariant under $[Z_1, Z_2]$.  Take $n = \epsilon^{-1/2}$ to conclude.

These remarks in hand, 
the assertion of the Lemma now follows easily. Regard
$k$ as fixed. We first note that
$\mu$ is $\epsilon^{\xappa(k)}$-almost invariant under each $t \in T^{(k)}$. 
(Observe also that $\|t\| \leq 1$ for every such $t$). 
From this, we deduce that $\mu$ is also $\epsilon^{\xappa'(k)}$-almost invariant 
under $C^{-1} \sum_t c_t t$ where $C=\sum_t|c_t|$.
Thus $\mu$ is
$\ll (C \epsilon^{\ref{stabilityofinvarianceconstant}(k)} + \delta)$-almost invariant
under $Z$, as required. 
\qed

\begin{lem} \label{one}
 There is a constant
$\consta \label{firstlemconstant}$%
\index{xappa41@$\ref{firstlemconstant}$, almost invariance, Lemma \ref{one}} 
so that:

Let $Q = S^{0}$ be the connected component of any
intermediate subgroup $H \subset S \subset G$; let $\mu$
be $\epsilon$-almost invariant under $Q$ w.r.t.\ $\Sob_d$.  
Then $$\left| \mu^{q}(f) -  \mu(f)  \right| 
\ll
 \epsilon \|q\|^{d \ref{firstlemconstant}} \Sob_d(f).$$
\end{lem}

\proof Because there are only finitely many possibilities
for $Q$, it suffices to prove it for each possible $Q$ individually.
(This marks a point where we make usage of $C_\g(\h)=\{0\}$).
\label{UsageA}
We claim that there are constants
$c_1, c_2$ depending only on $Q$ so that, for every $r \geq 2$:
\begin{multline} \label{generation} 
\mbox{Every $q \in Q$ with $\|q\|\leq r$ may be expressed as the product}\\
\mbox{of $\leq c_1  + c_2 \log r$ elements with $\|q\|\leq 2$} 
\end{multline} 
This
follows from the structure theory of semisimple groups:
 $Q$ is the connected component of the real
points of a semisimple algebraic group. So there exists a compact
subgroup $K_Q \subset Q$ and a Cartan subalgebra $\mathfrak{a}_Q
\subset \mathfrak{q}$ so that $Q = K_Q  \exp(\mathfrak{a}_Q) K_Q$.
We are reduced to verifying \eqref{generation} for elements of
$\exp(\mathfrak{a}_Q)$, which is elementary.

 Let $q\in
Q$. By \eqref{generation} there exist $q_1,\ldots,q_n\in Q$ 
with $q=q_1\cdots q_n, \|q_i\| \leq 2$ and $n\leq c_1+c_2\log \|q\|$. Let
$\bar{q}_j=q_{j+1}\cdots q_n$ for all $j=1,\ldots,n-1$,
$\bar{q}_n=1$. We have
\begin{multline}
 |\mu^q(f)-\mu(f)|=|\sum_{j=1}^n
 \mu^{q_k}(\bar{q}_k.f)-\mu(\bar{q}_k.f)|\leq \epsilon\sum_{j=1}^n
 \Sob_d(\bar{q}_k.f)\\
% \epsilon \Sob_d(f) \sum_{j=1}^n (N^{j-1}2^j)^{\ref{sdconstant} d}
\ll_d \epsilon (1 + \log \|q\|) \Sob_d(f)(2N)^{\ref{sdconstant} n d}
\ll
 \epsilon \|q\|^{d \ref{firstlemconstant}} \Sob_d(f).$$
\end{multline}
by our assumption, \eqref{sobolevdistort}, and \eqref{distortion}.
As in \eqref{distortion} the integer $N$ gives the dimension of the general linear group 
(with respect to which we defined $\|q\|$).
\qed

\subsection{Proof of Proposition \ref{almostsubalgebra}}

 It should be noted that this proof could be very considerably simplified by using
the fact that there are finitely many intermediate subalgebras between $\mathfrak{h}$ and $\mathfrak{g}$.  
However, the proof we indicate will work in a more general setup.\footnote{More precisely, the proof written establishes 
that, under the assumptions
of the Proposition \ref{almostsubalgebra}, the measure $\mu$ is almost invariant
under an intermediate Lie subalgebra $\mathfrak{h} \subset \mathfrak{q} \subset \mathfrak{g}$. 
This statement makes no use of the finiteness of intermediate subgroups.  However, in passing from the Lie algebra 
$\mathfrak{q}$ to the group, we invoke Lemma \ref{one}, which does use the finiteness assumption.}

We observe we are free to assume that $\epsilon$ is {\em sufficiently small},
i.e.\ $\epsilon \leq \epsilon_0(G,H)$, in the statement of the Proposition.
If not, the statement is obvious. 

Assumptions as in the statement of the Proposition \ref{almostsubalgebra}. 
Define $T$ to be a finite subset of $\mathfrak{g}$ obtained
by adjoining $Z$ to an arbitrary, orthonormal basis for $\mathfrak{s}$.
Let $k,c$ be as in Proposition \ref{any-delta}.
 We apply Proposition \ref{any-delta}
in the form indicated in its last sentence,
with $\delta = \epsilon^{\alpha}$, where $\alpha = 
\frac{\ref{stabilityofinvarianceconstant}}{k+c}$. 
It produces a subalgebra $\mathfrak{w} \supset \mathfrak{h}$, with basis $w_i$. 

By part (2) of Proposition \ref{any-delta}, $T$ is ``almost contained in $\mathfrak{w}$,''
this implies, in particular, that $\dim(\mathfrak{w}) > \dim (\mathfrak{s})$
so long as $\epsilon$ is sufficiently small.  

By part (1) of Proposition \ref{any-delta} and Lemma \ref{stabilityofalmost}, 
$\mu$ is $\ll \epsilon^{c \alpha}$-almost invariant under each $w_i$. 
From this, one deduces easily that $\mu$ is also $\epsilon^*$-almost invariant
under $W$, the connected Lie group with Lie algebra $\mathfrak{w}$. 
By our construction, $W \supset H$ and $\dim(W) > \dim(S)$. 
\qed

%%%%%%%%%%%%%%%%%%%%%%%%%%%%%%%%%%%%%%%%%%%%%%%%%%%%%%%%%%%%%%%%%%%%%%%%%%%%%%%%%%%%%%%%%%

\section{Effective ergodic theorem.} \label{firststep}

 In this section, we establish an effective version of the Birkhoff ergodic theorem, i.e. we prove  ``Proposition A''
from \S \ref{proof}.  Prior to doing so, we give a precise quantification of the notion of ``generic'' we shall need. 

Throughout, $\mu$ will
denote an $H$-invariant $H$-ergodic
 measure
on $X = \Gamma \backslash G$, so that $L_0^2(\mu)$ is $\frac{1}{\tempered-1}$
tempered as an $H$-representation. This applies, in particular,
 to the measures of central interest in this paper: $\mu$ as in the statement of
  Theorem \ref{thm:main}; or also to the $S$-invariant
probability measures on connected closed $S$-orbits,
for $H \subset S \subset G$.  These comments, and the definition of $p_G$, follow from Lemma \ref{automorphic}, together with the fact that the restriction
of a tempered $S$-representation to $H$ remains tempered.

\subsection{Generic points.} \label{subsec:generic}
Let $T \geq 1$. Let $M = 20 (1 +\tempered)$,
where $\tempered$ was defined in Lemma \ref{automorphic}. \footnote{In practice, 
it needs just to be a sufficiently large fixed number in our arguments; we have used
the notation $M$ to avoid distracting the reader with its specific
value, which is irrelevant.}
 We define $D_T(f)$ as the discrepancy between a
horocycle average of $f$ over a big stretch of the orbit and its
integral:
\begin{equation} \label{fTdef}
 D_T(f) (x) := \frac{1}{(T+1)^M-T^M}
\int_{T^M}^{(T+1)^M} f(x u(t)) \operatorname{d}\!t - \int f
\operatorname{d}\!\mu.
\end{equation} 
Clearly, $D_T$ depends on the choice of the unipotent subgroup.
Since we regard it as fixed, we suppress that dependence.

We say a point $x \in X$ is {\em $T_0$-generic} w.r.t.\ the Sobolev norm $\mathcal{S}$ if, for all integers
 $n \geq T_0$ and all $f \in C^{\infty}(X)$ %, and for each $u \in \Unipotents$,
 we have the bound
 \begin{equation} \label{genericdef} | D_n(f)(x)| \leq n^{-1} \Sob(f).\end{equation}
 We say a point $x \in X$ is {\em $[T_0,T_1]$-generic} w.r.t.\ $\mathcal{S}$ if the bound \eqref{genericdef}
 holds  for all integers\index{generic@$T_0$-generic, $[T_0,T_1]$-generic}
 $n \in [T_0,T_1]$.

\begin{prop} \label{lem:genericbound}
For $d\geq d_0$, 
where $d_0$ depends only on $G, H$,  the measure
of points that are not $T_0$-generic w.r.t.\ $\Sob_{d_0}$ is $\ll T_0^{-1}$.
\end{prop}

\proof
First, consider a fixed $f \in L^2(X)$ which is in the completion
of $C_c^\infty(X)$ with respect to $\Sob_d$ (for a $d$ specified below). 
The decay of matrix coefficients (cf. \S \ref{matrixcoeff}, esp. \eqref{mcunipotent}, \eqref{mc3} and) implies that:
$$
 \left| \langle u(t) f, f \rangle - \left( \int f \operatorname{d}\!\mu \right)^2 \right| 
  \ll (1+|t|)^{-1/\tempered}  \Sob_{\dim H} (f)^2.
$$

Let $\Box$ denote the square $[T^M, (T+1)^M]^2 \subset \R^2$. 
By definition of $D_T(f)$, 
$$\int_{X} |D_T(f)|^2 \operatorname{d}\!\mu=
\frac{\int_{\Box}\langle u(t) f, u(s) f\rangle \operatorname{d}\!t \operatorname{d}\!s
}{\bigl((T+1)^M-T^M\bigr)^2}
 - \left( \int f
\operatorname{d}\!\mu \right)^2.$$
Split $\Box$ into the sets where $|s-t|\leq T^{\frac{M}{2}}$,
and where $|s-t|\geq T^{\frac{M}{2}}$. Thus:
\begin{multline*}
\int_{X} |D_T(f)|^2
d\mu\ll\frac{1}{T^{2M-2}}(T^{M-1+\frac{M}{2}}+T^{2M-2-
\frac{M}{2\tempered}})\Sob_{\dim
H}(f)^2\\ \ll T^{-4}\Sob_{\dim H}(f)^2
\end{multline*}
In the last equality, we have used the fact that $M \geq 10+ 10 \tempered$. 

Therefore, \begin{equation}\label{basic-generic}
 \mu\bigl(\{x: |D_T(f)(x)|\geq s\}\bigr)\ll s^{-2}T^{-4}\Sob_{\dim H}(f)^2
\end{equation}
for any $s>0$.

To obtain the lemma for all functions $f\in C_c^{\infty}(X)$
we begin by choosing $d > d' > \dim H$ so that the relative traces
$\mathrm{Tr}(\Sob_{d'}^2|\Sob_d^2)$ and
$\mathrm{Tr}(\Sob_{\dim H}^2 | \Sob_{d'}^2)$ are both finite, cf. \eqref{relativetrace}
and \S \ref{sobnormproofs-trace}.

Next, choose an orthonormal basis $e^{(1)}, \dots, e^{(r)}, \dots$ 
for the completion of $C^{\infty}_c(X)$ with respect to the Hermitian norm
$\Sob_{d}$. By the spectral theorem,
we may choose such a basis which is also orthogonal w.r.t.\ $\Sob_{d'}$. 
Note that $\Sob_{\dim H}$ is continuous with respect to the $\Sob_{d'}$, so that
$\sum_{i} \left( \frac{\Sob_{\dim H} (e^{(i)})}{\Sob_{d'}(e^{(i)})} \right)
^2 < \infty$ by our assumption on the relative trace. 
We understand the summand as zero when $\Sob_{d'}(e^{(i)}) =\Sob_{\dim H} (e^{(i)})= 0$.

Let $c>0$ be a constant, which we will specify below in a way depending only on $X$.
Then, by applying \eqref{basic-generic} with $s=cn^{-1}\Sob_{d'}(e^{(k)})$
and $T=n$ for each $n \geq T_0$ and for each $k \geq 1$ we have:
$$
\mu \left( \bigcup_{n \geq T_0, k \geq 1}
 \bigl\{x: n |D_n(e^{(k)})(x)| \geq   c\Sob_{d'} (e^{(k)}) \bigr\} \right)  \ll
T_0^{-1}c^{-2}\ll T_0^{-1}.
$$
Let $E$ be the set indicated on the left-hand side. 
 We claim that any $x \notin E$ is generic w.r.t.\ $\Sob_d$ (once $c$ has been chosen correctly).

Take $f =
\sum_{k} f_k e^{(k)}\in C^\infty_c(X)$. We have for any $n \geq T_0$, and any $x \notin E$, 
\begin{multline*}
 n |D_n(f)(x)| \leq c\sum_{k} |f_k| \Sob_{d'}(e^{(k)}) \leq \\
  c\left( \sum_{k} |f_k|^2 \right)^{1/2}\left( \sum_{k} \Sob_{d'}(e^{(k)})^2 \right)^{1/2}
  =c\Sob_d(f) \left( \sum_{k} \Sob_{d'}(e^{(k)})^2 \right)^{1/2}
\end{multline*}
by linearity of $D_n(\cdot)$ and the finiteness of the relative trace \eqref{relativetrace}.  Choosing $c$ equal to the inverse of the last square root,
we have shown $ |D_n(f)(x)|\leq n^{-1}\Sob_{d}(f)$ for $n\geq T_0$ and $x\notin E$ as required.
\qed

We require also an adaptation of this Proposition to the setting when one works with a measure $\mu$ that is almost invariant (``Proposition A''). We refer to \S \ref{subsec:AE} for a discussion of
its meaning in qualitative terms. 

\begin{prop} \label{lemma3}
Let $H \subset S \subset G$, $S$ connected.  Suppose that $\mu$ is $\epsilon$-almost invariant
under $S$ w.r.t.\ $\Sob_{d}$, for $d \geq \ref{linftyconst}+1$.
Then there exists $\beta \in (0,1/2)$, $d' > d$, depending only on $G$, $H$, and $d$,
so that:

Whenever $R \leq \epsilon^{-\beta}$ and $T_0>0$, the fraction of points 
$(x,s) \in X \times \Sball(R)$ w.r.t.\ $\mu$ resp.\ the Haar measure on $S$
for which $x.s$ is not $[T_0,\epsilon^{-\beta}]$-generic w.r.t.\ $\Sob_{d'}$ 
is $\ll_d T_0^{-1}$.
\end{prop}

\proof
Let $B := \Sball(R)$. As in the proof of Proposition \ref{lem:genericbound} we will first estimate
$$
\frac{1}{\vol(B)}\int_{X\times B} \left| D_T(f)(xs) \right|^2
\operatorname{d}\!\mu \operatorname{d}\!s.
$$
This equals
$$
  \frac{1}{\vol(B)} \int_{B} \int_{X} |D_T(f)|^2 \operatorname{d}\!\mu^{s}(x)\operatorname{d}\!s \ll
 \epsilon R^{d \ref{firstlemconstant}} \Sob_{d}(|D_T(f)|^2) + \int_{X} |D_T(f)|^2 d\mu
$$
where $\mu^{s}$ denotes the translated measure;
 we have used the
definition of {\em almost invariant} from \S
\ref{def:almostinvariant}, the fact that $\Sball(R)$
consists of elements with $\|s\| \leq R$, and Lemma \ref{one}. 

 As in the proof of Proposition \ref{lem:genericbound}, the latter term is bounded, up to a constant depending only on 
 $G$ and $H$ by $T^{-4}\Sob_{\dim (H)}(f)^2$.

  On the other hand,
there exists a constant $\xappa$ so that
the first term is, by  \eqref{sobolevdistort}
and \eqref{product}, 
$$ \ll
 \epsilon R^{d \ref{firstlemconstant}} T^{d \xappa}
 \Sob_{d+\ref{linftyconst}}(f)^2.$$
 
We take $\beta = \frac{1}{2} (4 + \ref{firstlemconstant} d +  \xappa d )^{-1}$. 
This choice is made so that
$$\epsilon R^{d \ref{firstlemconstant}} T^{d \xappa} \leq T^{-4},$$
whenever $T \leq \epsilon^{-\beta}, R \leq \epsilon^{-\beta}$. 

Therefore, if $d \geq \dim(H), T \leq \epsilon^{-\beta}, R \leq \epsilon^{-\beta}$, 
$$
 \frac{1}{\vol(B)} \int_{B} \int_{X} |D_T(f)|^2 d\mu^{s}(x) \ll T^{-4}
 \Sob_{d + \ref{linftyconst}}(f)^2.
$$
Reasoning as in the second part of the proof of Proposition \ref{lem:genericbound}, and increasing $d$,  yields the conclusion for all $f\in
C_c^\infty(X)$.
 \qed

%%%%%%%%%%%%%%%%%%%%%%%%%%%%%%%%%%%%%%%%%%%%%%%%%%%%%%%%%%%%%%%%%%%%%%%%%%%%%%%%%%%%%%%%%%
\section{Nearby generic points effectively give additional invariance.} \label{subsec:production}

In this section, we shall use the polynomial properties of unipotent flows
and the effective ergodic theorem to establish ``Proposition B.''

Henceforth, $\Scomp$ will denote an $S$-invariant complement
  to the Lie algebra of $S$, inside
$\mathfrak{g}$.   It exists because $S$ is semisimple. 
We recall we have fixed a unipotent one-parameter subgroup $u(t) \subset H$,
see \eqref{udef}.

 There exists some $\consta\label{expconst0}>0$%
\index{xappa46@$\ref{expconst0}$, local coordinates in \S \ref{subsec:production}}
with the following property:
whenever $v,w \in \mathfrak{r}$
 $\|v\|,\|w\|\leq\ref{expconst0}$,
\begin{equation} \label{expconstdef} \exp(v) \exp(w)^{-1} = \exp(w^*) s, w^* \in \mathfrak{r}, s \in S.\end{equation}
where $d(s,1) \leq \|v-w\|$ and $\frac{\|w^*\|}{\|v-w\|} \in [1/2, 2]$. 
Indeed, the map $(v,w) \mapsto (w^*,s)$ is differentiable;
also $(v,v) \mapsto (0, e)$ and the derivative at $(0,0)$
is the map $(v,w) \mapsto (v-w, 0)$. 
The assertion follows. 

%the derivative $D_w: \mathfrak{g} \rightarrow \mathfrak{g}$ of the map $w \mapsto \log(\exp(v) \exp(w))$
%satisfies
% \begin{equation} \label{expconstdef}\|X-D(X)\|\leq\frac{1}{10}\|X\|\mbox{ for all }X\in \mathfrak g.  
%\end{equation} This follows easily from the fact that is valid when $v=w=0$. 

%The derivative of $\exp:\mathfrak g\to G$ and of $\log: G\to\mathfrak g$ at $0$ resp.\ $e$ is the identity map. Thus,
%such that the derivative $D_v:\mathfrak g\to\mathfrak g$ of $\exp$ at $v\in\mathfrak g$ with $\|v\|\leq\ref{expconst0}$ satisfies 
%A similar estimate holds for the derivative of $\log$, at any $g\in G$  of the form $g=\exp(v)\exp(w)$ for $v,w\in\mathfrak g$ with
%
\subsection{Unipotent trajectories of nearby points}
The discussion that follows may be regarded as an effective
form of the ``equicontinuity'' statements from \S \ref{subsec:B}. 

Let $x_1, x_2 \in X$ so that $x_2 = x_1 \exp(r)$, some $r \in \mathfrak{r}$.
Note that $x_2 u(t) = x_1 u(t) \exp(\mathrm{Ad}(u(-t)) r)$.
We decompose $r = r_0 + r_1$ according to \eqref{gdecomp}.

Then $\mathrm{Ad}(u(-t)) r_0 = r_0$ for all $t$, whereas $t \mapsto
\mathrm{Ad}(u(-t)) r_1$ is a polynomial of degree $\leq
\dim(\mathfrak{g})$ and all of whose coefficients (w.r.t.\ an
orthonormal basis for $\mathfrak{g}$) are $\ll \|r_1\|$.

Suppose $r_1 \neq 0$. Then there exists a positive real time $T$ satisfying $\|r\|^{-*} < T <
\|r_1\|^{-*}$, and a polynomial $q: \R \rightarrow \mathfrak{r}$ of degree $\leq
\dim(\mathfrak{g})$
    with image centralized by $u(\mathbb{R})$
    satisfying $q(0)=0$ and $\max_{s \in [0,2]} \|q(s)\| = \ref{expconst0}$, so that:
\begin{equation}\label{sp}
\mathrm{Ad}(u(-t)). r = q(t/T) +  O(\|r_1\|^{*})\mbox{ for all }t\leq 3T.
\end{equation}
Here, and in the following, the implicit constant in the $O(\cdot)$-notation is understood in
the same sense as in the $\ll$-notation. 

If $r_1 = 0$, the same statement remains true, except one sets ``$T = \infty$''
and ignores the statement about $\max \|q(s)\|$; the polynomial $q$ is in this case constant.

This statement will be used to give an effective version (``Proposition B'') of \eqref{slow-divergence}. 

\begin{prop} \label{MaL}
 Let $d\geq\ref{linftyconst}+1$.
    There exists constants $\ref{eipconstant}>0$ and $ \ref{gammaconstant}>\ref{deltaconstant}>0$ so that:
 \index{xappa19@$\ref{eipconstant}$, exponent in additional invariance}
 \index{xappa2@$\ref{gammaconstant}$, genericity needed for additional invariance}
 \index{xappa21@$\ref{deltaconstant}$, genericity needed for additional invariance}

    Suppose $\mu_1, \mu_2$ are $H$-invariant measures,
 that $x_1, x_2 \in X$ satisfy $x_2 = x_1 \exp(r)$ for some nonzero $r \in \Scomp$, and that
$x_i$ is $[\|r\|^{-\ref{gammaconstant}}, \|r_1\|^{-\ref{deltaconstant}}]$-generic w.r.t.\ $\mu_i$
and a Sobolev norm $\Sob_d$  (for $i=1,2$). 
    Then there is a polynomial $q: \R \rightarrow \mathfrak{r}$ of degree $\leq
\dim(\mathfrak{g})$,
so that:     
\begin{equation} \label{cylon} \left| \mu_2^{\exp{q(s)}}(f) - \mu_1(f)\right| \ll_d \|r_1\|^{\ref{eipconstant}} \Sob_d(f), 1 \leq s \leq 2^{1/M}\end{equation}
If $r_1 \neq 0$, then $\max_{s \in [0,2]}\|q(s)\| = \ref{expconst0}$. 
Moreover, if $\mu_1=\mu_2$ is $\epsilon$-almost invariant under $S$, then $\mu_1$ is $\ll_d \max(\epsilon,\|r\|^{\ref{eipconstant}})^{1/2}$-almost invariant under some $Z\in\mathfrak{r}$ with $\|Z\|=1$.
\end{prop}

\proof
  We let $T$ be as in the discussion before the proposition.
  The definition of generic (\S \ref{subsec:generic}) assures us (assuming that $\ref{deltaconstant}$ 
  is sufficiently small and $\ref{gammaconstant}$ is sufficiently large)
  for any integer
  $$n \in [T^{1/M}, (2 T)^{1/M}]\subset
  [\|r\|^{-\ref{deltaconstant}},\|r_1\|^{-\ref{gammaconstant}}]$$
   and for any $f \in C_c^{\infty}(X)$,  we have
  $$
  \left| \int f d\mu_j -  
  \frac{1}{(n+1)^M - n^M} \int_{n^M}^{(n+1)^M} f(x_j u(t)) dt \right| \leq n^{-1} \Sob_d(f), 
 $$
 for $j = 1,2$. 

For $t,t_0 \in [n^M,(n+1)^M]$, we have,
using \eqref{simple-Sobolev-estimate} and \eqref{sp}
\begin{multline*}
f(x_2  u(t))=
f(x_1 u(t) \exp(q(t/T)) ) + O(\|r_1\|^*  \Sob_d(f))  \\
= f(x_1 u(t) \exp q(t_0/T))  + O(T^{-1/M} \Sob_d(f)) + O(\|r_1\|^* \Sob_d(f)).
\end{multline*}
 Here we used in the last line that
$|t-t_0|\ll n^{M-1}\asymp T^{1-1/M}$ for $t,t_0\in[n^M,(n+1)^M]$ --- which shows
that the polynomial $q(t/T)$ has small variation within that
interval.

Thus, \begin{equation}\label{cyclon2}
\mu_2(f)=
  \mu_1(\exp q(t_0/T).f) + O\left(
( T^{-1/M}  + \|r_1\|^* ) \Sob_d(f)\right)
\end{equation}
which implies \eqref{cylon}.

Let us observe that if $r_1 = 0$, then the polynomial $q$ is constant,
say $q \equiv q_0$, 
and we have $\mu_2 (f) = \mu_1 (\exp(q_0) f)$ (exactly). 

Assume in the remainder of the proof that $\mu_1=\mu_2$. 
The prior remark shows that, if $r_1 = 0$, 
$\mu$ is indeed $\|r\|$-almost invariant under $Z=\|r\|^{-1}r$ by \eqref{simple-Sobolev-estimate}; this establishes the final assertion in the case when $r_1 = 0$.

 Assume, thus, that $r_1\neq 0$. 
In this case $q(0)=r$ and $\max_{s\in[0,2]}\|q(s)\|=\ref{expconst0}$, so that the coefficients of $q'(s)$ are $\gg 1$. 
By \eqref{cylon}, $\mu_1$ is
 $\ll\|r\|^{\ref{eipconstant}}$-almost invariant under $\exp(q(s))$ for 
any $s\in[1,2^{1/M}]$. 

Put $E = \max(\epsilon, \|r\|^{\ref{eipconstant}})^{1/2}$. 
There exists $s\in[1,2^{1/M}]$
 with $s+E \in[1,2^{1/M}]$ for which $\|q(s)-q(s+E)\|\gg
E$. (Rather, we may assume this is so; if $E$ is so large that this is false
for a trivial reason, then the final statement of the Proposition correspondingly
becomes trivial). 

With $s$ being so chosen, $\mu_1$ is $\ll \|r\|^{\ref{eipconstant}}$-almost
invariant under $\exp(q(s))$, $\exp(q(s+E)$, and so also under $\exp(-q(s))\exp(q(s+E)) $.
 Put $v=q(s), w=q(s+E)- v$. 
Then, by \eqref{expconstdef},  $\exp(-v)\exp(v+w)=\exp(w^*)s$ where $w^* \in \mathfrak{r}, 
\|w^*\|\asymp E$, and 
 $s\in S$
satisfies $d(s,e)\ll  E$.  

This claim proves the proposition: Since $\mu_1$ is assumed to be $\epsilon$-almost invariant under $S$, we get that $\mu$ is $\ll\max(\epsilon,\|r\|^{\ref{eipconstant}}) = E^2$-invariant under $\exp(w^*)$. 
Since $w\in\mathfrak r$ with $\|w^*\|\asymp E$
 we can iterate this statement $E^{-1}$-many times;
using \eqref{simple-Sobolev-estimate}, we deduce that $\mu$ is
 $\ll E$-almost invariant under $Z=\frac1{\|w^*\|}w^*$. See
the proof of Lemma \ref{stabilityofalmost} (esp. second bulleted point) for details of this type of iteration.

\qed

\subsection{Quantitative isolation of closed orbits.}
From this argument we may draw the following useful corollary regarding the isolation
of closed orbits for semisimple groups. It gives a quantitative
estimate on the spacing between two distinct, closed $S$-orbits,
as well on how closely such an orbit can approach itself. 

The idea of proof is simple. Suppose given two very nearby points $x_1, x_2$
so that $x_1 S, x_2 S$ are both closed and are distinct. 
By modifying $x_1, x_2$ slightly, we may assume that we are in the situation
of Proposition \ref{MaL}, with $\mu_i$ the $S$-invariant probability measure on $x_i S$. The conclusion of Proposition \ref{MaL} then implies that $x_2 S$
contains, loosely speaking, ``many different translates of $x_1 S$ along the direction
$q(s)$.'' Thus, the volume of $x_2 S$ was necessarily large.

\begin{lem} \label{iso} 
There are constants $\ref{separationheightconstant}$
and $\ref{separationvolumeconstant}$ with the following property.
\index{xappa212@$\ref{separationheightconstant}$, exponent of height in isolation}
\index{xappa2121@$\ref{separationvolumeconstant}$, exponent of volume in isolation}

Suppose $H \subset S \subset G$. 
Let $x_1, x_2\in X$ be so that $x_i S$ are closed orbits with volume $\leq V$.
Then either:
\begin{enumerate}
\item  $x_1$ and $x_2$ are on the local $S$-orbit, i.e.\ there exists some $s\in S$ with $d(s,1)\leq 1$
and $x_2=x_1 s$, or
\item $d(x_1, x_2) \gg \min(\height(x_1),\height(x_2))^{-\ref{separationheightconstant}} V^{-\ref{separationvolumeconstant}}$.
\end{enumerate}
\end{lem}

The proof will use a simple argument, which we will also use elsewhere.
Recalling the definition of $r_1$ from \eqref{gdecomp},
there exists $ \constb\label{stronggeneric} > 0$%
\index{yota3@$\ref{stronggeneric}$, lower bound for $r_1$}
 so that, for any $r \in \mathfrak{r}$, 
\begin{equation} \label{iotaclaim}
\vol \Bigl\{s : d(s,1) \leq 1\mbox{ with }\frac{\|(\Ad(s^{-1}) r)_1\|}{\|r\|}\leq \ref{stronggeneric}\Bigr\} < \frac{\vol(\{s \in S: d(s,1) \leq 1\})}{2}.
\end{equation}
Indeed, given $r' \in \mathfrak{r}$, there exists $\varepsilon(r') > 0$ 
so that the measure of 
$
\Bigl\{s : d(s,1) \leq 1\mbox{ with }\frac{\|(\Ad(s^{-1}) r')_1\|}{\|r'\|}\leq \varepsilon\Bigr\}
$ 
is less than $\frac12\vol(\{s \in S: d(s,1) \leq 1\})$. This is so, since for the
function $s \mapsto (\Ad(s^{-1}) r')_1$ is real-analytic and not identically zero; if it were $0$, then $r'$ is centralized by every conjugate of $\{u(t)\}$,
and so also by the Lie algebra of $S$, a contradiction. 
By a compactness argument, one may choose $\varepsilon=\ref{stronggeneric}$ uniform
over $r'\in\g$, establishing \eqref{iotaclaim}. 

\proof 
Throughout, we fix a Sobolev norm $\Sob_d$,
where $d$ is sufficiently large (depending only on $G,H$).\footnote{
Precisely, $d \geq \max(d_0, \ref{linftyconst}+1)$,
$d_0$ as in Lemma \ref{lem:genericbound}.}

We may suppose (adjusting $x_1$ or $x_2$ by an element of $S$) that
$x_1  = x_2 \exp(r)$, with $r \in \mathfrak{r}$, and will establish
a lower bound on $\|r\|$.

Let $\mu_i$ be the $S$-invariant probability measure on $x_i S$.
In view of Proposition \ref{lem:genericbound}, the measure of 
$$
 \{s \in S: d(s,1) \leq 1\mbox{ and $x_i s$ fails to be $V^2$-generic w.r.t.\ }\mu_i\},
$$
with respect to the volume measure on $S$, is $\ll V^{-1}$  for $i \in \{1,2\}$. We  may suppose that $V$ is sufficiently large, so that this measure,
as a fraction of the volume of $\{s \in S: d(s,1) \leq 1\}$, is at most $\frac{1}{10}$. 

We claim there exists $s \in S$, with $d(s,1) \leq 1$ with:
\begin{equation}\label{cond}
\mbox{$x_i s$ are 
$V^2$-generic w.r.t.\ $\mu_i$ and $\|(\Ad(s^{-1})r)_1\| \geq \ref{stronggeneric} \|\Ad(s^{-1})r\|$.}
\end{equation}

Indeed, let $\Omicron \subset \{s \in S: d(s,1) \leq 1\}$ consist of those $s$ 
for which $x_i s$ is $[V^2,\infty)$-generic w.r.t.\ $\mu_i$. 
By \eqref{iotaclaim}, $\Omicron$ must contain an element $s$
satisfying \eqref{cond}.

Replace $x_i$ by $x_i s$ and $r$ by $\Ad(s^{-1})r$, where $s$ is so that \eqref{cond} is satisfied.
This has the effect of increasing $\|r\|$ by, at most, a constant
factor depending only on $G,H$. 
Thereby, we may assume that $x_i$ are both $V^2$-generic and that $\|r_1\| \gg \|r\|$.

Let $\eta>0$ be smaller than the injectivity radius $\ref{hc1}
 \height(x)^{-\ref{hc2}}$ of \eqref{injfact} (specified below). 
Let $f$ be a smooth bump function on $G$ supported in $\{g : d(g,1) \leq \eta\}$ with $0\leq f\leq 1$
which is one in a neighborhood of radius $\gg\eta$ and whose partial derivatives up to order $d$ are everywhere $\ll\eta^{-d}$.
Let $f_{x_1}$ be the function defined by $x_1g \mapsto f(g)$ when $d(g,1) \leq \eta$, and zero outside that ball.

Then, evidently, $\int f_{x_1}\operatorname{d}\!\mu_1 \gg V^{-1}\eta^{\dim S}$. 
On the other hand, the definitions of the Sobolev norm imply that
\begin{equation}\label{fxnorm}\Sob_d(f_{x_1}) \ll  \height(x_1)^d \eta^{-d}.\end{equation}

Without loss of generality, $\|r\| \leq V^{-2/\ref{gammaconstant}}$. 
We apply Proposition \ref{MaL}. 
Therefore, Proposition \ref{MaL} assures us that,
so long as 
\begin{equation}\label{imp-equation}
\height(x_1)^d\eta^{-d} \|r\|^{\ref{eipconstant}} \leq
c V^{-1}\eta^{\dim(S)}
\end{equation}
(for some $c$ obtained from the above implicit constants), then,
\begin{equation} \label{nearby} \int f_{x_1}^{q(s)}\operatorname{d}\!\mu_2>0\mbox{
for }s \in (1, 2^{1/M}).
\end{equation}

In words, \eqref{nearby} asserts that the orbit $x_2 S$ 
passes very close to $x_1 S$, many times.
More precisely, given $s_1, s_2$ so that 
$ \eta \ll  d(q(s_1)^{-1},q(s_2)^{-1}) \ll \height(x)^{-\ref{hc2}}$, 
the functions $f_{x_1}^{q(s_1)}$ and $f_{x_1}^{q(s_2)}$ have 
distinct support. Moreover, $x_2 S$ must intersect
both of their supports.
Thereby, we obtain from $s_1, s_2$, 
two distinct
discs of radius $\asymp \height(x)^{-\ref{hc2}}$ on $x_2 S$.

% More precisely,
%with the given displacement $q(s)$ there exists a piece of the orbit
%$\eta$-close to $x_1q(s)^{-1}$. 
%
%Using two values $s_1$ and $s_2$ here
%we find different pieces of the $x_2$-orbit if $d(q(s_1)^{-1},q(s_2)^{-1})$
%is greater than $2\eta$ and smaller than
%a certain muliple of the injectivity radius $\ref{hc1}
% \height(x)^{-\ref{hc2}}$ --- for then the functions
%$f_{x_1}^{q(s_1)}$ and $f_{x_1}^{q(s_2)}$ have disjoint support. Moreover,
%in this case the discs of radius $\asymp \height(x)^{-\ref{hc2}}$ within these 
%ipieces of the $S$-orbit of $x_2$ are disjoint.
%Using $\asymp\frac{\height(x)^{-\ref{hc2}}}{\eta}$ many values of $s$ this implies, in particular, that
Using not just $s_1, s_2$ but many such $s_i$s, we conclude that:
\begin{equation}\label{volbound}\vol(x_2 S) \gg 
 \height(x_1)^{-\ref{hc2}(\dim S+1)}\eta^{-1}.\end{equation}

Take $\eta$ so that the right hand side is (including the implicit constant) equal to $2V$. Then \eqref{imp-equation} cannot hold, for $\vol(x_2 S) \leq V$. 
This bounds $\|r\|$ from below by
an expression of the form $\height(x_1)^{-\star} V^{-\star}$ as required.
    \qed

%%%%%%%%%%%%%%%%%%%%%%%%%%%%%%%%%%%%%%%%%%%%%%%%%%%%%%%%%%%%%%%%%%%%%%%%%%%%%%%%%%%%%%%%%%%%%%%%%%%%%%%%%%%%%%%%%%%%%%%%%%%%%%%%%%%%%%%%%%%%%%%%%%%%%%%%%%%%
\section{The measure of points near periodic orbits is small.} \label{sec:quantiso}

Let $\mu$ be as in Theorem \ref{thm:main}. 
Using the polynomial properties of unipotent flows, together
with the quantitative isolation of periodic orbits established in Lemma \ref{iso}, we will be able to establish that the total $\mu$-measure of the set of points
near periodic $S$-orbits of bounded volume remains rather small.  

The proof proceeds as follows.  By Lemma \ref{iso}, we can restrict attention
to a single periodic $S$-orbit. We then divide $X$ into a piece near the cusp,
and its complement.  The measure of the piece near the cusp
is handled by Lemma \ref{measureoutsidecompact}; the complement
is handled using a linearization of the flow near the periodic $S$-orbit. 

\begin{prop} \label{lem:Quantiso} %(The measure of points near periodic-orbits is small.)
Let $S \supset H$. There exists some $V_0$ depending on $\Gamma$, $G$, and $H$ and some
$\ref{expQuantiso}>0$ with the following property.\index{xappa22@$\ref{expQuantiso}$, exponent in isolation}
 Let $V\geq V_0$ and suppose $\mu(Y) = 0$ if $Y$ is any closed $S$-orbit of volume
$\leq V$.  Then:
\begin{equation}\label{freddy}
\mu\bigl(\bigl\{x  \in X: \mbox { there exists }x'
\stackrel{V^{-\ref{expQuantiso}}}{\sim} x \mbox{ with }\vol(x' S)
\leq V\bigr\}\bigr) \leq1/2
\end{equation}
\end{prop}

\proof This is a quantitative form of polynomial nondivergence. More
precisely, we use Lemma \ref{iso} together with the
``linearization'' technique (see \S \ref{existing-work:HD}). 

Let $\mathfrak{r}$ be as in \S \ref{subsec:production}.  Let $x_0 S$ be any closed $S$-orbit.
In view of Lemma~\ref{iso} and Lemma~\ref{measureoutsidecompact} (1),
the number of such orbits with volume $\leq V$ is bounded by a polynomial function of $V$.

%Consequently, it would suffice to show that a $\delta$-neighborhood of a fixed closed orbit with volume $\leq V$
%has measure $\ll V^*\delta^*$. For then choosing $\delta$ to be the reciprocal
%of a large enough power of $V$ we can get \eqref{freddy}. However, handling the $\delta$-neighbhorhood
%far up in the cusp might be difficult. For that reason we fix some maximal height $R\geq 1$
%and prove the following weaker statement
%existence of constants $\consta \label{dconstant}, \consta \label{Vconstant}$ so that

We show that for any closed orbit $x_0S$ of volume $\leq V$: %\index{$\ref{dconstant}$}
\begin{equation} \label{boundnearone} 
\mu\bigl( (x_0  S \cap\Siegel(R))\exp\{r \in \mathfrak{r} : \|r\| \leq \delta \} \bigr) \ll 
 V^\star R^\star\delta^\star %{\ref{Vconstant}} \delta^{\ref{dconstant}}  
\end{equation}
We shall only sketch the argument for \eqref{boundnearone},
for it is by now quite standard. The deduction of \eqref{freddy} is then straightforward: take $R=V$, $\delta$
the reciprocal of a large power of $V$, and apply Lemma \ref{measureoutsidecompact},

%Note that even this weaker statement easily implies
%\eqref{freddy}: By setting $R=V$ and $\delta$ equal to the reciprocal
%of a large power of $V$ we can get that the measure of the union of the sets in \eqref{boundnearone}
%over all closed orbits of volume $\leq V$ is $\leq \frac{1}{4}$ while
%the set $\Siegel(V/2)$ has measure $\leq V^{-\star}$ by Lemma \ref{measureoutsidecompact}.

 Let $\Omega_\delta = \{r \in \mathfrak{r}: \|r\| < \delta\}$.
For $ \delta_0 =  R^{-\star}V^{-\star}$, the map
$$(x_0 S \cap \Siegel(R)) \times \Omega_{\delta_0} \rightarrow X, \ \ (y, r) \mapsto y \exp(r)$$
is a diffeomorphism onto an open neighbourhood $\mathcal{N}_{\delta_0}$ of $x_0 S \cap
\Siegel(R)$, as follows from the implicit function theorem, Lemma~\ref{iso} and \eqref{injfact}.

 Let $\pi: \mathcal{N} \rightarrow \Omega_{\delta_0}$
be the natural projection.  For $\delta \leq \delta_0$, let  $\mathcal{N}_\delta =
\pi^{-1} (\Omega_\delta)$.

For $y \in \mathcal{N}$ and $t \in \R$, we have:
\begin{equation} \label{obvious} \pi(y u(t)) = \mathrm{Ad}(u(-t)) \pi(y)\end{equation}
so long as $y u(s) \in \mathcal{N}_{\delta_0}$ for all $s\in [0,t]$. Notice that the
latter is equivalent to $\mathrm{Ad}(u(-s)) \pi(y) \in 
\Omega_{\delta_0}$ for all $s\in[0,t]$.

Choose $x \in X$ which is generic for the flow $u(\cdot)$ w.r.t.\ $\mu$.
Recall that, as in Section \ref{semisimple-measures}, 
this means the measure along the trajectory $\{ x u(t): 0 \leq t \leq T\}$
approximates $\mu$ as $T \rightarrow \infty$.
Let us observe that, for any such $x$ and all points 
$y \in \{x u(t) \cap \mathcal{N}_{\delta_0}\}$, the polynomial
$t \mapsto \mathrm{Ad}(u(-t)) \pi(y)$ is nonconstant. 
Suppose this to be false; set $Y = \pi(y)$. 
The genericity of $y$ implies that 
$\supp(\mu) \subset x_0 S . \exp(Y)$. 
Since $\supp(\mu)$ is an $H$-orbit,
this implies $H\subset \exp(-Y) S\exp (Y)$.
This is impossible, in view of Lemma \ref{assumption-2},
when $\|Y\|$ sufficiently small.  Since
$\|Y\| \leq \delta_0$ and $\delta_0=R^{-\star}V^{-\star}$, we may assume $\|Y\|$ is sufficiently small by changing the implicit constants. 

Let $\delta \leq \delta_0$. 
Let $Z_\delta = \{t \in \R_{\geq 0}: x u(t) \in \mathcal{N}_\delta\}$ 
be the set of times where the orbit is $\delta$-close to $x_0S$. 
Notice that $Z_\delta$ is a union of intervals.
There 
exists a constant $\consta \label{dmconstant}$ so that%
\index{xappa7@$\ref{dmconstant}$, exponent in linearization argument}
we may cover $Z_\delta$ by intervals $B_j$ satisfying:
$\frac{|B_j \cap Z_\delta|}{|B_j|}  \leq \constc \bigl(\delta\delta_0^{-1}\bigr)^{\ref{dmconstant}}$ where
we use the notation $|I|$ to denote the length of an interval $I\subset\R$.  
That follows from \eqref{obvious} and
an argument using the growth properties of the polynomial
$t \mapsto \mathrm{Ad}( u(-t)) \pi(y)$. \footnote{Essentially, the argument here is that a polynomial which is 
$\delta$-small on an interval stays $\delta_0$-small on a bigger interval.} 
Here the intervals $B_j$ can be chosen to be disjoint (since they correspond
to different visits of the orbit of $x$ to $\mathcal N_0$).

We are using this to bound for a given $T$ the fraction of times $t\in[0,T]$ with $xu(t)\in \mathcal N_\delta$. 
Clearly, we may assume each $B_j$ intersects $[0,T]$. Let $J$
be the interval which is the union of $[0,T]$ with the $B_j$. 
Then, 
$$
 |J \cap Z_{\delta}| \leq \sum |B_j \cap Z_{\delta}|
\leq  \constc \bigl(\delta\delta_0^{-1}\bigr)^{\ref{dmconstant}}\sum |B_j| 
\ll \bigl(\delta\delta_0^{-1}\bigr)^{\ref{dmconstant}} |J| .
$$

This being true for a sequence of intervals $J$ of increasing length, we see that 
$$
\mu(\mathcal{N}_\delta) \ll  \bigl(\delta\delta_0^{-1}\bigr)^{\ref{dmconstant}}
\mbox{ whenever }\delta \leq \delta_0=R^{-\star}V^{-\star}.
$$
This is \eqref{boundnearone} (since for $\delta\geq \delta_0$ the claim is trivial anyway).
\qed

%%%%%%%%%%%%%%%%%%%%%%%%%%%%%%%%%%%%%%%%%%%%%%%%%%%%%%%%%%%%%%%%%%%%%%%%%%%%%%%%%%%%%%
\section{Some lemmas connected to lattice point counting.}\label{lattice}

\myparagraph It is a general feature that, if $H_1 \subset H_2$ are nice groups, the spectrum of $L^2(H_2/H_1)$
is related to the volume growth of $H_1$ inside $H_2$.
In the case when $H_1$ is a {\em lattice} inside $H_2$, this may be used to derive asymptotics for the number of points of $H_1$ inside a large ball.

See \cite{Duke-Rudnick-Sarnak, EM} for various instances of this technique in 
the context of Lie groups. 
A closely related idea, in a slightly different context, was introduced
and utilized by G.M. in \cite{Margulis-thesis}.

We shall need a slight variant, where we shall give
upper bounds on the counting functions for {\em cosets} of a lattice. 
 \begin{prop}  \label{sl32}
 Let $H \subset S \subset G$ with $S$ connected.  Let $\Lambda \subset S$ a discrete subgroup, $B \subset S$ an open set, $s_0 \in B$.

   Then:
   \begin{enumerate}
   \item \begin{equation} \label{lowerbound} \bigl|\Lambda \cap \tilde{B} \bigr| \gg \frac{\vol(B \cap S)}{\vol(\Lambda \backslash S)} - C
   \int_{\tilde{B}} \varphi_0^{\wk}(s)^{\rho}\operatorname{d}\! s\end{equation}
        \item
    the cardinality of any $1$-separated subset $\Delta\subset B\cap (\Lambda  s_0) $ is bounded by
\begin{equation} \label{upperbound2} |\Delta|^2 \ll 
\left( \frac{\vol(\tilde{B})^2}{\vol(\Lambda \backslash S)} + 
 \int_{\tilde{B}} \varphi_0^{\wk}(s s'^{-1})^{\rho }\operatorname{d}\! s\operatorname{d}\! s' \right)\end{equation}
\end{enumerate}
where:
\begin{itemize}
\item $s \mapsto \varphi_0(s)$ is the Harish-Chandra spherical function associated to $S$, and $\varphi_0^{\wk}$ the modification
for the case when $S$ has multiple simple factors (see \S \ref{matrixcoeff});
\item $\rho > 0$ depends only on the spectral gap
for $L_0^2(\Lambda \backslash S)$ and $C$ depends only on $G,H$. Here $L_0^2$ 
denotes the orthogonal complement of locally constant functions. 
\item $\tilde{B} = \{s \in S: d(s,1) \leq 1\}   B \{s \in S: d(s,1) \leq 1\}$.
\end{itemize}  Moreover, we interpret $\vol(\Lambda \backslash S) := \infty$
if $\Lambda$ is not of cofinite volume.
  \end{prop}
Note that one deduces from this bounds for the size of any $\delta$-separated  subset
of $\Lambda s_0 \cap B$, for $\delta < 1$.  Indeed, if there exists a $\delta$-separated
subset of a given metric space of size $N_1$, there exists a $1$-separated subset of size $\geq N_1/N_2$, where $N_2$ is the largest possible size of a $\delta$-separated set within a $1$-ball.

  \proof

  Let $\chi$ be a non-negative smooth function supported in the neighbourhood
$\Omega = \{s \in S: d(s,1) \leq 1/2\}$, satisfying $\int \chi =  1$.  Note that $\Omega = \Omega^{-1}$.
Set $$B_1 = B. \Omega, B_2 = \Omega B_1 = \Omega. B. \Omega,
B_3 = \Omega. B_2 . \Omega,$$ and note that $B_1, B_2, B_3$ are all contained in $\tilde{B}$.

  Let $\pi: S  \rightarrow \Lambda \backslash S$ be the projection and
  $\pi_*: C_c(S) \rightarrow C_c(\Lambda\backslash S)$ be the natural projection map,
  i.e.\ $\pi_*(f)(\Lambda g) = \sum_{\lambda \in \Lambda} f(\lambda g)$.
 Let $\pi^*$ be the pullback
  $C(\Lambda \backslash S) \rightarrow C(S)$.
  Then, for $f_1 \in C_c(S), f_2 \in C_c(\Lambda \backslash S)$, we have:
  $$\langle \pi_*(f_1),f_2  \rangle_{\Lambda \backslash S} = \langle f_1, \pi^* f_2 \rangle_{S}$$

Let $f(s) = \int_{B_2} \chi(ss_1^{-1}) \operatorname{d}\!s_1 \in  C_c(S)$.
Then $\supp(f) \subset \Omega B_2$ and $f \geq 1_{B_1} $, where $1_{B_1}$ is the characteristic function of $B_1$.
Indeed, for $s_1 \in B_1$ we have $f(s_1) = \int_{ B_2} \chi( s_1 s^{-1})\operatorname{d}\!s  \geq
\int_{\Omega s_1} \chi(s_1 s^{-1})\operatorname{d}\!s  = 1$.  By these definitions and since $0 \leq f \leq 1$ always,
we have that
\begin{multline*}
 \langle \pi_* f, \pi_* \chi \rangle_{\Lambda \backslash S} = \langle \pi^* \pi_* f, \chi\rangle_{S}
\\ =  \sum_{\lambda \in \Lambda} \int_{s \in S} f(\lambda s)  \chi(s) ds \leq |\Lambda \cap \mathrm{supp}(f). \Omega|
\leq |\Lambda \cap B_3|.
\end{multline*}

Also note that $\pi^*\pi_*1_{B_1}(s)=\sum_{\lambda\in\Lambda}1_{B_1}(\lambda s)=\bigl|\Lambda s \cap B_1\bigr|$ and so
\begin{multline} \label{dodo}
\langle \pi_* f, \pi_* f \rangle_{\Lambda \backslash S} = \langle f, \pi^* \pi_* f \rangle_{S}
\geq \langle 1_{B_1}, \pi^* \pi_* 1_{B_1} \rangle_{S} =\\ \int_{B_1}
\bigl|\Lambda s \cap B_1\bigr| \operatorname{d}\!s  =
\int_{ B_1} \bigl| \Lambda \cap B_1s^{-1}\bigr|  \operatorname{d}\!s
\end{multline}
Next we need to understand the relationship between
separated subsets of $\Lambda s_0\cap B$ and the expression $\bigl| \Lambda \cap B_1s^{-1}\bigr|$.
Suppose that there exists a $1$-separated subset $\Delta \subset \Lambda s_0 \cap B$,
for some $s_0 \in B$.  Clearly for each $\delta=\lambda s_0 \in \Delta$, we have
$\Delta \delta^{-1} \subset \Lambda s_0 \delta^{-1} \cap B \delta^{-1} \subset \Lambda \cap B \delta^{-1} $.   
In particular, whenever $s \in \delta \Omega$, we must have
$\Delta \delta^{-1} \subset \Lambda \cap B_1 s^{-1}$ and so $|\Lambda\cap B_1 s^{-1}|\geq |\Delta|$.
Because the set $\Delta$ is $1$-separated, the balls $\delta \Omega$ for $\delta\in\Delta$
are disjoint. Each such $\delta \Omega$ with $\delta \in \Delta \subset B$ is contained in $B \Omega \subset B_1$.
Therefore, the integral on the right hand side of \eqref{dodo} is $\geq |\Delta|^2 . \vol(\Omega)$.

We have shown that
$$|\Lambda \cap B_3| \geq \langle \pi_* f, \pi_* \chi \rangle \mbox{ and }
|\Delta|^2 \leq \vol(\Omega)^{-1} \| \pi_* f\|_{L^2(\Lambda \backslash S)}^2.$$
It remains to estimate $\| \pi_* f\|_{L^2(\Lambda \backslash S)}$
and $\langle \pi_* f, \pi_* \chi \rangle$.

 The
integral of $f$ equals $\vol(B_2)$. Consequently, the
projection of $\pi_* f$
onto the locally constant functions has $L^2$-norm $\asymp \frac{\vol(B_2)}{\vol(\Lambda \backslash S)^{1/2}}$. 
%(Here the $\asymp$ accounts for the possibility of multiple connected components.)

Similarly, if $\Proj$ denotes the orthogonal projection from $L^2(\Lambda \backslash S)$ to constant functions,
% on the identity component
 we see that
$\langle \Proj \pi_* f, \Proj \pi_* \chi \rangle \gg \frac{\vol(B \cap S)}{\vol(\Lambda \backslash S)}$.

 To handle the projection of $\pi_* f$ onto the orthocomplement of the constants, we use standard bounds on matrix coefficients (see \eqref{mc1} -- \eqref{mc3}
and \eqref{mcweak}). 
We write  $\Proj_0$ for the projection onto the orthogonal complement of the locally constant functions. 
Then the bounds on the matrix coefficients show that there is $\rho > 0$, depending only on the spectral gap, so that:
 $$| \langle s.  \Proj_0 \pi_* \chi ,  \Proj_0 \pi_* \chi \rangle  | \ll \varphi_0^{\wk}(s)^{\rho}  \ \ \ (s \in G). $$
Recall that $f$ was defined as the integral of the right translate of $\chi$ by $s_1^{-1}$ with $s_1\in B_2$,
which obviously give the same description of $\pi_*f$ in terms of $\pi_*\chi$. Therefore, we have proved
the claims \eqref{lowerbound}
  and \eqref{upperbound2} -- at least with certain instances of $\tilde{B}$ replaced by $B_2$.
  But it is easy to see that this is harmless because of the inclusion  $B_2 \subset \tilde{B}$.
  \qed

%%%%%%%%%%%%%%%%%%%%%%%%%%%%%%%%%%%%%%%%%%%%%%%%%%%%%%%%%%%%%%%%%%%%%%%%%%%%%%%%%%%%%%%%%%%%%%%
\section{Effective closing Lemma.}

Let us recall the closing lemma for hyperbolic flows. Let $M$ be a compact manifold; let $h_t: M \rightarrow M$ be a one-parameter flow of smooth diffeomorphisms with hyperbolicity transverse to the flow direction.  Suppose
that $x \in M$ is so that the distance between $h_T x$ and $x$, measured w.r.t.
a fixed Riemannian metric on $M$, is sufficiently small.  Then there exists $y$ close to $x$
and $T'$ close to $T$ so that $h_{T'}y = y$, i.e. $y$ has a periodic orbit under $\{h_t\}$.

We shall need a method for producing
periodic $S$-orbits, which is, in a certain sense, an analogue of this result.  This is Proposition \ref{CaseIIlemma} (``\propHname'' from the outline) -- it asserts that,
if $x \in \Gamma \backslash G$ and if we are given a ``sufficiently large'' collection of elements $s_i \in S$ so that $x s_i$  are all mutually close to each other, then there exists $x'$ near $x$ so that $x' S$ is closed.

\begin{prop} \label{CaseIIlemma}
Let $H \subset S \subset G$, with $S$ connected.
Let $\delta \leq 1 \leq N$. 
There exists some $T_0 = T_0(\Gamma, G, H,N)$ with the following
property:

 Let $T\geq T_0$ and let $v = \vol \Sball(T)$. Suppose that
$\{s_1,\ldots,s_k\}\subset \Sball(T)$ is $\frac1{10}$-separated, that
$k\geq v^{1-\delta}$, and that there exists $x \in \Xcompact$ so
that $x s_i \stackrel{T^{-N}}{\sim} x s_j$.
Then there is $x' \stackrel{T^{-N_\uparrow}}{\sim} x$ so that $x' S$
is a closed orbit of volume $\leq T^{\delta_{\downarrow}}$.
\end{prop}
Recall that the $\uparrow, \downarrow$ notation was defined in
\S \ref{notation}.   In particular, the result is vacuous
unless $\delta$ is small enough. 
\subsection{Preparations for the proof.}

In the course of proof, we shall use the phrase $T$ sufficiently
large to mean that $T \geq T_0(\Gamma,G,H,N)$; we are free to assume that $T$ is sufficiently large. This allows us to replace any term of the form $cT^{N_\uparrow}$ by $T^{N_\uparrow}$.  Similarly, we shall say $N$ is sufficiently large
if $N \geq N_0(G,H)$. In view of the notation $N_{\uparrow}$,
we are free to assume that $N$ is sufficiently large. 

Fix $g_0 \in G$ so that $\Gamma g_0 = x$.
Here $g_0$ may be chosen
in a compact subset of $G$, depending only on $\Gamma, G$ (see discussion 
subsequent to Lemma \ref{measureoutsidecompact}.)
There is $\gamma_{ij} \in \Gamma$ so that
$ \gamma_{ij} g_0 s_i  \stackrel{T^{-N}}{\sim}  g_0 s_j$ inside $G$.
Thus, by \eqref{distortion2},  $\gamma_{ij} \stackrel{T^{-N_{\uparrow}}}{\sim} g_0 s_j s_i^{-1} g_0^{-1}$
for sufficiently large $T$ and $N$. The assumption $s_i\in \Sball(T)
\implies \|s_i\| \leq T$ together with \eqref{distortion} implies that
 $\|\gamma_{ij} \| \leq \constc \label{aconstant} T^2$
 and  $\gamma_{jk}  \gamma_{ij}  \stackrel{ T^{-N_{\uparrow}}}{\sim} \gamma_{ik}$;
which implies that
 $ \gamma_{jk}\gamma_{ij}  = \gamma_{ik}$
 if $T$ and $N$ are sufficiently large.

To sum up,  in the setting of Proposition \ref{CaseIIlemma}
we have produced a collection of elements $\gamma_{ij}  \in \Gamma$
so that:
 \begin{eqnarray}
\label{aconstantequation} \| \gamma_{ij}\| \leq \ref{aconstant} T^2 \\
 \label{multiplicativity} \gamma_{jk}.\gamma_{ij}  = \gamma_{ik}. \\
 \label{eq:gammadef} \gamma_{ij} g_0 s_i  \stackrel{T^{-N}}{\sim}  g_0 s_j %\ \ \ \mbox{inside $G$}
\end{eqnarray}

Our proof proceeds as follows:  First, we show we may slightly adjust
$g_0$ to a nearby $g_1$ so that, in fact, $\gamma_{ij} \in g_1 \tilde{S} g_1^{-1}$. Here $\tilde{S}$ is the 
normalizer of $S$ as defined in \S \ref{sec: intermediate}.
Next, we show that $x' = \Gamma g_1$ has the property required:
$x' S$ is a closed orbit of small volume.

\subsection{Proof of Proposition \ref{CaseIIlemma}.} \label{c2lp}
We observe in advance that we will defer two results in the course of proof to the end of the section,
to avoid interrupting the flow.

 Let $e_1, \dots, e_d$ be a basis for $\mathfrak{s}$, the Lie algebra of $S$.
Let $V := ( \wedge^{\dim(\mathfrak{s})}\mathfrak{g} )^{\otimes 2}$,  and set $v_S = (e_1 \wedge \dots \wedge e_d)^{\otimes 2} \in V_{\R} := V \otimes_{\Q} \R$. Then $\G$ acts on the $\Q$-vector space $V$,
$G$ acts on the $\R$-vector space $V_{\R}$, and
the stabilizer of the vector $v_S$ is precisely $\tilde{S}$;
the orbit $G. v_S$ is a smooth submanifold of $V_{\R}$,
as follows from general facts about orbits for algebraic groups. 
%Indeed, the orbit $G.v_S$ is (algebraically) a locally closed
%algebraic subset of the expected dimension, defined over the real numbers. It follows
%that its real points are smooth; moreover, by the implicit function theorem,
%the image of the map g--->gv_S contains a neighbourhood of v_S
%which shows that G.v_S is an open subset of the real points of this 
%algebraic variety.  

  There exists a $G$-equivariant projection map from an open neighbourhood
of $G. v_S$ to $G. v_S$. A precise statement
and self-contained discussion of what we need in \S \ref{eqproj}. 
See \cite[Theorem 2.7]{Luna} for much stronger and general statements. 
%Luna deals with closed orbits; I think G. v_S is closed but, if it were not, 
%we could always switch to a suitable choice of v_S and representation V
% with this property.  

For any finite subset $F$ of the $\gamma_{ij}$s,
let \begin{equation} \label{Xdef} X(F) = \{g \in G: F \subset
 g \tilde{S} g^{-1}\}.\end{equation} (We work with $\tilde{S}$ instead of $S$ only to make
the following argument as explicit as possible.)
There is an integer $\consta \label{effnoethconstant}$%
\index{xappa71@$\ref{effnoethconstant}$, Noetherian argument}, depending only on $G$ and $H$ so that:
\begin{equation} \label{Xset} F \subset \{\gamma_{ij}\}, \ \ |F| \leq \ref{effnoethconstant}, \ \ X(F) = X(\{\gamma_{ij}\}).\end{equation} This is an example of an ``effective'' Noetherian argument. Start with $F  = \emptyset$
and adjoin one $\gamma_{ij}$ at a time to $F$. The resulting sets $X(F)$
are the real points of a descending sequence of algebraic varieties, which must terminate; \eqref{Xset} gives an explicit estimate for how far one must go.
We give a precise argument in \S \ref{Arg:EffectiveNoetherian}\hyperlink{Arg:EffectiveNoetherian}.
Take such an $F = \{\delta_1, \dots, \delta_m\}$, where $m \leq \ref{effnoethconstant}$.

The Euclidean norm on $\mathfrak{g}$ induces one on $V_{\R}$ and $V_{\R}^m$.
 Consider the map
$$V_{\R} \xrightarrow{A} V_{\R}^m$$
where $A = \oplus_{i=1}^{m} (\delta_i -1)$. The element $g_0 v_S$ almost belongs to the kernel of $A$: in fact $\|A.  g_0 v_S\| \ll T^{-N_\uparrow}$, as follows from \eqref{eq:gammadef} and the fact $\|s_i\| \leq T$. 

Moreover, with respect to a fixed basis for the $\Q$-vector space
$V$, all the entries of $A$ are rational numbers of numerator and denominator $\leq T^*$. This follows from \eqref{aconstantequation}, and the fact that,
since $\Gamma$ is arithmetic, the denominators of matrix entries of $A$ are bounded below. 

 It
 follows that (see Lemma \ref{dioph} at the end of the present section for an explication) there is a nearby vector in the kernel of $A$
which belongs to the kernel, i.e. an element $v_S' \in V_{\R}$
such that
\begin{equation} \label{deltaequation}\delta_i v_S' = v_S',\ 1 \leq i \leq m
\mbox{ and } \|v_S' -g_0 v_S\| \ll T^{-N_\uparrow}.\end{equation}

We claim that there exists $g_1$, satisfying 
$g_1 \stackrel{T^{-N_\uparrow}}{\sim} g_0$ for large enough $T$ and $N$, so that $v_S' := g_1 v_S$
also satisfies \eqref{deltaequation}.  This claim will imply that
 $g_1\in X(F) \neq \emptyset$ (see \eqref{Xset}) and therefore that
 $\gamma_{ij} \in g_1 \tilde{S} g_1^{-1}$.

We have noted that $g_0$ may be chosen to belong to a compact subset of $G$ depending only on $H$ and $\Gamma$. Therefore, $ v_S$ and so also $g_0^{-1} v_S'$ belongs to a fixed compact subset of $V_{\R}$
depending only on $\Gamma$, $G$, $H$, and $S$. The equivariant projection map 
$\Pi$ mentioned above restricted to this subset
is Lipshitz, with a constant depending only on  $\Gamma$, $G$, $H$, and $S$. Therefore -- replacing $v_S'$ by $g_0 \Pi(g_0^{-1} v_S')$ -- we may assume that
$v_S' \in G. v_S$ without changing the fact that  $\|v_S' - g_0 v_S\| \ll T^{-N_\uparrow}$.
Next, the map $g \mapsto g. v_S$ is a submersive map from $G$ to $G. v_S$, in a neighbourhood of $g_0$. 
We may thereby find $g_1$ near $g_0$ so that $v_S' = g_1 v_S$. The claim after \eqref{deltaequation} follows, at least for $T$ sufficiently large.

Having modified $g_0$ to a nearby  $g_1$, we now proceed to show that
$x' = \Gamma g_1$ has the properties required by the Proposition. 
Precisely, we shall show that:
\begin{equation}\label{desid} \Lambda:=g_1^{-1} \Gamma g_1 \cap S
\mbox{ is a lattice within }S \mbox{ of covolume } \ll T^{\delta_{\downarrow}.}\end{equation}
This will establish the proposition. 

By construction, $g_1 \stackrel{T^{-N_\uparrow}}{\sim} g_0$ and $\gamma_{ij} \in g_1 \tilde{S} g_1^{-1}$.
Set $ \gamma_{ij}' = g_1^{-1} \gamma_{ij} g_1, s_i' = \gamma_{1i}' s_1$.
Then $\gamma_{ij}', s_{i}' \in \tilde{S}$ and $s_i' = \gamma_{ji}' s_j'$ by \eqref{multiplicativity} for all pairs $i,j$.
Also, $s_i' \stackrel{T^{-N_\uparrow}}{\sim} s_i$ for sufficiently large $T$, because $s_i \stackrel{T^{-N_{\uparrow}}}{\sim}
g_0^{-1} \gamma_{1i} g_0 s_1$ (see \eqref{eq:gammadef}).
%\footnote{Indeed,
%$d(s_i', s_i) = d(g_1^{-1} \gamma_{1i} g_1 s_1, g_0^{-1} \gamma_{1i} g_0 s_1) + T^{-N_{\uparrow}}$, and
%$ d(g_1^{-1} \gamma_{1i} g_1 s_1, g_0^{-1} \gamma_{1i} g_0 s_1) 
%\ll T d(g_1^{-1} \gamma_{1i} g_1, g_0^{-1} \gamma_{1i} g_0)$.
%Write $g_1 = g_0 u, \gamma' = g_0^{-1} \gamma_{1i} g_0$;
%then the last distance equals $d(u^{-1} \gamma' u, \gamma )
%= d(\gamma' u, u \gamma') \leq d(\gamma', \gamma' u) + d(\gamma', u \gamma')
%\ll (1 + \|\gamma\|) d(u,1)$; whence the conclusion. [[THIS FOOTNOTE TO BE DELETED -- FOR VERIFICATION ONLY.]]}
Consequently, for $T$ sufficiently large, the $s_i'$ are
 $\frac{1}{20}$-separated and belong to 
$B_{S}(\constc \label{grc} T)$.  \footnote{Indeed, the fact that
 $s_i' \stackrel{T^{-N_\uparrow}}{\sim} s_i$ forces
$s_i' = s_i \omega_i$, where $\omega_i \stackrel{T^{-N_{\uparrow}}}{\sim} 1$.
For $T$ sufficiently large, this forces $\omega_i \in S$.
Therefore, by property (3) of \S \ref{intermediateballs}
and the fact $s_i \in B_S(T)$, we conclude that
$s' \in B_S(\ref{grc} T)$. }

We are going to apply the upper bound (part 2) of Proposition \ref{sl32} with
 $\Lambda$ as in \eqref{desid},  $B := B_{S}(\ref{grc} T)$
and with $\{s_i'\}$ as the separated subset of $B \cap \Lambda s_1'$. 
By the remark after Proposition \ref{sl32}, the proposition
also yields bounds for the cardinality of any $\frac{1}{20}$-separated subset, such as $\{s_i'\}$; the upper bound is weaker by
a constant that depends only on $H,G$.
  Notations as in that proposition,
the ball ``$\tilde{B}$'' is contained in $B_{S}(\constc \label{3const} T)$, for $T$ sufficiently large, as we see by applying property (3)
of \S \ref{intermediateballs}. 

In order to apply that Proposition, we apply Proposition \ref{sectionslice}
with $x = \Gamma g_1$ to show that the action
of $S$ on the orthogonal complement of the locally constant functions in
 $L^2(\Lambda \backslash S)$ has a
spectral gap (depending only on $G,H$).
Let $\rho$ be as in the statement of Proposition \ref{sl32}; it depends only on this spectral gap, and, therefore, only on   $H,G$.

  We note that, by results established in \S \ref{volume},
  we have 
  \begin{align} 
\frac{1}{\vol B_{S}(\ref{3const} T)^2} \int_{B_{S}(\ref{3const} T)}
  (\varphi_0^{\wk}(g_1 g_2^{-1}))^{\rho} \operatorname{d}\!g_1\operatorname{d}\!g_2\ll  T^{- \zeta_S \rho} \ll v^{-\frac{\zeta_S \rho}{2 A_S}}\end{align}
for suitable $\zeta_S, A_S> 0$ depending on $S$.

Apply Proposition \ref{sl32}, (2).
We conclude that:
  $$v^{1-\delta} \leq \bigl|\{s_i\}\bigr| \ll
\frac{v}{\sqrt{\vol(\Lambda \backslash {S})}} + v v^{-\zeta_S \rho/4A_S}$$
 In other terms, $ \vol(\Lambda \backslash {S})^{-1/2}  + v^{-\frac{\zeta_S \rho}{4A_S}} \gg v^{-\delta}$.  In particular, if $\delta < \zeta_S \rho / 4 A_S$ and $T$ is sufficiently large, we see that 
 $\vol(\Lambda \backslash {S}) \ll v^{2 \delta}$.  In particular, $\Lambda$ must be a lattice in ${S}$: our conventions dictated that its covolume is $\infty$ if this is not the case. 

  This concludes the proof of \eqref{desid}, and therefore the Proposition.  
 \qed

We now establish some (simple) results that were used in the above proof. 
\subsection{A vector almost in the kernel of a rational matrix,
is near a vector in the kernel.}
In the following statement, $\|\cdot\|$ refers to the standard
Euclidean norm on $\R^n$ and $\R^m$. 
\begin{lem}\label{dioph}
Let $A \in M_{n \times m}(\Z)$ be an $n \times m$ integer matrix,
all of whose entries are $\leq E$ in absolute value. Suppose $v \in
\R^m$ with $\|A v\| \leq \delta$. Then there exists $v_0 \in
\mathrm{ker}(A)$ with $\|v - v_0\| \leq \delta (nm)^{n/2} E^n$.
\end{lem}

\begin{proof} There exists an orthonormal basis $v_1, \dots, v_m$ for
$\R^m$ with the property that $A v_i$ are orthogonal in $\R^n$ and
so that the lengths $\sigma_i  := \|A v_i\|$ are decreasing. The
$\sigma_i$ are the singular values and $\sigma_i^2$ are the eigenvalues of 
$A A^{t}$ (see \S \ref{eucl}). 

The matrix $A. A^{t}$ is an $n \times n$ matrix, all of whose
entries are integers $\leq mE^2$ in absolute value. The absolute
value of any eigenvalue is thereby bounded by $nm E^2$. Moreover,
the eigenvalues are algebraic integers; in particular, the product
of all their algebraic conjugates is a nonzero integer. It follows
from this that the absolute value of any nonzero eigenvalue of $A.
A^{t}$ is bounded {\em below} by $(nmE^2)^{-n}$. Therefore, each
nonzero $|\sigma_i| \geq (nm E^2)^{-n/2}$.

Now take $v \in \R^m$ so that $\| A v\| \leq \delta$. The projection
$v'$ of $v$ onto the span of all $v_i$ with $Av_i \neq 0$ therefore has length
$\leq \delta (nm)^{n/2} E^{n}$. Set $v_0 = v-v'$.
\end{proof}

\subsection{Existence of an equivariant projection.}\label{eqproj}
Notations as in the proof. We shall show the existence of an equivariant
projection from a fixed small neighbourhood of $v_S$, onto $G. v_S$. This suffices
for the argument in the text, although one can do much better, see \cite{Luna}. 

The map $g \mapsto g. v_S$ induces a map on tangent spaces and in particular
a map $\mathfrak{g} \rightarrow V_{\R}$. Since $S$ is semisimple, we may choose an $S$-invariant complement
$W$ to the image of $\mathfrak{g}$.  Let $G \times_{S} W$ be the quotient of pairs $(g,w)$
by the equivalence $(gs, s^{-1} w) \sim (g,w)$; it is a vector bundle over $G/S$ and so a manifold. Consider the map $G \times_{S} W \rightarrow V_{\R}$,
given by $(g, w) \mapsto g. (v_S+w)$. The differential of this map is an isomorphism at $(g,w)=(1,0)$
and therefore, by the implicit function theorem, it defines
a diffeomorphism from a neighbourhood $\mathcal{N}_{1} \in G \times_S W$
of $(1,0)$, to a neighbourhood $\mathcal{N}_2$ of $v_S$.  In this neighbourhood there exists a projection that
is, at least locally, equivariant, which is given by
\begin{equation} \label{proj} \Pi: g(v_S +w) \in \mathcal{N}_2
 \mapsto g v_S. \end{equation}

This crude construction already suffices for our simple application. Indeed,
it implies the following weak Lemma, which is easily seen to imply our application:
\begin{lem}\label{lem-eqproj}
There exists a constant $\consta\label{projection-const}>0$%
\index{xappa8@$\ref{projection-const}$, Equivariant projection, Lemma \ref{lem-eqproj}}
with the following property.
Let $v \in \mathcal{N}_2$ and $\delta\in\Gamma$.  
If $\delta v = v$, and $\|v-v_S\| \leq \|\delta\|^{-\ref{projection-const}}$, 
then also $\delta \Pi(v) = \Pi(v)$.
\end{lem}
\proof
Choose $(g,w) \in \mathcal{N}_1$ so that
$g(v_S+w) = v$.  Observe that $\|g v_S - v_S\|, \|g w\| \ll
\|v_S - v\|$.

Then $\Pi(v) = g v_S$. Since $(\delta v - v) = (\delta g v_S - v_S)
+ (v_S - g v_S)
+ (\delta g w - g w)$, we get
\begin{equation} \label{close}\| \delta g v_S -  v_S \| \leq \|g v_S - v_S\| + \| \delta (gw) - (gw)\| \ll
\|\delta\|^{\star} \|v_S - v\|.\end{equation}

From this, we see that $(\delta g, w) \in \mathcal{N}_1$ if $\ref{projection-const}$ was chosen
big enough.
(Indeed, $(\delta g, w) = (\delta g s, s^{-1} w)$ in $G \times_S W$.
The assumption $\|v - v_S \| \leq \|\delta\|^{-\ref{projection-const}}$ and \eqref{close} 
ensure that $(\delta gs, s^{-1} w)$ is close to $(1,0)$
for a suitable choice of $s$.)

However, $(\delta g, w)$ and $(g, w)$ both map to $\delta v=v$
under the diffeomorphism from $\mathcal{N}_1$ to $\mathcal{N}_2$. Therefore, $\delta g S = g S$, i.e. $\delta \Pi(v) = \Pi(v)$. 
\qed

%In particular, for $v$ in a sufficiently small neighbourhood $\mathcal{N}$ of $v_S$, there is a unique
%(equivalence class) $[g,w]$ so that $g (v_S+ w) = v$.  Then the map $\pi: v \mapsto g. v_S$
%defines an equivariant projection, in the following sense: there are neighbourhoods
%$\mathcal{N}' \subset \mathcal{N}$ and $\Omega \subset G$ so that
%$\pi(g. v) = g. \pi(v)$ for $v \in \mathcal{N}, g \in \Omega$.
%

%Let $\Omega_W$ be a small neighbourhood in $W$, and consider the map $\Omega_{\mathfrak{r}} \times \Omega_W \rightarrow V_{\R}$ given by $(r, w) \mapsto \exp(r). (v_S+w)$. Computing derivatives
%at the identity shows it to be a diffeomorphism onto a neighbourhood $\mathcal{N}$ of $v_S$;
%in particular, for each $v \in \mathcal{N}$ there exists unique $(r,w)$ so that $\exp(r) . (v_S+w) = v$. The map $v \mapsto \exp(r). v_S$ gives a map $\mathcal{N} \rightarrow G. v_S$, which is $G$-equivariant
%
%given a bounded neighbourhood $\Omega_G \subset G$,
%then for a sufficiently small neighbourhood $\Omega_W \subset W$, the map
%$\pi: \Omega_G \times \Omega_W \rightarrow V_{\R}$ given by $(g, w) \mapsto g. (v_S+w)$
%is a submersion onto an open neighbourhood $\mathcal{N}$ of $\Omega_W. v_S$, as follows from the implicit function theorem. The map $\pi(g, w) \mapsto g. v_S$ is easily seen to define
%a well-defined equivariant projection $\mathcal{N} \rightarrow G. v_S$.

\subsection{Effective noetherian arguments}
\label{Arg:EffectiveNoetherian}
\hypertarget{Arg:EffectiveNoetherian}{} Set as before $X(F) = \{g
\in G: F \subset g \tilde{S} g^{-1}\}$. We shall
show that, given any set $F_1$, there exists a subset $F \subset F_1$
of size bounded only in terms of $G, S$, 
so that $X(F_1) = X(F)$. 

Indeed, notations as before, we have:
$$X(\{\alpha\}) = \{g \in G: \alpha g v_S  = g v_S\}
= \{g \in G: g v_S \in \mathrm{Fix}(\alpha)\}$$
where $\mathrm{Fix}(\alpha)$ is the fixed locus of
$\alpha$ acting on $V$. 

Given a vector space of dimension $d$, and a collection
of linear subspaces, any intersection of these subspaces
can in fact be expressed as the intersection of at most $d$ of them. 
In particular, there exists $F_1 \subset F$ of size $\leq \dim V$
so that:
$$\bigcap_{\alpha \in F_1} \mathrm{Fix}(\alpha) = \bigcap_{\alpha \in F} \mathrm{Fix}(\alpha).$$ 
This implies that $X(F_1) = X(F)$, as desired. 
\qed

%%%%%%%%%%%%%%%%%%%%%%%%%%%%%%%%%%%%%%%%%%%%%%%%%%%%%%%%%%%%%%%%%%%%%%%%%%%%%%%%%
 \section{A corollary of the closing Lemma.} \label{closing}

 \myparagraph

In the previous section, we presented a version of a closing lemma
for actions of semisimple groups. In the present section, we shall apply
it to the setting needed in this paper, proving the following (``Proposition C''
from \S \ref{proof}).  $\mu$ is as in Theorem \ref{thm:main}. 

 \begin{prop}\label{prop:dichotomy}
Given a connected intermediate subgroup $H\subseteq S\subseteq G$,
$\mysymbol \in(0,1)$, and $d_S \geq 1$, there exists $\xi$ and $d'$ 
depending only on $d_S$, $\mysymbol$, $G$, $H$, and $\epsilon_0
\ll_{d_S,\zeta} 1$, so that:

 Suppose for some $\epsilon \leq \epsilon_0$ that
\begin{enumerate}
\item  $\mu$ is $\epsilon$-almost invariant under $S$, with respect to the Sobolev norm $\Sob_{d_S}$, and that 
\item  $\mu(x S) = 0$ for all closed $S$-orbits of volume $\leq \epsilon^{-\mysymbol}$.
\end{enumerate}
Then there exists $x_1, x_2$ so that $x_2 = x_1 \exp(r)$,
$r \in \mathfrak{r}$, $\|r\| \leq \epsilon^{\xi}$, and
 $x_1,x_2$ are both
$[\|r\|^{-\ref{gammaconstant}},
\|r_1\|^{-\ref{deltaconstant}}]$-generic w.r.t.\ $\Sob_{d'}$. 

%  If $\mu$ is $[S, R, d_S]$-almost invariant for some $R\geq
%R_0$, then at least one of the following holds:
%\begin{enumerate}
% \item \label{closedorbitproperty} $\mu$ is supported on a closed $S$-orbit of volume $\leq R^{\mysymbol}$.
% \item  $\mu$ is $[S_*, R^{\xi}, d_{S}+\ref{rtconstant}]$-almost invariant \end{enumerate}
% There exists $T_0 = T_0(\epsilon, N, G, \Gamma,H)$
% and $f_0=f_0(G,H)\in (0,1)$ with the following property for any of the
% intermediate subgroups $H\subseteq S\subseteq G$.

%
% Let $x \in \Xcompact$, let $T  > T_0$
% and let $\mathcal{B} \subset Sball^{T}$ be a subset satisfying
% $\frac{\vol(\mathcal{B})}{\vol Sball(T)} > f_0$ and $x \mathcal{B} \subset \Xcompact$.
%  Then one of the following statements hold:
% \begin{enumerate}
% \item There exists $b_1, b_2 \in \mathcal{B}$
% so that $xb_1 = x b_2 \exp(r)$ for some $r \in \Scomp$, $T^{-N}\leq \|r\| \leq T^{-\epsilon} $.

% 
% \item There exists $x' \stackrel{T^{-N_\uparrow}}{\sim} x$ such that $x' S$ is closed
% with volume $\leq T^{\epsilon_{\downarrow}}$.
\end{prop}

We strongly suggest the reader glance at the description of this argument,
provided in \S \ref{subsec:CE}, 
prior to reading what follows. Also recall that $r_1$ stands for the non-fixed components of $r$ in the splitting defined in \eqref{gdecomp}.

\subsection{Beginning of the proof}
In this proof, $\epsilon$ {\em sufficiently small}
will mean that $\epsilon$ is bounded above by a constant
$\constc(d_S,\mysymbol)$. In view of the formulation of the Proposition,
we are free to assume $\epsilon$ is sufficiently small.

Let $\beta$ as in Proposition A (\ref{lemma3}) with $d = d_S$.  
Let $A_S$ be as in \eqref{volest}. 

Let $N_{\uparrow}, \delta_{\downarrow}$ be the functions defined in Proposition \ref{CaseIIlemma}.  Take $N$ so large that 
$N_{\uparrow} \geq 2 A_S \ref{expQuantiso}$ where $\ref{expQuantiso}$ is as in 
\propGname (\ref{lem:Quantiso}).  
Choose $\delta$ so small that $\delta_{\downarrow} \leq A_S/2$. 
Let $q$ be chosen so small that:
\begin{itemize}
\item $q \leq \beta$; 
\item $2 q N  \ref{deltaconstant} < \beta$. 
\item $2 q A_S \leq \zeta$. 
\item $q d_S \ref{firstlemconstant} < \frac{1}{2}$. 
\end{itemize} 
Finally, we put $\xi = \frac{\delta q A_S}{5 \dim(G)}$. 
These choices of $N, \delta, q$ depend only on $G,H, \zeta, d_S$.  
We set $T = \epsilon^{-q}$ and $v = \vol(\Sball(T))$.  Our constraints imply that,
for $\epsilon$ sufficiently small, we have 
\begin{equation} \label{constraints} T^{N_{\uparrow}} \geq 
v^{\ref{expQuantiso}}, \
T^{\delta_{\downarrow}} \leq v, \
 v^{-\frac{\delta}{2\dim(G)}} \leq \epsilon^{2 \xi}, \ T \leq \epsilon^{-\beta},
\ v \leq \epsilon^{-\zeta}.
\end{equation}

The reader should not pay too much attention to the mass of constants above:
simply, $N$ and $\delta$ are {\em fixed}; $\epsilon$ is {\em very small},
and $v^{-1}$ is {\em small}.  Set

\begin{equation}\label{eprime}
 E_1 = \bigl\{x : \mbox { there is $x' \stackrel{v^{-\ref{expQuantiso}}}{\sim} x$ such that $x' S$ is closed of volume $\leq v$
}\bigr\}.
\end{equation}
Under our assumption, \propGname
(\ref{lem:Quantiso}) furnishes an upper bound for the $\mu$-measure
of $E_1$. In fact,
\begin{equation}\label{e1bound}
 \mu(E_1)\leq 1/2 \mbox{ for } v\in [V_0(\Gamma, G, H), \epsilon^{-\mysymbol}].
\end{equation}
Note that the condition on $v$ will be satisfied if $\epsilon$
is sufficiently small. 
% with the exponent $q$ being determined in the
%remainder of the proof. Obviously we will require $q>0$ and as we go
%along more and more constraints will follow. E.g.\ we will use
%\eqref{e1bound} for $V=R^{q\zeta}$ which will be allowed if $q\leq
%1$ and $R$ is big enough --- in the remainder of this proof the
%latter is always meant as large enough in relation to $d_S$,
%$\Gamma$, $G$, $H$. This gives one of many constraints on $R_0$.

Apply Proposition~A~(\ref{lemma3}) with $d = d_S, R=T$. 
It produces points generic with respect to a new Sobolev norm
$\Sob_{d'}$; generic is always understood in that sense, in what follows.
More precisely, Proposition \ref{lemma3} asserts that the fraction of pairs $(x \in X, s \in \Sball(T))$ so that
$x.s$ is not
$[T_0, \epsilon^{-\beta}]$-generic is $\ll_{d_S} T_0^{-1}$. We choose $T_0$
large enough so that this fraction of points is
$<\frac{1}{10^{10}}$. This choice of $T_0$ depends only on $\Gamma, G, H, d_S$. 

Next, we observe that:
\begin{equation} \label{observation}
\mbox{The fraction of $(x \in
X, s \in \Sball(T))$ for which $x.s \notin \Xcompact$ is $\leq
\frac{1}{10^{10}}$}\end{equation}
 if $\epsilon$ is sufficiently small. Here ``fraction''
is measured w.r.t. $\mu \times \vol_S$. Indeed, recalling
that $\Xcompact=\Siegel(R_0)$, take a smooth function $F$ so that:
$$1_{X-\Siegel(R_0/2)} \geq F \geq 1_{X-\Siegel(R_0)}$$
Invoking Lemma \ref{one} and the almost invariance, 
$$|\int F(xs) d\mu(x) - \int F d\mu| \ll
\epsilon T^{d_S \ref{firstlemconstant}}
\Sob_{d_S}(F)\mbox{ whenever }s \in \Sball(T).
$$
Thus 
$$\frac{\int_{s \in \Sball(T), x \in X} F(xs)}{\vol \Sball(T)}
- \mu(X-\Siegel(R_0/2)) \ll \epsilon^{1-q d_S \ref{firstlemconstant}}
\Sob_{d_S}(F).$$
 $F$ may be fixed in a fashion depending only on
$\Gamma, G$. Moreover, $1 - q d_S \ref{firstlemconstant} > 1/2$. Therefore, if $\epsilon$
sufficiently small, the observation \eqref{observation} follows. 

For each $x \in X$, let
$$
f(x) = 
\frac{1}{\vol(\Sball(T))} \vol\bigl(\bigl\{s \in \Sball(T): xs \in \Xcompact, xs 
\mbox{ is $[T_0,\epsilon^{-\beta}]$-generic}
\bigr\}\bigr)
$$  

The function $f$ takes
values in $[0,1]$. In view of our remarks above, $\int (1-f(x))
\operatorname{d}\!\mu \leq \frac{2}{10^{10}}$. Thus, the set $E_2 =
\{x \in X: f(x) < 1 - 10^{-6}\}$ satisfies $\mu(E_2) <
\frac{1}{10}$.
Let $\Xgood = \Xcompact - E_1 - E_2$. 

The set $\Xgood$ is a $\mu$-nonempty set and has the following properties:
\begin{enumerate}
\item $\Xgood \subset \Xcompact$. 
\item For any $x \in \Xgood$, the set 
\begin{equation} \label{bdef}\mathcal{B}_x
=\bigl\{s \in \Sball(T):
\mbox{$x s$ is $[T_0, \epsilon^{-\beta}]$-generic and belongs to $\Xcompact$}\bigr\}
\end{equation}
has measure larger than $(1-10^{-6}) \,\vol \Sball(T)$. 
\item For any $x \in \Xgood$,
there does not exist $x' \stackrel{v^{-\ref{expQuantiso}}}{\sim} x$ such that $x' S$ is closed of volume $\leq v$. 
\end{enumerate}

Let $x \in \Xgood, \mathcal{B}_x$ as in \eqref{bdef}. 
\begin{lem*}  
 There exists $b_1, b_2 \in \mathcal{B}_x$
 so that $xb_1 = x b_2 \exp(r)$ for some $r \in \Scomp$, where
 $v^{-\frac{\delta}{2\dim(G)}}
\gg \|r\| \gg T^{-N}$.
 Moreover, $\|r_1\| \gg \|r\|$.
\end{lem*}

The Proposition follows from this, as we now explicate:

 By choice, 
$v^{-\frac{\delta}{2\dim(G)}} \leq \epsilon^{2 \xi}$. 
Then $\|r\| \leq \epsilon^{\xi}$, for $\epsilon$ sufficiently small. 
Let us verify that $$[\|r\|^{-\ref{gammaconstant}},\|r_1\|^{-\ref{deltaconstant}}]  \subseteq
[T_0,\epsilon^{-\beta}].$$ 
Since $T_0 \ll_{d_S} 1$,  the inequality  
$\|r\|^{-\ref{gammaconstant}}\geq T_0$ holds if $\epsilon$ is sufficiently small. The other inequality $\|r_1\|^{-\ref{deltaconstant}}\leq \epsilon^{-\beta}$
follows for sufficiently small $\epsilon$ because $\|r_1\|^{-\ref{deltaconstant}}
\ll \epsilon^{-q N \ref{deltaconstant}}$, and $q N \ref{deltaconstant} \leq \beta/2$.

\subsection{Proof of the lemma}
Let us recollect the situation. We are given $x \in \Xcompact$ and
a subset $\mathcal{B}_x \subset \Sball(T)$ satisfying
$\frac{\vol(\mathcal{B}_x)}{\vol \Sball(T)} \geq 1-10^{-6}$,
and we know
that there does not exist $x' \stackrel{v^{-\ref{expQuantiso}}}{\sim} x$ such that $x' S$ is closed of volume $\leq v$.  

We are free to prove the Lemma for $T$ sufficiently large:
here and in the course of the proof,
 the phrase ``for $T$ sufficiently large''
to mean ``for $T$ larger than a constant that may depend on $\delta$, $N$, $\Gamma$, $G$, and $H$.'' Indeed, we can guarantee $T$ sufficiently large by taking $\epsilon$ sufficiently small. 

Take
{\em sufficiently small} symmetric neighbourhoods $\Omegatwo$, and $\Omegaone$ of the identity in $S$, resp.\ $\Omega_{\mathfrak{r}}$ of $0\in\mathfrak{r}$ such that:
\begin{equation} \label{omega31def} \Omegatwo\subset\Omegaone \subset\{s \in S: d(s,e) \leq 1/4\} \subset S. \end{equation}
The precise notion of {\em sufficiently small} will be specified in the course of the argument; however, it  will depend only on $\Gamma, G, H$ and therefore
constants that depend on $\Omegatwo, \Omegaone, \Omega_{\mathfrak{r}}$
will be absorbed into $\ll$ notation.

Set \begin{equation}\label{bprime} \mathcal{B}  = \{s \in \mathcal{B}_x: \vol(s \Omegaone \cap \mathcal{B}_x)
\geq 0.99 \vol(\Omegaone)\}\end{equation}
We shall verify that, for $T$ sufficiently large,
  we have also 
\begin{equation} \label{Minvolequation}\vol(\mathcal{B})
\gg  \vol(\Sball(T)) (=v).\end{equation}
If we set $f(s) =  \frac{\vol(s \Omegaone \cap \mathcal{B}_x)}{\vol(\Omegaone)}$, then
$\int_{\mathcal{B}_x} f(s) \operatorname{d}\!s=(\vol(\Omegaone))^{-1} \int_{\Omegaone} \vol(\mathcal{B}_xs \cap \mathcal{B}_x)\operatorname{d}\!s$, which exceeds 
 $$\vol(\Omegaone)^{-1}\int_{\Omegaone} \bigl(\vol(\Sball(T) s \cap \Sball(T))
- 10^{-5} \vol(\Sball(T))\bigr)\operatorname{d}\!s.$$
By property (3) of \S \ref{intermediateballs}, there is a constant $c<1$ so that
$\Sball(cT) \subset \Sball(T) s \cap \Sball(T)$ for any $s\in S$ with $d(s,e)\leq 1/4$.  
In particular, using \eqref{volest},
property (3b) of \S \ref{intermediateballs}, and choosing $\Omegaone$ sufficiently small, we can arrange that
\begin{equation} \label{omegachoice}
 \vol\bigl((\Sball(T) s \cap \Sball(T)\bigr) \geq (1-10^{-5}) \vol \Sball(T)
\end{equation}
for any $s\in\Omega$, at least for $T \gg 1$. 
Thus, $\int_{\mathcal{B}_x} f(s)\operatorname{d}\!s > (1-10^{-4}) \vol \Sball(T)$;
in particular, the set of $s \in \Sball(T)$ for which $f(s) > 0.99$
has volume at least $\frac{1}{2} \vol \Sball(T)$.

%In view of the hypothesis in the Proposition that $x \mathcal{B} \subset \Xcompact$, we
%note that $x s_i \in \Xcompact \Omega^{-1}$ for each $i$
%that has not been discarded.
For $y \in X$, let
$$N(y) :=  \vol\bigl(\bigl\{s \in \mathcal{B}: x s \Omega \cap B(y, T^{-N}/2) \neq\emptyset \bigr\}\bigr)$$
Loosely speaking, $N(y)$ measures the number of ``times'' in $\mathcal B$ for which the corresponding
point in $x \mathcal{B}$ comes within $T^{-N}$ of $y$. We consider two separate cases:
\begin{enumerate}
\item[Case I]: For every $y \in X$, 
$N(y) \leq v^{1-\delta/2}$.
\item[Case II]: There is $y \in X$ with
$N(y)  \geq v^{1-\delta/2}$.
\end{enumerate}

Let us show that Case II cannot occur by using that $x$ is not close to a closed orbit of small volume. Suppose that Case II occurred; let $y_0 \in X$
be so that $N(y_0) \geq v^{1-\delta/2}$. 
Choose a maximal $1$-separated subset $\{s_i\}_{1 \leq i \leq I}$ of
$$\bigl\{s \in \mathcal{B}: x s \Omega\cap B(y_0,T^{-N}/2) \neq \emptyset\bigr\}$$
Then $I \gg v^{1-\delta/2}$.
In explicit terms,  there are
 $\omega_i \in \Omegaone \ (i \in I)$ so that $$x s_i \omega_i \stackrel{T^{-N}}{\sim} x s_j \omega_j \ (i,j \in I).$$

In particular, the elements $s_i' = s_i . \omega_i$ are $1/2$-separated\footnote{Indeed, $d(s_i, s_i') \leq 1/4$ by the left-invariance of the metric,
and $d(s_i, s_j) \geq 1$ by choice.}
 and belong to $\Sball(\constc T) $. In view of our assumptions on $\delta, N$,
 \propHname (\ref{CaseIIlemma}) -- applicable so long as $\epsilon$ is sufficiently small -- would show there is $x' \stackrel{v^{-\ref{expQuantiso}}}{\sim}x$
 so that $x' S$ is a closed orbit of volume $\leq v$; but that contradicts our assumption on the point $x$. 

%  \footnote{For the present version of the argument, it will suffice that
%  $\Omegatwo = \Omegatwo^{-1} \subset  \Omegaone$, that $\Omegatwo \times \exp(\Omega_{\mathfrak{r}})$
% belongs to the ``injectivity ball'' at every point $x \in X$, and that
% for any $\sigma, \sigma' \in \Omegatwo$ we have
% $\vol(\Omegaone \cap \sigma^{-1} \sigma' \Omegaone) > 0.9 \vol(\Omegaone)$. \Omegathree \subset Moreover, $\Omega_{\mathfrak{r}}$
% should have the property specified in footnote \ref{irrit}.}

Therefore, we are in Case I. It will be convenient to pass to a local coordinate system. Clearly,
there is $x' \in \Xcompact$ so that:
$$\vol\{s \in \mathcal{B}: x s \in x' \exp(\Omega_{\mathfrak{r}})  \Omegatwo\} \gg v.$$
Indeed, we may  cover $\Xcompact$ by finitely many neighborhoods of the
 form $ x'\exp(\Omega_{\mathfrak{r}})  \Omegatwo$; and, by assumption,
$x \mathcal{B} \subset \Xcompact$.

For each $s \in \mathcal{B}$ so that $x s$ belongs to our chosen neighbourhood,
we may write:
$$ x s =x' . \exp(r_s). \sigma_s, \ \  r_s \in \Omega_{\mathfrak{r}}, \sigma_s \in \Omegatwo.$$

Let $B$ be any metric ball of radius $T^{-N}/4$ in the Euclidean space $\mathfrak{r}$, so that $B \cap \Omega_{\mathfrak{r}} \neq \emptyset$. 
We claim that the preimage of $B$,
under the map $s \mapsto r_s$, has measure $\leq v^{1-\delta/2}$.
Were this false, there exists a subset $\mathcal{B}' \subset \mathcal{B}$
of volume $> v^{1-\delta/2}$,
so that $x \mathcal{B}' \subseteq x' \exp(B) \Omegatwo$. If $\Omega_{\mathfrak{r}}$
is sufficiently small, $x' \exp(B) \subset B(x'', T^{-N}/2)$,
for a suitable $x'' \in x'\exp(B)$. 

In particular, $N(x'') > v^{1-\delta/2}$, in contradiction to the assumption in Case I.

  Cover $\Omega_{\mathfrak{r}}$ by $T^{-N}/4$-balls
so that each one overlaps with $O(1)$ others. The previous paragraph
 shows that
$\gg v^{\delta/2}$ of these balls contain a point of the form $r_s$, for some $s \in \mathcal{B}$. Thus there is $s_1, s_2$ so that
$T^{-N} \ll \|r_{s_1} - r_{s_2} \| \ll v^{-\delta/(2\dim(\mathfrak{r}))}$.
Thus,
\begin{equation} \label{yellow_ribbon} xs_1 = x' \exp(r_{s_1}) \sigma_1, \ \ x' s_2 = x \exp(r_{s_2}) \sigma_2, \ \
T^{-N} \ll \|r_{s_1} - r_{s_2} \| \ll v^{-\delta/(2 \dim(\mathfrak{r}))}
\end{equation}
and $\sigma_1, \sigma_2 \in \Omegatwo \subset S$.

Using \eqref{yellow_ribbon}, we shall now perturb $s_1$ and $s_2$ along $S$ slightly to
find $b_1,b_2\in\mathcal{B}_x$ as in the first claim of the lemma.
For $i \in \{1,2\}$, let $W_i=\{s\in\Omegaone:s_is\in\mathcal{B}_x\}$. Then $\vol(W_i)>0.9\vol(\Omegaone)$ by \eqref{bprime} and the surrounding discussion. Our
(intuitively obvious) conclusion follows from the following

\begin{sublem}\label{midlemma}
Suppose we are given $\sigma_i \in \Omega'$, $r_i' \in \Omega_\mathfrak{r}$,
and subsets $W_i \subset \Omegaone$ of measure $\vol(W_i) > 0.9 \vol(\Omegaone)$ for $i=1,2$. 
Then there exists $w_i \in W_i$ for $i=1,2$ so that
$$  \exp(r_1') \sigma_1 w_1=  \exp(r_2') \sigma_2 w_2 \exp(r), \ r \in \mathfrak{r}, \|r\| \asymp \|r_1'-r_2'\|, \|r_1\| \gg \|r\|.$$
\end{sublem}

This statement easily implies the Lemma:
First of all, $b_i= s_i w_i \in \mathcal{B}_x$. 
Therefore, $x_i := x b_i$ satisfy the constraints of the Lemma
(i.e.\ $x_2 = x_1 \exp(r)$, where $v^{-\frac{\delta}{2 \dim(G)}} \gg \|r\| \gg T^{-N}$.)

\subsection{Proof of the sublemma.} 
Define functions $\phi_S \times  \phi_\mathfrak{r}(s): \Omegaone \rightarrow S \times \mathfrak{r}$ according to:
%by writing the expression $(\exp(r_2) \sigma_2)^{-1} (\exp(r_1) \sigma_1) s$
%in the coordinate system introduced by the local diffeomorphism $S\times \mathfrak{r}\rightarrow G$. More formally:
\begin{equation}\label{diffeo}
 \bigl(\exp(r_2')\sigma_2\bigr)^{-1}\exp(r_1')\sigma_1 s=\phi_S(s)\exp(\phi_\mathfrak{r}(s)), \ \ \phi_S(s) \in S, \phi_{\mathfrak{r}}(s) \in \mathfrak{r}.
\end{equation}
Strictly speaking, the maps $\phi_{S}$, $\phi_{\mathfrak{r}}$ are functions
of $r_1'$, $r_2'$, $\sigma_1$, $\sigma_2$, and  $s$; but we shall suppress the dependence on the first four variables, which will be fixed throughout the proof.

If $\Omegaone$, $\Omega_\mathfrak{r}$, and $\Omegatwo$ are sufficiently
small, the expression on the left of \eqref{diffeo} is ``sufficiently close'' to the origin. This implies that, if we take the three sets $\Omegaone, \Omegatwo,
\Omega_{\mathfrak{r}}$
sufficiently small, $\phi_S$ and $\phi_{\mathfrak{r}}$ will be defined and smooth.  In that statement, the notion of ``sufficiently small'' depends only on $G,H$.

We require that $\Omegaone$, $\Omegatwo$, and $\Omega_{\mathfrak{r}}$ are so small that:
\begin{itemize}
\item
 $\phi_S$ is injective,
\item
The preimage $\Omegathree := \phi_S^{-1}(\Omegaone)\subset\Omegaone$ of $\Omegaone$
has volume
$\vol(\Omegathree)>0.9\vol(\Omegaone)$.
\item
 The Jacobian of $\phi_S$ is everywhere on $\Omegaone$ between $0.9$ and $1.1$, i.e.\ the map $\phi_S$ almost preserves
 volume.
\end{itemize}
This is possible since the dependence of $\phi_S$ on $s$ and the parameters
 $r_1'$, $r_2'$, $\sigma_1$, $\sigma_2$
is smooth and since $\phi_S(s)=s$ if the latter parameters are trivial. Clearly these restrictions
only depend on $S$.

These conditions imply $\vol(\phi_S(W_1\cap\Omegathree))>0.5\vol(\Omegaone)$. Therefore, $W_2\cap
\phi_S(W_1\cap\Omegathree)$ must be nonempty. Take $w_1\in W_1$ with
$\phi_S(w_1)=w_2 \in W_2$. 
We may also assume that $w_1$ is, in a weak sense, a point of density for this set: $\phi_S^{-1}(W_2) \cap W_1$
intersects $w_1 \Omegatwo$ in a set of volume bounded below.

 Then:
$$
  \exp(r_1')\sigma_1 w_1
= \exp(r_2')\sigma_2 w_2
 \exp(\phi_\mathfrak{r}(w_1)).$$
by \eqref{diffeo}.

%If $r_1=r_2$, then $\phi_\mathfrak{r}(w_1)=0$.
%By smothness this shows $\|\phi_\mathfrak{r}(w_1)\|\ll\|r_1-r_2\|$.

We claim $\|\phi_\mathfrak{r}(w_1)\| \asymp \|r_1'-r_2'\|$.
To see that, rearrange \eqref{diffeo} as follows:
 \begin{multline*}
\exp(r_2')^{-1}\exp(r_1')=\\
\sigma_2\phi_S(w_1)\exp(\phi_\mathfrak{r}(w_1))(\sigma_1 w_1)^{-1} \in S \exp(\mathrm{Ad}(\sigma_1 w_1)
\phi_{\mathfrak{r}}(w_1))
\end{multline*} 
It will therefore suffice to prove that, if we express
$\exp(r_2')^{-1} \exp(r_1')$ in the form $s \exp(X) \, (s \in S, X \in \mathfrak{r})$, we have
the majorization $\|X\| \gg \|r_1'-r_2'\|$ for sufficiently small $\Omega_{\mathfrak{r}}$. 
This has already been seen: \eqref{expconstdef}.

It remains to check that 
$\|\phi_{\mathfrak{r}}(w_1)_1\| \gg
\|\phi_{\mathfrak{r}}(w_1)\|$.  We will do this
after possibly modifying $w_1$, using the ``point of density'' assumption.
For $\sigma \in S$ near to the identity, we have
$\phi_{\mathfrak{r}}(w_1 \sigma) = \mathrm{Ad}(\sigma^{-1})
\phi_{\mathfrak{r}}(w_1)$; replacing $w_1$ by a nearby
$w_1 \sigma \in \phi_S^{-1}(W_2) \cap W_1$,
 and applying a suitable variant
of \eqref{iotaclaim}, we conclude.

% This follows from the implicit function theorem: the map $(r_1', r_2') \mapsto (r_1',X)$ is a smooth map
%from $\Omega_{\mathfrak{r}} \times \Omega_{\mathfrak{r}}$ to $\Omega_{\mathfrak{r}}\times\mathfrak{r}$,
%and its derivative at zero is precisely $(r_1', r_2') \mapsto (r_1',r_1'-r_2')$, smoothness of the inverse function
%for small enough $\Omega_{\mathfrak{r}}$ gives the desired result. 

This concludes the proof of the Sublemma. \qed

%Choose $\epsilon$ so small that the term $\epsilon_{\downarrow}$ in
%the exponent of (\ref{second}) is less than $\zeta$. Similarly, we
%choose $N$ so big that the term $N_{\uparrow}$ in the other exponent
%is bigger than $\ref{expQuantiso}$. With these choices we know that
%(\ref{second}) cannot take place since $x\notin E_1$. Note that the
%functions $\epsilon_\downarrow$ and $N_\uparrow$ appearing here do
%not depend on $T$ and so also not on $q$ (while as noted before the
%validity of one the claims does require $T=R^q$ being large enough).
%Therefore, the choices of $\epsilon$ and $N$ depend only on
%$\Gamma$, $G$, and $H$.

%To summarize, if $R^q$ is big enough it follows that (\ref{firstP})
%holds. We now define $q=\frac{\beta}{1+N \ref{deltaconstant}}$ where
%$\ref{deltaconstant}$ is as in Proposition
%\ref{lem:Moreinvariance-copy} (\ref{lem:Moreinvariance}). We note in
%particular, that this definition of $q$ depends only on $G$, $H$,
%and ... Moreover, $0<q<\beta<1$ which takes care of our earlier
%conditions on $q$.

%\qed

%%%%%%%%%%%%%%%%%%%%%%%%%%%%%%%%%%%%%%%%%%%%%%%%%%%%%%%%%%%%%%%%
\section{An almost invariant measure on a closed orbit is close to the invariant measure.}
We shall prove (``Proposition F'' from \S \ref{proof})

\begin{prop} \label{nearness}
Let $x_0 \in X$ be so that $x_0 S$ is a closed orbit of
volume $V$,  for some connected $S \supset H$.
 Suppose $\mu$ is a probability measure on $x_0 S$ that is
$\epsilon$-invariant under $S$ w.r.t.\ a Sobolev norm $\Sob_d$.
Let $\nu$ be the $S$-invariant probability measure on $x_0 S$.

   Then there are $\ref{V-constant}$ and $\ref{Rconstant}$%
\index{xappa3@$\ref{V-constant}$, exponent of $V$ in
Lemma~\ref{nearness-copy}~(\ref{nearness})}\index{xappa31@$\ref{Rconstant}$,
exponent of $R$ in Lemma~\ref{nearness-copy}~(\ref{nearness})} so
that
\begin{equation}\label{smoothness}
  \left| \mu(f)  - \nu(f) \right| \ll_{d}
V^{\ref{V-constant}} \epsilon^{\ref{Rconstant}/d}\Sob_d(f)   \mbox{ for }
   f \in C_c^{\infty}(X). \end{equation}
 In particular, there are constants $\ref{Aconstant}, \ref{Bconstant} > 0$%
\index{xappa32@$\ref{Aconstant}$, bound on $V$ in terms of $\epsilon$ in
Lemma~\ref{nearness-copy}~(\ref{nearness})}\index{xappa33@$\ref{Bconstant}$,
exponent for closeness in
Lemma~\ref{nearness-copy}~(\ref{nearness})}
 such that if $V\leq \epsilon^{-\ref{Aconstant}/d}$ then
 $\mu$ and $\nu$ are $\epsilon^{\ref{Bconstant}/d}$-close:
 \begin{equation*}
  \left|\mu(f)  - \nu(f) \right| \ll_{d} \epsilon^{\ref{Bconstant}/d}\Sob_d(f)   \mbox{ for }
   f \in C_c^{\infty}(X).
 \end{equation*}
\end{prop}

One cannot -- at least naively -- dispense with the occurrence of
$V$ in the statement. 

\proof Let $\chi \in C^{\infty}_c(S)$ be a fixed
probability measure supported in the ball $\{s \in S: \|s\| \leq 2\}$. 
%with integral $1$ with respect to the measure $\operatorname{d}\!\vol$ on $S$
%(see \S \ref{notation}). 
%
%In what follows, we shall identify the function
%$\chi$ with the measure $\chi \operatorname{d}\!\vol$. In particular, we may speak of
%the convolution $\chi \star \chi$, etc.

Denote by $L^2_0(\nu)$ the orthogonal complement of the constant functions. 
 By the spectral gap, furnished by Proposition \ref{automorphic},
we have for $F \in L_0^2(\nu_{x_0 S})$:
  $$\|F \star \chi\|_{L^2(\nu )} \leq (1-\delta) \|F\|_{L^2(\nu)}$$
  where $\delta > 0$ depends only on $\chi$ and the spectral gap for $S$
acting on $L^2(\nu_{x_0 S})$.  In view of the statement of Proposition \ref{automorphic}, which guarantees such uniform spectral gaps, we may regard (after having fixed a choice of $\chi$ for each subgroup $S$) $\delta$ as depending only on $G,H$.

Let $\chi^{(1)} = \chi$ and $\chi^{(n)} = \chi^{(n-1)} \star \chi$ for $n$ a positive integer.  Here $\star$ denotes
convolution on $S$. 
By the definition of ``almost invariance'' (\S \ref{def:almostinvariant}), we have 
$|\mu(F \star \chi) - \mu(F) | \leq
\epsilon  \Sob_d(F)$.
 In consequence, there exists $E_{\chi}>1$ so that: 
 %\marginpar{CHECK CHECK}
\begin{equation} \label{round}| \mu(F \star \chi^{(n)}) - \mu(F) | \ll
\epsilon  E_{\chi}^{nd} \Sob_d(F) \end{equation}

Now, by \eqref{injfact}, we see that that
the fibers of $x \mapsto xs$, considered
as a map $\{s \in S: \|s\| \leq 2\} \rightarrow X$, 
 have size $\ll \height(x)^{\dim(S) \ref{hc2}}$; thus
 $$|F \star \chi(x)| \ll
\height(x)^{\ref{hc2}\dim S} \int |F| d\mathrm{vol} \ll \vol(x_0 S)\height(x)^{\ref{hc2} \dim(S)} \|F\|_{L^2(\nu)}.$$
Applying this to $F \star \chi^{(n-1)}$, we obtain:
\begin{equation}\label{andround}|F \star \chi^{(n)}(x)| \ll \vol(x_0 S) \height(x)^{\ref{hc2} \dim(S)} (1-\delta)^n \|F\|_{L^{\infty}}.\end{equation}

 We combine \eqref{round}, \eqref{andround}, and the fact $| F \star\chi^{(n)}(x)| \leq \|F\|_{L^{\infty}} \ll \Sob_d(F)$.  It results
 $$\mu(F)  \ll \Sob_d(F) \left( \epsilon E_{\chi}^{nd} +  \int_{x_0 S} \min(1,
\vol(x_0 S) (1-\delta)^n \height(x)^{\ref{hc2} \dim(S)}) d\mu \right) $$
To estimate the bracketed quantity on the right-hand side,
we split the integral into $x_0 S \cap \Siegel(R_1)$ and its complement,
where $R_1$ is a parameter that will be optimized.
 In view of Lemma \ref{measureoutsidecompact}, the bracketed quantity
on the right-hand side is bounded by
$$ \epsilon E_{\chi}^{nd}  + \vol(x_0 S)
(1-\delta)^n R_1^{\ref{hc2}\dim S}
+ \ref{m1constant} R_1^{-\ref{m2constant}}$$
We now choose the free parameters $R_1$ and $n$ so that the three quantities
$\epsilon E_{\chi}^{nd}$, $\vol(x_0 S)(1-\delta)^n R_1^{\ref{hc2}\dim S}$, and $R_1^{-\ref{m2constant}}$ are of comparable size. The stated result follows.
\qed

%%%%%%%%%%%%%%%%%%%%%%%%%%%%%%%%%%%%%%%%%%%%%%%%%%%%%%%%%%%%%%%%%%%%%%%%%%%%%%%%
\section{Proof of Theorem \ref{thm:main}}  \label{finalproof}

We shall say that $\mu$ is $[S, \epsilon,d_S]$-almost invariant if $\mu$ is
$\epsilon$-almost invariant under a connected intermediate subgroup $S$
w.r.t.\ a Sobolev norm $\Sob_{d_S}$. 

The precise analog to the non-effective proof, or of \S \ref{subsec:D},
would be to choose a ``maximal'' $S$ under which $\mu$ is almost invariant.
However, we find it clearer in this effective setting
to present instead the proof by iteration, i.e.\ building up, dimension by dimension, a larger and larger $S$ under which $\mu$ is almost invariant.

\subsection{Completion of the proof of Theorem \ref{thm:main}} \label{completeness}

\begin{lem*}
Suppose that $\mu$ is $[S,\epsilon, d_S]$-almost invariant.
There exists constants $\ref{epsconst}(d_S)$, 
$\ref{volumeconst1}$,
$\ref{volumeconst2}$,
 $\ref{moreinvarianceconstant}$, and $d_S'$%
 \index{xappa331@$\ref{volumeconst1}$, Lemma in \ref{ABC} and \ref{completeness}}%
\index{xappa332@$\ref{volumeconst2}$, Lemma in \ref{ABC} and \ref{completeness}}%
\index{xappa333@$\ref{moreinvarianceconstant}$, Lemma in \ref{ABC} and \ref{completeness}}%
so that for any $\epsilon$ sufficiently small (i.e. $\epsilon \ll 1$) either:
\begin{itemize} \item
$ |\mu(f) - \mu_{x_0 S}(f)| \leq \constc(d_S)
\epsilon^{\ref{volumeconst2}(d_S)} \Sob_{d_S}(f)$, 
for some closed orbit $x_0 S$ of 
volume $\leq \epsilon^{-\ref{volumeconst1}(d_S)}$.
\item The measure $\mu$ is 
$[S_* ,  \constc(d_S) \epsilon^{\ref{moreinvarianceconstant}(d_S)},
d_S']$
almost invariant, where the connected subgroup $S_* \supset H$ has larger dimension than $S$. 
\end{itemize}
\end{lem*}

Apply Proposition \ref{prop:dichotomy} with $\mysymbol = \ref{Aconstant}/d_S$. 
If it fails to be applicable, we are in the first case of the Lemma,
by Proposition \ref{nearness}. Indeed, we may take
$\ref{volumeconst1}(d_S) = \ref{Aconstant}/d_S$
and $\ref{volumeconst2}(d_S) = \ref{Bconstant}/d_S$. 

Otherwise, Proposition \ref{prop:dichotomy} produces $d_S' > d_S$
and $\xi$,
both of which depend only on $G,H, d_S$, and
two points $x_1, x_2$ so that
 $x_1,x_2$ are both
$[\|r\|^{-\ref{gammaconstant}},
\|r_1\|^{-\ref{deltaconstant}}]$-generic for $\mu$ w.r.t $\Sob_{d_S'}$
moreover $\|r\| \leq \epsilon^{\xi}$. 

Now we apply  Proposition \ref{MaL} with $\mu_1 = \mu_2 = \mu$. 
It shows that $\mu$ is $\ll_{d_S} \epsilon^{\min(1/2,\ref{eipconstant} \xi/2)}$-almost invariant
under an element $Z \in \mathfrak{r}$ which satisfies $\|Z\|=1$. 
In view of Proposition \ref{almostsubalgebra}, $\mu$
is $\ll_{d_S} \epsilon^{\ref{almostinvconstant} \min(1/2, \ref{eipconstant} \xi/2)}$-
almost invariant under a subgroup $S_* \supset H$ of strictly larger dimension. 

The proof of the Lemma is complete. 

\subsection{Proof by iteration}
We now prove Theorem \ref{thm:main} by iterating the Lemma.
Now, by definition, $\mu$ is $[H, \epsilon, d_H]$-almost invariant for
arbitrary $\epsilon$ and $d_H = \ref{linftyconst}+1$. By the lemma, for small enough $\epsilon$ either:
\begin{enumerate}
\item $|\mu(f) - \mu_{x_0 H}(f)| \ll \epsilon^{\ref{volumeconst2}(d_H)}
 \Sob_{d_H}(f)$, and
 $x_0 H$ has volume $\leq \epsilon^{-\ref{volumeconst1}(d_H)}$.\footnote{In fact, in the case, $\mu = \mu_{x_0 H}$; we write
it in a way that will resemble the other steps of the iteration.}
\item  $\mu$ is $[S_1, c_1 \epsilon^{\xi_1}, d_1]$-almost invariant for some $S_1 \supset H$ of strictly larger dimension;
here $\xi_1,d_1$ depend only on $H,G$ and $c_1 \ll 1$. 
\end{enumerate}

Suppose the second case occurs. By the Lemma again,
applied with $c_1 \epsilon^{\xi_1}$ instead of $\epsilon$, 
one of the following occur for small enough $\epsilon$:
\begin{enumerate}
\item
$|\mu(f) - \mu_{x_0 S_1}(f)| \ll \epsilon^{\xi_1 \ref{volumeconst2}(d_1)}
 \Sob_{d_1}(f)$, where
 $x_0 S_1$ is a closed $S_1$-orbit of volume $\ll \epsilon^{-\xi_1 
\ref{volumeconst1}(d_1)}$, or:
\item $\mu$ is $[S_2, c_2\epsilon^{ \xi_2}, d_2]$-almost invariant
 for some $S_2 \supset  H$ of strictly larger dimension than $S_1$.
  Here $ \xi_2, d_2$ depends only on $H$ and $G$, 
whereas $c_2 \ll 1$. 
\end{enumerate}

Iterating this process -- which we can do
at most $\dim(G)-1$ times -- we arrive at the following conclusion:
$$|\mu(f) - \mu_{x_0 S_j}(f)| \ll
\epsilon^{ \xi_j \ref{volumeconst2}(d_j)} \Sob_{d_j}(f)$$
 where $S_j \supset H$ has dimension $\geq \dim(H) + j$,
and $x_0 S_j$ is a closed orbit of volume $\ll \epsilon^{- \xi_j 
\ref{volumeconst1}(d_j)}$. 

In the rest of the proof,
we will abbreviate an expression like $|\mu(f) - \mu_{x_0 S}(f)|
\leq \epsilon \Sob_{d}(f)$ to the phrase ``$\mu$ and $\mu_{x_0 S}$
are $\epsilon$-close w.r.t. $\Sob_{d}$;'' sometimes
we will suppress mention of the Sobolev norm.

Let $\Delta$ be the supremum of all
quantities $\xi_j \ref{volumeconst1}(d_j)$ that may arise through the above
process. Let $\delta$ be the infimum of all quantities
$\xi_j \ref{volumeconst2}(d_j)$ that may arise through the above process.
Let $d$ be the supremum of all $d_j$ that may arise through the above process. 
These three constants depend only on $H$ and $G$, for there are a finite list
of chains of intermediate subgroups $H \subset S_1 \subset S_2
\subset \dots S_j$. 

Thus, we conclude that for $\epsilon$ smaller than some constant $\epsilon_0 \ll 1$, 
there is an intermediate subgroup $S_j \supset H$
and a closed $S_j$-orbit $x_0 S_j$ of volume $\leq \constc \label{FinalVolumeConstant} \epsilon^{-\Delta}$,
so that $\mu$ is $\leq \constc \epsilon^{\delta}$-close w.r.t.\ $\Sob_d$
to the $S_j$-invariant probability measure on $x_0 S_j$. 

The conclusion of our theorem follows: choose $\epsilon$
so that $\ref{FinalVolumeConstant}\epsilon^{-\Delta}=V$. Then the 
 above applies for $V \gg 1$, showing that $\mu$ is $\ll V^{-\delta/\Delta}$-close w.r.t.\ $\Sob_d$ to the 
$S$-invariant probability measure on the closed orbit $x_0S$ of volume
 $\leq V$, where $S$ is some intermediate subgroup.

 By choosing appropriately
 $V_0$ (from the statement of Theorem \ref{thm:main}) in a way depending on $G$, $H$ and $\Gamma$,  we can remove the implicit constant by lowering the exponent:
 $\mu$ is $\leq V^{-\delta/(2\Delta)}$-close to the measure on $x_0S$ whenever $V \geq V_0$. \qed

\subsection{Proof of the topological Theorem \ref{mcor1}.}

\proof[Proof of Theorem \ref{mcor1}]

Notations as in Theorem \ref{thm:main}; especially, let $\delta$, $d$, and $V_0$ be
as in that Theorem. 
Let $N$ be a large integer; it will be chosen in course of the proof to depend only on $G,H$. 

Let $\gimel$ equal the number of intermediate subgroups
$H \subseteq S \subseteq G$; it is finite by Lemma \ref{assumption-2}. 

Consider the set $\mathcal{V}$ of all volumes of closed $S$-orbits $x_0 S$,
where $H \subset S \subset G$. 
The cardinality of $\mathcal{V}$ is $\leq \gimel$. 

Take $V> V_0^{N^{\gimel+1}}$ and consider the intervals 
$$ (V^{1/N^{\gimel+1}}, V^{1/N^{\gimel}}],\ldots,
(V^{1/N^{d}}, V^{1/N^{d-1}}], \dots,(V^{1/N}, V],
\mbox{ for } 1 \leq d \leq \gimel+1.
$$
It is clear that one of these intervals will contain no element of $\mathcal{V}$. Call it  $(X, X^N]$. 
We apply Theorem \ref{thm:main} with the parameter ``$V$'' set to $X^{N}$.
It follows that there exists a closed $S$-orbit $x_0 S$ of volume $\leq X$,
for some intermediate subgroup $H \subset S \subset G$, so that:

\begin{equation}\label{Clifford}\biggl|\int f
\operatorname{d}\!\mu-\int f \operatorname{d}\!\mu_{x_0 S}\biggr|<X ^{-N\delta} \Sob_d(f),
\end{equation}

Now, let $x \in x_0 S$. Let $\eta, f_x$ be as in the proof of 
Lemma \ref{iso}; roughly $\eta$ is any number less than the injectivity radius at $x$, and $f_x$ is a bump function around $x$
of radius $\eta$.  The Sobolev norms of $f_x$ are estimated in \eqref{fxnorm}. 

Applying \eqref{Clifford}, we see that:
$$\left| \mu(f_x) -  \mu_{x_0 S}(f_x)  \right| \ll X^{-N \delta}
\height(x)^{d} \eta^{-d}.$$
On the other hand, it is clear that 
$$\mu_{x_0 S}(f_x) \gg 
\eta^{\dim(S)} \vol(x_0 S)^{-1} \gg \eta^{\dim(S)} X^{-1}.$$
In particular, so long as:
$$\eta^{\dim(S) + d} \gg X^{1-N \delta} \height(x)^{d}$$
we will have $\mu(f_x) > 0$; in particular, $\supp(\mu)$ must intersect
the $\eta$-ball around $x$.  

Specialize to the case when $\Gamma \backslash G$ is compact; the general case of 
Theorem \ref{mcor1} follows in a precisely similar way. 
Taking $N$ sufficiently large, applying Lemma \ref{iso}, and noting that 
$X \geq V^{1/N^{\gimel+1}}$, we conclude 
that $\supp (\mu)$ intersects any ball on $x_0 S$ of radius
$\geq  V^{-\star}$. Here the notion of ``ball'' is taken with respect to the induced Riemannian metric on $x_0 S$. 
\qed

From the proof we can extract the following Corollary.
In essence it is equivalent to the Theorem, but it is a formulation
that seems at first a little stronger and is helpful in various contexts. 

\begin{cor}\label{mcor2}
 Notation as in Theorem \ref{thm:main}, let $\Delta>0$ be real.
There exists $r\in(0,1)$ and $W_0>0$,
 depending on $G$, $H$, $\Gamma$, and $\Delta$, so that:

For any closed $H$-orbit $x_0H$ and any $W>W_0$
 there exists $V\in [W^r,W]$, an intermediate subgroup $S$, and a closed orbit $x_0S$ of volume $\leq V$, so that $\mu$ is $V^{-\Delta}$-close to $\mu_{x_0S}$ w.r.t.\ $\Sob_d$. 
\end{cor}
In other terms, we have ``amplified'' the exponent $\delta$
of Theorem \ref{thm:main} to an arbitrarily large $\Delta$.
Of course, this carries a hidden cost. 
\qed

%%%%%%%%%%%%%%%%%%%%%%%%%%%%%%%%%%%%%%%%%%%%%%%%%%%%%%%%%%%%%%%%%%%%%%%%%%%%%%%%%%%%
\section{An arithmetic application: distribution of integral points
on prehomogeneous hypersurfaces} \label{arithappl}

\subsection{Introduction.}

\subsubsection{Discriminant and height.}
Notation as in our main theorem, closed $H$-orbits also have an {\em arithmetic} invariant, the ``discriminant''
which measures their arithmetic complexity (cf.\ \cite{ELMV}).
In arithmetic applications, what one can easily measure is the discriminant,
rather than the volume, of a closed $H$-orbit; thus we shall
present a proposition relating the two. This result is Proposition \ref{heightdisc}. 

We have already indicated some arithmetic applications of our results in
\S \ref{sarithmetic}; mainly to see how Proposition \ref{heightdisc} arises naturally,
we shall present 
  another class of applications: to problems of Linnik type.

\subsubsection{Linnik problems.}
By a {\em Linnik-problem}, we have in mind the following: $f$ is a homogeneous polynomial on a $\Q$-vector space, and we wish to analyze
the distribution of integral points on the level set $\{f^{-1}(d)\}$. 
For a discussion of problems of this type, see also \cite{MV-ICM}. 
In the case when $f$ is {\em prehomogeneous}, this problem is amenable
to analysis by our methods (cf. \cite{EO}).  We shall present
a quantitative theorem in this direction, for certain classes of such $f$,
in Proposition \ref{HEEprop}. 

This class of applications builds on the work of others.
In particular, the applicability of Ratner's theorem
to these problems was observed by A. Eskin and H. Oh \cite{EO2}, 
related problems were studied by W. Gan and H.Oh.  The idea
of using invariant theory to handle focussing problems originates
in work of A. Yukie \cite{YukieI, YukieII}.

\subsection{Some comments on heights.}
Let $W$ be a $\Q$-vector space equipped with a Euclidean norm
and an integral lattice $W_{\Z} \subset W$. 
The {\em height} of a subspace $W'  \subset W$ is, 
by definition, the norm $\|e_1 \wedge \dots \wedge e_r\|_{\wedge^{r} W}$,
where $e_1, \dots, e_r$ is a basis for $W' \cap W_{\Z}$,
and the norm on $\wedge^r W$ is that derived from $W$. 
In explicit terms:
$$\height(W')^2 = \det((e_i, e_j)).$$

We shall use the following simple principles:
\begin{equation} \label{UP0} \mbox{A subspace of low height has a basis of low height.}  \end{equation}
In explicit terms, one may choose a basis for $W' \cap W_{\Z}$
so that $\|e_i\| \ll \height(W')$.  This assertion is simply lattice
reduction theory, together with the observation that
the lengths of elements of $W_{\Z}$ are bounded below. 
\begin{equation} \label{UP1}\mbox{A system
of linear equations of low height  has a solution set of low height,}
\end{equation}
In more explicit terms, given a $m \times n$ integral matrix $A$,
all of whose entries are bounded above by a constant $\|A\|$,
the kernel of $A$ -- considered as a subspace of $\Q^m$,
where we endow $\R^m$ with the Euclidean norm -- has height bounded by $\ll \|A\|^{\kappa(m,n)}$. 

\subsection{The discriminant of a closed $H$-orbit.}
Notation as in our main theorem.
Suppose that $\Gamma g H$ is a closed $H$-orbit.  We shall attach to it an arithmetic invariant, the
{\em discriminant}. 
The corresponding procedure when $H$ is a torus was introduced in
\cite{ELMV}. 

First of all, we observe that, with $\Lambda_g = \Gamma \cap g H g^{-1}$, 
\begin{equation} \label{zar} \mbox{the Zariski closure of $\Lambda_g$ has Lie algebra $\mathrm{Ad}(g) \mathfrak{h}$.}\end{equation}
 This follows from the Borel-Wang density theorem
\cite[Chapter II, Corollary 4.4]{Margulis}, together with the fact that the algebraic group underlying $H$
has no compact factors.
Thus, $\Ad(g) \mathfrak{h}$ is a $\Q$-subspace of $\mathfrak{g}$.

Let $r = \dim(H)$, and
define $V = (\wedge^r \mathfrak{g})^{\otimes 2}$,  $V_{\Z} = (\wedge^r \mathfrak{g}_{\Z})^{\otimes 2}$.
Choosing any $\Q$-basis
$e_1, e_2, \dots, e_d$ for the Lie algebra of $\Ad(g) \mathfrak{h}$, we
set
$v_{gH}= \frac{(e_1 \wedge e_2 \wedge \dots \wedge e_r)^{\otimes 2}}{\mathrm{det}(B(e_i, e_j))},$ where
$B$ is the Killing form\footnote{The restriction of $B$ to $\mathfrak{h}$
is nondegenerate. This statement may be verified at the level of complexifications. Choose a real form of $\mathfrak{h}_{\C}$ which is compact. 
The restriction of $B$ to this subalgebra is negative definite, whence the assertion.}
 on $\mathfrak{g}$.  Then $v_{gH}$ is independent of the choice of basis $(e_i)$ for $\Ad(g) \mathfrak{h}$.  
We define:\footnote{By virtue of the fact that $\mathfrak{g}_{\Z}$
is $\Gamma$-stable, this definition is independent of the 
choice of $g$.}
$$\disc(\Gamma g H) = \min\{ m \in \Z:  v_{gH} \in m^{-1} V_{\Z}\}.$$

\begin{prop} \label{heightdisc}
There exists $\consta \label{discvol} >0$%
\index{xappa81@$\ref{discvol}$, bound on volume in terms of discriminant}
so that, for any $x \in \Gamma \backslash G$ with $x H$ closed:
\begin{enumerate}
\item There exists a representative $g \in G$ for $x$ so that the height
of $\Ad(g) \mathfrak{h}$ is $\asymp \disc(xH)^{1/2}$.
\item $\vol(x H) \gg \disc(x H)^{\ref{discvol}}.$
\item $\vol(x H) \ll \disc(xH)^{\xappa}$. \end{enumerate}
\end{prop}

We will not prove the third assertion, because we do not need it,
but in fact the proof is substantially easier than the second assertion. 
\proof (Sketch).
 We have seen (Lemma \ref{measureoutsidecompact}) that there exists a fixed compact set $\compactumfortyone \subset G$
whose projection to $\Gamma \backslash G$ necessarily intersects $xH$. 
Take $g \in \compactumfortyone$ to be any representative for $x$.

For the first observation, we observe that both the restriction of the Killing form,
and the chosen Euclidean structure on $\mathfrak{g}$, induce definite quadratic forms
on $\wedge^r \left(\Ad(g) \mathfrak{h}\right)$, for each $g \in G/H$. Clearly, the ratio of these forms vary continuously.
In particular, if $g$ lies within a fixed compact subset of $G/H$, the height of $\Ad(g) \mathfrak{h}$
is, up to constants, comparable to the square root of the discriminant, whence the first conclusion.

 There are multiple methods of proof for the second assertion. We indicate a
proof using dynamical ideas. 

We first claim an effective version
of \eqref{zar}. Namely,  there exist constants $\constc \label{c1const}, 
\consta \label{c3const}$\index{xappa9@$\ref{c3const}$, effective generation of Zariski closed group} so that:
\begin{eqnarray*}\label{mdr}\mbox{the Zariski closure $\mathfrak{L}'$ of the group generated by} \\  \nonumber
\{\lambda \in \Lambda_g : \|\lambda\| \leq \ref{c1const} \vol(x H)^{\ref{c3const}}
%\|g\|^{\ref{c3const}}
 \}    \
\mbox{ has Lie algebra $\mathrm{Ad}(g) \mathfrak{h}$.}\end{eqnarray*}
%where $c_1, c_2 \ll 1$, i.e. are constants that may be taken to depend only on $\Gamma, G, H$.
This can be established using results about lattice point counting,
such as those from \S \ref{lattice}. 

This being established,
let $T \geq \ref{c1const} 
\vol(xH)^{\ref{c3const}}$. 
Let $F := \{\lambda \in \Lambda: \|\lambda\| \leq T\}$. 

% Then the subalgebra $\mathrm{Ad}(g) \mathfrak{h} \otimes \C$ can be verified to be the unique subalgebra of $\mathfrak{g} \otimes \C$, of dimension equal to that of $H$, and fixed by the adjoint action of
%$\{\lambda \in \Lambda: \|\lambda\| \leq T\}$. 
%
% By reasoning similar to that of
%\S \ref{Arg:EffectiveNoetherian}, there exists a subset 
%$F \subset \{\lambda \in \Lambda_g: \|\lambda\| \leq T\}$, with cardinality bounded in terms of $G$ alone,
%so that $\mathrm{Ad}(g) \h \otimes \C$ is the unique subalgebra
%of $\mathfrak{g} \otimes \C$ fixed by $F$.

Let us fix a representation $\G$ on a $\Q$-vector space
$W$ with the following property: if $W_{\R}^{H}$ denotes
the $H$-fixed vectors in $W_{\R}$, then the pointwise stabilizer
of $W_{\R}^{H}$ in $G$ is precisely the Zariski-closure of $H$
within $G = \G(\R)$. It is possible to do this, by Chevalley's theorem
and the fact that the Zariski-closure of $H$ is connected
and does not admit algebraic characters.  In particular, 
\begin{equation} \label{PS}
\{v \in \mathfrak{g}: v . W_{\R}^{H} = 0\} = \mathfrak{h}.
\end{equation}

Fix, once and for all, a Euclidean norm on $W_{\R}$
and an integral lattice $W_{\Z}$.

Consider the fixed subspace $W^F$ for $F$ inside $W$.
We claim that $W^F$ has ``low height'', i.e.
bounded by a power of $\|g\|$ and $\vol(x H)$, when considered as a subspace
of $W$.  Indeed, one
may replace $F$ by a subset $F' \subset F$, with cardinality
bounded in terms of $G$ alone, so that $W^{F'} = W^F$,
by similar arguments to  \S \ref{Arg:EffectiveNoetherian}. 
The claim follows from \eqref{UP1}. 

Our definitions are so that the subgroup generated
by $F$ is Zariski-dense in $g H g^{-1}$;
it follows that $W_{\R}^F$ equals $g . W_{\R}^H$. 

Therefore, the subalgebra of $\mathfrak{g}$ defined by:
\begin{equation} \label{fixed}\{ v \in \mathfrak{g}: v .  W^F = 0\} \end{equation}
also has ``low height'', in the same sense.
This follows from \eqref{UP0} and \eqref{UP1}. 
By our assumption, the subalgebra defined by \eqref{fixed}
coincides with the Lie algebra of $\Ad(g) \mathfrak{h}$.

We have established that $\Ad(g) \mathfrak{h} \subset \mathfrak{g}$
has ``low height'' in terms of $\vol(x H)$; in view of the proof of the first assertion of the present Proposition, we are done.  \qed
%It follows that the subalgebra of $\mathfrak{g}$
%that fixes 

%In algebraic terms, the point in the Grassmannian corresponding $\mathrm{Ad}(g) \h \otimes \C$
%is the unique solution to a ``small'' system of equations, all of whom involve only rational
%coefficients of ``small'' height. From this one may deduce,
%by an application of the principle \eqref{UP}, that the
%height of $\Ad(g) \h$, considered as a $\Q$-point in 
%the Grassmannian of $\g$, is not too large.

%ndeed,
%suppose $W \subset \g_{\C}$ is a subalgebra fixed under the adjoint
%action of $\Lambda_g \cap B_G(T)$, so also under the adjoint action of
%$\mathrm{Ad}(g)(\mathfrak{s})$. Then $\mathrm{Ad}(g)(\mathfrak{s}) \oplus W$
%is a subalgebra of $\mathfrak{g}_{\C}$ containing $\mathrm{Ad}(g)(\mathfrak{s}_{\C})$, which is necessarily semisimple by Lemma \ref{assumption-2}.  But then this means that the centralizer of $\mathrm{Ad}(g)( \mathfrak{s})$ in $\mathfrak{g}$ is nontrivial, a contradiction.}
%In the above paragraph, the assumption $H=H^{+}$
%is used. \label{UsageB}

We observe that this Proposition gives immediately an alternate proof of Lemma \ref{iso}.

\subsection{Application to Linnik problems.} \label{Linnikapp}

We shall focus on the following setting:
Let $\G$ be a semisimple algebraic $\Q$-group which acts
on a $\Q$-vector space $V$, which
preserves a polynomial invariant $f: V \rightarrow \Q$,  
and so that (algebraically) each level set of $f$ is a single $\G$-orbit. 
Let $V_{\Z} \subset V$
be a lattice, and $\Gamma$ a congruence subgroup of $G := \G(\R)$
that preserves $V_{\Z}$.

This is, roughly speaking, the setting for the (arithmetic) study
of ``prehomogeneous vector spaces.'' We shall impose an additional condition: 
that the stabilizer of a generic point
 is semisimple and has finite centralizer in $\G$. 

Here are two specific instances with these properties (more can
be found by examining the tables of prehomogeneous vector spaces). 

\begin{enumerate}
\item[Case A.] Fix an integer $r \geq 3$.
Let $\G = \SL(r), \Gamma=\SL(r,\Z)$, take $V_{\Z}$ to consist
of $r \times r$ integral symmetric matrices, and $f = \det$. 

Here $\dim(V) = r(r+1)/2$, $\deg(f) = r$, and the stabilizer
of a generic point is a form of the orthogonal group $\SO(r)$. 

\item[Case B.] Fix $r \in \{7,8\}$.\footnote{The space with $r=6$ would
still be prehomogeneous; however, the stabilizer of a generic point
has an infinite centralizer.}
Let $\G = \SL(r), \Gamma=\SL(r,\Z)$, let $V_{\Z}$
be the space of alternating trilinear forms on $\Z^r$,
with integral values; 
and $f$ is the {\em discriminant} (cf. \S \ref{Alt}). 

Here, if $r=7$, then $\dim(V)=35, \deg(f)=7$ and the stabilizer
of a generic point is a form of $G_2$.  If $r=8$,
then $\dim(V) = 36$ and $\deg(f)=16$, and the stabilizer
of a generic point is a form of $\SL(3)$. 
\end{enumerate}

In Case B,  some questions
in Diophantine geometry on $V$, analogous to the Oppenheim conjecture,
 were studied by ergodic methods by Yukie and collaborators
(\cite{YukieI}, \cite{YukieII}) and, indeed, we use a technique
to handle ``focussing'' analogous to that of \cite{YukieI, YukieII}.

The following Proposition gives a quantitative solution to Linnik's distribution
problem in the Cases A and B above.  A non-effective version in Case A was established
by Eskin and Oh, \cite{EO}. We observe that, in both these cases,
the degree is ``too large'' relative to the dimension for the Hardy-Littlewood method to be applicable. (For instance, Case A includes the case
of a (special) cubic form in five variables.)

\begin{prop} \label{HEEprop}
Let $(V, f)$ be as in Case A or Case B, defined above.

Let $V^{nc}_{\R}$ be the open subset of $V_{\R}$ comprising
points whose stabilizers in $\G(\R)$ are noncompact. 

Suppose $\Omega \subset \{x \in V^{nc}_{\R}: f(x) = 1\}$ is compact
with smooth boundary.

 Let $d \rightarrow \infty$
vary through integers with bounded square part\footnote{I.e. the
set of perfect squares which divide some $d$ is bounded}
 and so that $f^{-1}(d) \cap V_{\Z} \neq \emptyset$. 

Then:
$$\bigl|\{ M \in V_{\Z}: f(M) = d, \frac{M}{d^{1/\deg(f)}} \in \Omega\}\bigr| = C_d( \vol(\Omega) + O_\Omega(d^{-\delta})).$$
Here $\liminf_{d} \frac{\log C_d}{\log d} > 0$.
\end{prop}

The proof will be only sketched. Presumably $\frac{\log C_d}{\log d}
\rightarrow \frac{\dim(V) - \deg(f)}{\deg(f)},$ but we do not establish this.\footnote{
It is related to the question of the precise relationship between
discriminant and volume of a periodic orbit. In principle, this
can be reduced to the computation of Tamagawa number, but the local computations seem difficult in the most general setting.}
We intend to elaborate on this and other applications
in the $S$-arithmetic sequel to this paper.  

The restriction that $\Omega \subset V^{nc}$ can presumably be removed
by $p$-adic methods, e.g. \cite{EV}.
 For ``Case A''
it is clear that $f$ takes on all integer values, and so
we can dispense with the restriction $f^{-1}(d) \cap V_{\Z}$ be nonempty;
it would be nice to establish the exact range of values taken by $f$
in ``Case B.''

\subsection{Prehomogeneous vector spaces: generalities} \label{PHVS}
Let us proceed in the general setting enunciated at the start
of \S \ref{Linnikapp}.   

The set of points in $V - f^{-1}(0)$ whose stabilizer
is a prescribed subgroup of $\G$ is a finite set of lines,
by virtue of the assumption that the generic stabilizer is
semisimple and has finite centralizer. 

The level set $f^{-1}(1)$ is a union of a finite collection of $G = \G(\R)$-orbits. Fix representatives $x_i$ for each orbit, and let $H_i$ be the 
stabilizer of $x_i$.  (We do not assume that $H_i$ is connected;
we will not go into details about the easy arguments
required, in what follows, to get around connectedness issues.)
We fix also a compact subset $\Omega \subset f^{-1}(1)$.

For $d >0$, each level set $f^{-1}(d)$ is -- by scaling -- identified 
with 
$\bigcup_{i} G/H_i$.  For $y \in f^{-1}(d) \cap V_{\Z}$, let $\bar{y}$
be its projection to $f^{-1}(1)$; we may choose 
$i(y), g_y$ so that $\bar{y} = g_y . x_{i(y)}$. 

The orbit $\Gamma g_y H_i$ is then closed,
because the group $g_y H_i g_y^{-1}$ is the stabilizer of $y$
and therefore corresponds to the real points of a $\Q$-algebraic group.

\begin{lem} \label{lowerbounddisc}
There exists $c> 0$ with the following property:
for any $y$ with $\bar{y} \in \Omega$ and $f(y) \neq 0$,  the discriminant of $\Gamma g_y H_i$ is $\gg_{\Omega} \height(\Q.y)^{c}$. 
\end{lem}
\proof
Consider the fixed points for the Lie algebra $\Ad(g_y) \Lie(H_i)$.

This set of fixed points is a linear subspace;
on the other hand, it intersects the Zariski-open set $V - f^{-1}(0)$ in a nonempty finite set of lines.
Therefore,  it must consist of a single line. 
We may therefore characterize $\Q.y$ as the fixed line for
$\Ad(g_y) \Lie(H_i)$.  \eqref{UP0} and \eqref{UP1} imply that the height of $\Q.y$
is bounded by a power of the height of $\Ad(g_y) \Lie(H_i)$. 

Because $\bar{y} \in \Omega$, we may suppose that $g_y$
belongs to a fixed compact subset within $G$. 
Our claim now follows from the first assertion of Proposition \ref{heightdisc}. 
\qed  

\subsubsection{Proof of Proposition \ref{HEEprop} in Case A} \label{Acaseproof}
Let us restrict to Case A. 

Suppose that $y \in V_{\Z}^{nc} \cap f^{-1}(d)$,
with $\bar{y} \in \Omega$. 

It follows from Lemma \ref{lowerbounddisc} and
the assumption on $d$ (bounded square part) that:
\begin{equation} \label{shopper}\vol(\Gamma g_y H_i) \gg d^{\kappa},\end{equation}
for some $\kappa > 0$. 

Indeed, if we let $\spadesuit$ be the g.c.d.
of the entries of $y$, i.e. the largest integer
so that $\spadesuit^{-1} y \in V_{\Z}$, then
$\spadesuit^r$ divides $\det(y) = d$. Because of the assumption
that $d$ had bounded square part, $\spadesuit$ is also bounded.
But the height of $\Q.y$ is, up to bounded multiples,
the norm of $\spadesuit^{-1} y$, whence our conclusion. 

In combination with our main theorem, this implies
 Proposition \ref{HEEprop} in ``Case A.''
See \cite{EO} for details of translating equidistribution of $H$-orbits to statements in the style of Proposition \ref{HEEprop}. 

 The idea, in words,
is that our main theorem shows, in an effective sense, the equidistribution
of $\Gamma g_y H_i \subset \Gamma \backslash G$. (We know, by
\eqref{shopper}, that this orbit has large volume). 
This, however, is equivalent to the uniform distribution (interpreted suitably)
of $\Gamma g_y$ on $G/H$, or, equivalent to the uniform distribution
of the $\Gamma$-orbit $\Gamma y$ on $f^{-1}(d)$.
The integral points on $f^{-1}(d) \cap V_{\Z}$ are the union of finitely many such orbits, whence our conclusion. 

\subsubsection{Preparations for the proof in Case B}
In order to establish Proposition \ref{HEEprop} in ``Case B,''
we shall give a little more background on alternating trilinear forms. 

Let us recall that $V_{\Z}$ is the space of {\em trilinear alternating tensors}
on $\Z^r$. In explicit terms, an element $t \in V$
is an alternating, trilinear map:
$$t: \Z^r \times \Z^r \times\Z^r \rightarrow \Z.$$

%$\SL_3 \times \SL_3. \mu_2$ when $r=6$, $G_2 . \mu_3$ when $r=7$,
%and $\SL_3. \mu_3.  \langle \sigma \rangle$ when $r=8$;
%here $\sigma$ generates a $2$-group and $\det(\sigma) =1$. 
%(For the last-mentioned fact, we consider the 
%form $(X,Y,Z) \mapsto \Tr(X(YZ-ZX)$ on traceless $3 \times 3$ matrices;
%it is stabilized by $X \mapsto -X^t$). 
%
%Correspondingly, the stabilizer of such a point under the $\SL(d,\C)$-action
%equals $G_2$ when $r=7$ and $\SL_3 . \{\pm 1\}$ when $r=8$. We observe
%that the normalizer of $Gboth consists of scalars. _2$ inside $\SL_7$ and of $\SL_3$ inside $\SL_8$
%both consists of scalars. 

\subsubsection{Invariants and covariants} \label{Alt}
For $6 \leq r \leq 8$, trilinear tensors have a {\em discriminant}: 
a polynomial function $\disc: V \rightarrow \Q$, 
 homogeneous of degree $4$ (resp. $7, 16$) when $r=6$ (resp. $7, 8$).  For definitions, we refer to \cite{YukieI, YukieII};
unfortunately, we do not know of any entirely simple definition. 

There exists a polynomial $\SL(d)$-equivariant map:
$$\cov: V \rightarrow \Sym^2 \Q^r$$
of degree $3$ (when $r=7$) and of degree $10$ (when $r=8$). 
This is proven in \cite[1.16]{YukieI} and \cite[Definition 2.13]{YukieII}. 

The determinant of $\cov(t)$ is a scalar multiple of 
$\disc(t)^3$ ($r=7$)
and a scalar multiple of $\disc(t)^5$ (when $r=8$).  

By ``clearing denominators'',  we may suppose
that $\disc$ maps $V_{\Z}$ into $\Z$, and that $\cov$ maps
$V_{\Z}$ into $\Sym^2 \Z^r$. 

\subsubsection{Proof of Proposition \ref{HEEprop} in Case B}
These preliminaries being established, let us turn to the proof of Proposition
\ref{HEEprop} in the case at hand. 

We follow the general notations of \S \ref{PHVS}, with $\G = \SL_r$,
$V$ as above, $f = \disc$. In particular,
we choose points $x_i \in f^{-1}(1)$ as in \S \ref{PHVS}. 
Let $S_i$ be the stabilizer of $\cov(x_i)$ inside $G$; it is a special orthogonal group, which contains $H_i$ by the equivariance of $\cov$. 

Take $y \in V_{\Z} \cap f^{-1}(d)$
so that the stabilizer of $y$ is non-compact. 
Then $\Gamma g_y S_{i(y)}$ is closed;
for, by equivariance, $\Ad(g_y) S_{i(y)}$ is the stabilizer
of the integer vector $\cov(y_i) \in \Sym^2 \Z^r$, and therefore the real points
of a $\Q$-subgroup. 

It follows from the Lemma of \S \ref{PHVS}
that the volumes of $\Gamma g_y H_{i(y)}$ and $\Gamma g_y S_{i(y)}$
are bounded below in terms of the respective heights of the lines
$\Q.y$ and $\Q. \cov(y)$. 
Now $\disc(y) = d$ and $\det(\cov(y))$ is
a fixed multiple
of $d^3$ or $d^{10}$,
according to whether $r=7$ or $8$. 
Reasoning as in \S \ref{Acaseproof}, and
we conclude that the height of $\Q.y$ exceeds $d^{1/\deg(f)}$.
and that the height of $\Q. \cov(y)$ exceeds
$d^{3/7}$ resp. $d^{1/8}$, according to whether $r=7$ or $r=8$. 

Finally, it is proven in \cite{YukieI, YukieII} that 
the Lie algebra of $S_i$ is the only
such algebra intermediate between $H_i$ and $G$. 

The claimed result of Prop. \ref{HEEprop} follows from Theorem \ref{thm:main}. 
Indeed, our discussion above has shown that both $\Gamma g_y H_{i(y)}$
and $\Gamma g_y S_{i(y)}$ has ``large'' volume; applying the Theorem shows
that $\Gamma g_y H_{i(y)}$ is approximately uniformly distributed
in $\Gamma\backslash G$. This translates into the statement of Proposition \ref{HEEprop}, as carried out in \cite{EO} and recalled in approximate
form in \S \ref{Acaseproof}. 
\qed

%%%%%%%%%%%%%%%%%%%%%%%%%%%%%%%%%%%%%%%%%%%%%%%%%%%%%%%%%%%%%%%%%%%%%%%%%%%%%%%%

\appendix

%%%%%%%%%%%%%%%%%%%%%%%%%%%%%%%%%%%%%%%%%%%%%%%%%%%%%%%%%%%%%%%%%%%%%%%%%%%%%%%%

\section{Proof of Lemma \ref{assumption-2}}\label{sec-ass-2}

In this section, $k$ will be a local field of characteristic zero,
$\G$ a semisimple algebraic group over $k$, $\g$ its Lie algebra, $\h \subset \g$
a semisimple subalgebra. 

We establish some preliminary results.

\subsection{Embeddings of semisimple Lie algebras}

\begin{lem}\label{finite-embeddings}
Let $\mathfrak{s}$ be a semisimple Lie algebra over $k$.
There exist finitely many embeddings of $\mathfrak{s}$
into $\mathfrak{g}$, up to $\G(k)$-conjugacy. 
\end{lem}
\proof
It suffices to prove this statement over the algebraic closure $\bar{k}$. 
Indeed, that being assumed,
the affine variety $\mathbf{X}$ parameterizing
embeddings of $\mathfrak{s}$ into $\mathfrak{g}$
is a finite union of homogeneous $\G$-spaces. The finiteness 
of Galois cohomology over local fields
assures that, given a homogeneous $\G$-space $\mathbf{Y}$, the
set of $\G(k)$-orbits on $\mathbf{Y}(k)$ is also finite.

For the statement over $\bar{k}$, choose an embedding
of $\G$ into the general linear group $\GL(n)$.
The representation theory of semisimple Lie algebras
assures that the number of $\GL(n, \bar{k})$-orbits
on homomorphisms $\mathfrak{s} \rightarrow \mathfrak{gl}_n$
is finite. Our assertion then follows from \cite[Theorem 7.1]{Richardson}. 
\qed

\subsection{Parabolic subgroups}

\begin{lem} 
If $\mathbf{S}$ is a proper algebraic subgroup of $\mathbf{G}$, so that $\mathbf{S} = N_G(R_u(\mathbf{S}))$, 
then $\mathbf{S}$ is parabolic. 
\end{lem}

See \cite{BT} or \cite{Weisfeiler}.

\begin{lem} \label{UR} Any algebraic subgroup of $\G$, with a nontrivial unipotent radical, is contained
in a parabolic subgroup of $\G$. 
\end{lem} 

\proof
Let $\mathbf{S}$ be an algebraic subgroup with nontrivial unipotent radical.
Define, inductively, $\mathbf{S}^{(0)} = \mathbf{S}$ and $\mathbf{S}^{(j+1)} = N_G(R_u(\mathbf{S}^{(j)}))$, for $j \geq 0$. Note that $\mathbf{S}^{(j)}$ normalizes $R_u(\mathbf{S}^{(j)})$ so that $\mathbf{S}^{(j)}\subset \mathbf{S}^{(j+1)}$. Furthermore, $R_u(\mathbf{S}^{(j)})\subset R_u(\mathbf{S}^{j+1})$. Since unipotent groups are connected in characteristic zero, the increasing chain $R_u(\mathbf{S}^{(j)})$ of algebraic subgroups must necessarily stabilize. Therefore, the same must hold for the chain $\mathbf{S}^{(j)}$.  The previous lemma allows us to conclude the proof. 
\qed

\subsection{Intermediate subgroups}

\begin{lem} 
Suppose $\h \subset \g$ is semisimple and has trivial centralizer.
Then any intermediate subalgebra is semisimple.
\end{lem} 
\proof
Take $\h \subset \mathfrak{s} \subset \g$. Suppose $\mathfrak{s}$ were not semisimple. 
Then $\mathbf{S}$, the connected component of the normalizer of $\mathfrak{s}$, must also fail to be semisimple.  

It suffices, therefore, to prove that any intermediate algebraic subgroup between 
$\mathbf{H}$ and $\G$ is semisimple. Here $\mathbf{H}$ is the connected algebraic group
with Lie algebra $\h$. 
 
We claim that $\mathbf{S}$ cannot have nontrivial unipotent radical. 
In view of Lemma \ref{UR}, were this false, then $\mathbf{H}$ would be contained
in a proper parabolic subgroup $\mathbf{P}$ of $\G$. 
Consider the representation of $\mathbf{H}$ on the Lie algebra $\mathfrak{p}$;
let $\mathfrak{n}$ be the Lie algebra of the unipotent radical $\mathbf{N}$ of $\mathbf{P}$. The quotient group $\mathbf{P}/\mathbf{N}$ has nontrivial center;
therefore, $\mathbf{H}$ fixes a subspace in its adjoint action
on $\mathfrak{p}/\mathfrak{n}$. By semisimplicity of $\mathbf{H}$,
it also fixes a subspace in its action on $\mathfrak{p}$. 
Therefore, $\mathfrak{h}$ has a nontrivial centralizer, contradiction.

%then we can choose a maximal torus of $\mathbf{P}$ containing both a maximal torus of $\mathbf{H}$ 
%and the image of $\chi$. We claim that $\mathbf{H}$ must commute with $\chi$, otherwise there is
%some
%This shows that $\mathbf{H}$ commutes with the image of $\chi$
% which is a contradiction to our assumptions. 

To conclude note that in absence of a unipotent radical, the radical of the subgroup $\mathbf{S}$ is central. Since $\h$ has trivial centralizer, $\mathbf{H}$ has finite centralizer. 
Therefore, we see that the radical of $\mathbf{S}$ must be trivial. 
 \qed
% For the proof of the above claim we will show that any algebraic $\R$-subgroup of $G$
% with nontrivial unipotent radical must be contained in an algebraic parabolic $\R$-subgroup of $G$.
% Assume this for now, then the claim is implied as follows. If an intermediate subgroup $S$
% were not semisimple, then its radical were nontrivial. If the unipotent radical were trivial,
% it would have a nontrivial torus as its center which is impossible by our assumption on $H$. 
% If the unipotent radical is nontrivial, then by the above it would be contained in a parabolic
% subgroup $S'$ defined by some cocharacter $\chi$ over $\R$ (see \cite[Lemma 15.1.2]{Springer}).
% However, the image of $\chi$ must then commute with the Levi subgroup of $S'$ which has
% to contain a conjugate of $H$. Again this is impossible by our assumption on $H$.

\begin{lem} 
Suppose $\h \subset \g$ is semisimple and has trivial centralizer.
Then there exist only finitely many intermediate subalgebras $\h \subset \mathfrak{s} \subset \mathfrak{g}$. 
\end{lem} 

  \proof
It suffices to prove this statement over the algebraic closure. 
 
 There exist only finitely many isomorphism classes of semisimple Lie algebras $\mathfrak s$ that can be embedded into $\mathfrak g$. For each such,
there exist -- by Lemma \ref{finite-embeddings} -- only finitely many $\G(\bar{k})$-conjugacy classes
of embeddings. Choose representatives for the image
of every such embedding, calling them $\mathfrak{s}_1, \dots, \mathfrak{s}_m$. 
Let $\mathbf{S}_j \subset \mathbf{G}$ be a connected semisimple algebraic group
with Lie algebra $\mathfrak{s}_j$. 

Then, for each $j$, there exist but finitely many $\mathbf{S}_j(\bar{k})$-conjugacy classes of subalgebras of $\mathfrak{s}_j$, isomorphic to $\mathfrak{h}$. 
Call them $\mathfrak{h}_{ij}$, $1 \leq i \leq N_j$. 

Take any intermediate subalgebra $\mathfrak{s} \supset \mathfrak{h}$.
The pair $(\mathfrak{h} \subset \mathfrak{s})$ is conjugate
under $\mathbf{G}(\bar{k})$ to $(\mathfrak{h}_{ij} \subset \mathfrak{s}_j)$ for
some $j$ and $i$. 

 Thus, there exist only finitely many
possibilities for the conjugacy class of the pair $(\mathfrak{h} \subset \mathfrak{s})$. However, the normalizer of $\mathfrak{h}$ in $\mathbf{G}$, 
contains $\mathbf{H}$ as a finite index subgroup.  
The claimed result follows. 
\qed

\proof[Proof of Lemma \ref{assumption-2}] 
 The results of this section
imply all the claims, save that every intermediate subgroup $S \supset H$
has no compact factors. However, $\mathfrak{h}$ would, by necessity,
centralize the Lie algebra of any compact factor. This contradicts
the assumption that $\mathfrak{h}$ has no center. \qed

%%%%%%%%%%%%%%%%%%%%%%%%%%%%%%%%%%%%%%%%%%%%%%%%%%%%%%%%%%%%%%%%%%%%%%%%%%%%%%%%

\section{ Proof of Lemma \ref{measureoutsidecompact}.} \label{manfredsec}
%\marginpar{$g_{\Z}$ needs to have covol 1}

For the proof of Lemma \ref{measureoutsidecompact} we will use the non-divergence results for actions of unipotent subgroups.

In the paper \cite{M-non-divergence}, by G.M.,
 it has been shown that a point of low height ``returns to the set
of low height infinitely often'' under a one-parameter unipotent flow. 

 Dani has refined this in \cite{Dani-unipotent}
to show that the corresponding conclusion remains valid,
even without the constraint that the original point have low height, 
unless there is a rational subspace
that is invariant under the unipotent one-parameter subgroup considered (i.e.\ a rational constraint prohibits that).

This is almost what we need to prove (1) of Lemma \ref{measureoutsidecompact}. What we will use and actually need for
(2) of Lemma \ref{measureoutsidecompact} is
a quantification of this phenomenon.
We use results from the paper \cite{Kleinbock-Margulis-Ann} by Kleinbock and G.M.

\begin{proof}
For $g \in G$, we call a subspace $V \subset \mathfrak{g}$ 
{\em $g$-rational} if $V \cap \mathrm{Ad}_g^{-1}\mathfrak g_\Z$ 
is a lattice in $V$. 
 We define the covolume of a $g$-rational subspace to be the volume of
 $V/(V\cap (\mathrm{Ad}_g^{-1}\mathfrak g_\Z))$.  
A $H$-invariant, $g$-rational subspace $V$ of low covolume prohibits a point of low height on $\Gamma g H$.  (Let us note that $\Ad_h: V\rightarrow V$
has determinant $1$, because $H$ is semisimple). 

We shall show that, no matter what $g$ is, there 
are {\em no} $g$-rational subspaces of low covolume. 
This will be by induction on $\dim(V)$. 
 The notion of ``low covolume'' will be specified as we go along. 
After this is done, we may establish statement (1) of the Lemma. 

 The beginning of the induction is rather trivial; there are no $H$-invariant
lines $V \subset \mathfrak g$
 by our assumption that the centralizer of $H$ on $\mathfrak g$ is trivial.

 Suppose now we have already established that there are no $H$-invariant $g$-rational subspaces
 of covolume less than $c<1$ and dimension $<k$ for some $k\leq\dim G$.
 Suppose also that $V$ is an $H$-invariant $g$-rational subspace
 with covolume $v$ and dimension $k$ for some $v>0$. As a first step towards our inductive step:
 
{\em Claim.} If $v$ is sufficiently small, there exists $h \in H$ 
so $\mathrm{Ad}_h\mathrm{Ad}_g^{-1}\mathfrak g_\Z\cap V$ 
 has a basis consisting of vectors of length $\ll v^{\frac{1}{k}}$.

 To prove the claim we shall use the result of \cite{Kleinbock-Margulis-Ann} mentioned. 
There are only finitely many sublattices $L_i \subset \mathrm{Ad}_g^{-1}\mathfrak g_\Z\cap V$ of covolume less than $c$ and dimension strictly less than $k$. 
 Let $U$ be the image of $u(t)$, defined in \eqref{udef};
consider the set of $h \in H$ s.t. $\Ad(h) U$ preserves each $L_i$. 
This is an algebraic condition, i.e. the real points of a real algebraic
subvariety of $\G_{\R}$. Thus either:
\begin{enumerate}
\item There is $h\in H$ such that $hUh^{-1}$
 does not leave any of these subspaces invariant;
 \item There exists a sublattice $L_j$ so that 
 we have that $hUh^{-1} (L_j) = L_j$ for all $h \in H$. 
\end{enumerate}

The normal subgroup generated by $U$ must coincide with $H$. 
Thus, in the second case, $H$ preserves $L_j$, a contradiction. 
We are thereby in the first case; conjugating $U$, we may assume
without loss of generality that $U$ leaves invariant
no $g$-rational subspace $W \subset V$ of dimension $<k$ and covolume
$< c$. 

Put 
 $h(t)=\mathrm{Ad}_{u(t)}|_V g'\in\operatorname{SL}(V)$, a polynomial in $t$.
 The same holds for $\bigwedge^\ell h(t)\in\operatorname{SL}(\bigwedge ^\ell V)$ for any $\ell<k$. Therefore, 
 the square of $\psi_W(t)=\operatorname{covol}(h(t)(W\cap \mathrm{Ad}_g^{-1}
\mathfrak{g}_{\Z})$ is also a polynomial
 for any subspace $W \subset V$ of dimension $\ell<k$ which intersects
$\mathrm{Ad}_g^{-1} \mathfrak{g}_{\Z}$ in a lattice. 
 \footnote{Also note that $\psi_W$ agrees up to a bounded multiplicative factor with the function $\psi_{W\cap\mathfrak g^{-1}g_\Z}$
 defined in \cite[\S 5]{Kleinbock-Margulis-Ann} which is defined by a different norm on $\bigwedge^{\dim W}\R^k$.}

There are only finitely many subspaces $W \subset V$ of dimension $\ell<k$
 with $\psi_W(0)<cv^{-\ell/k}$.
For such $W$, it is possible to choose 
some $t_W$ with $\psi_W(t_W)\geq c v^{-\ell/k}$.
 I.e.\ for some $r>0$ we know that if $v$ is sufficiently small depending on $c$ and $k$
 that the supremums norm of these functions satisfy
 $\|\psi_W\|_{[-r,r]}\geq 1$ for all such subspaces $W$.

 By \cite[Prop.~3.2]{Kleinbock-Margulis-Ann}
 the polynomials $\psi_W^2$ are all uniformly $(C,\alpha)$-good for some $C>0$ and $\alpha>0$
 that depend only the degree of $h(t)$ and so on $\dim G$ ---
 the same holds for $\psi_W$ (with slightly different constants).
 We do not need to define this notion since this is
 only used as the second and last assumption of \cite[Thm.~5.2]{Kleinbock-Margulis-Ann}.
 From this quantitative non-divergence theorem we conclude that
 \[
  |\{t\in[-r,r]:~h(t) (\Ad_{g^{-1}} \mathfrak{g}_{\Z} \cap V)\mbox{ contains an element of size }<\epsilon\}|\ll \epsilon^\alpha r,
 \]
 where the implicit constant depends on $C$ and $k$. If we choose $\epsilon$ small
 enough in comparison to the implicit constant we see that there exists some $t$
 such that $h(t)(\Ad_{g^{-1}} \mathfrak{g}_{\Z} \cap V)$ does not contain an element of size smaller than $\epsilon$.
 Going back to $\mathfrak g$ this shows that $V\cap(\mathrm{Ad}_{u(t)g^{-1}}\mathfrak g_\Z)$ does not
 contain an element of size $< \epsilon v^{1/k}\ll v^{1/k}$. Therefore, $V\cap(\mathrm{Ad}_{u(t)g^{-1}}\mathfrak g_\Z)$
 is generated by elements of size $\ll v^{1/k}$ by Minkowski's theorem on successive minima
 for lattices and since $v$ is the covolume.

 To conclude the induction we need to show that there cannot be any $H$-invariant $g$-rational
 subspace $V$ which has small covolume $v$. By the above claim
 we may assume that $V\cap(\mathrm{Ad}_g^{-1}\mathfrak g_\Z)$ is generated by elements $\{w_i\}$ of size $\ll v^{1/k}$.
 Recall that for any $w,w'\in\mathfrak g$ we have
 $\|[w,w']\|\leq  \|w\|\|w'\|$. Suppose $v$
 is small enough such that $V\cap(\mathrm{Ad}_g^{-1}\mathfrak g_\Z)$ is generated by elements of size $<\frac{1}{2 }$. (This choice of $v$ depends only
on the implicit constants in the argument above). 

Therefore, $[w_i,w_j]$ have length $\leq\frac{1}{4 }$;
 similarly for higher order commutators. Also recall that $\g_\Z$ satisfies
 $[\g_\Z,\g_Z]\subset \g_\Z$ and so all of these commutators belong to the lattice $\mathrm{Ad}_g^{-1}\mathfrak g_\Z$.
 Therefore, the Lie algebra $\mathfrak l$ generated by $V$ is nilpotent. 

 Moreover, since $V$ is assumed to be invariant under $H$, the same holds for the Lie algebra $\mathfrak l$.
 Since $H$ is semisimple, we must have that the Lie algebra $\mathfrak h$ of $H$ and $\mathfrak l$ intersect trivially
 since otherwise we have found a nilpotent Lie ideal $\mathfrak l\cap\mathfrak h$ contained in $\mathfrak l$.
 This shows that there is an intermediate
 subgroup $S\subset G$ with Lie algebra $\mathfrak h\otimes \mathfrak l$ which fails to be semisimple in contradiction to
 Lemma \ref{assumption-2}.

 The proof of (1) is now similar to part of the above induction. In fact we have established that there are no
 $H$-invariant $g$-rational subspaces $V$ of small covolume for any dimension $k<\operatorname{dim} G$,
 and as above we may assume this is also true for $U$-invariant $g$-rational subspaces.
 Therefore, the square of $\psi_V(t)=\operatorname{covol}(\mathrm{Ad}_{u(t)}(V\cap \mathrm{Ad}_g^{-1}\mathfrak g_\Z))$
 (for any $g$-rational subspace $V$) is either an unbounded polynomial or equal to a constant $\geq\rho\gg 1$.
 Therefore, for any large enough $r>0$ we will have
 $\|\psi_W\|_{[-r,r]}\geq \rho$ for all $g$-rational subspaces, and so
 \begin{equation}\label{DKM-nondivergence}
  |\{t\in[-r,r]:\Gamma g u(t)\notin\Siegel(\epsilon^{-1})\}|\ll (\frac{\epsilon}{\rho})^\alpha r,
 \end{equation}
 again by \cite[Thm.~5.2]{Kleinbock-Margulis-Ann}. Here the implicit constant
 only depends on $G$. If we choose $\epsilon=R_0^{-1}$ small enough,
 then for some $t$ we have $\Gamma g u(t)\in\Siegel(R_0)$ as claimed.

 Turning to (2) assume $\mu$ is $H$-invariant and $H$-ergodic, then by the Mautner phenomenon
 $U$ acts also ergodically. Therefore, by the pointwise ergodic theorem
 we can find some $x \in X $
 for which the ergodic averages along $u(t)$ for the characteristic functions of the sets $\Siegel (n)$
 converge to $\mu(\Siegel(n))$ for all integers $n\geq 1$. Combined with \eqref{DKM-nondivergence}
 this gives
 \[
  \mu\bigl(X\setminus \Siegel(n)\bigr)=\lim_{r\rightarrow\infty}\frac{1}{2r}\int_{-r}^r\chi_{X\setminus\Siegel(n)}(xu(t))\operatorname{d}\!t\ll n^{-\alpha}
 \]
 as required.
\end{proof}

%%%%%%%%%%%%%%%%%%%%%%%%%%%%%%%%%%%%%%%%%%%%%%%%%%%%%%%%%%%%%%%%%%%%%%%%%%%%%%
\section{Proof of \eqref{temp}: uniform spectral gap implies $L^p$ coefficients.}
This section is devoted to a discussion of how to extract \eqref{temp}
from statements in the literature; we thank Erez Lapid
for his help with the details.   We apologize to the reader
for the piecemeal nature of this proof. 

Essentially we make use of a simple ``quantization'' phenomenon:
any point in the unitary dual of a semisimple Lie group, at which the behavior of irreducible representations ``changes,'' is restricted to sets where certain parameters are restricted to be integral.  In particular, away from the identity representation,
such ``changes in behavior'' can only occur at certain discrete points; in combination with the Howe-Moore theorem, this will establish the result. 

See \cite[Theorem 2.4.2]{Cowling} for groups of real rank $\geq 2$
and \cite[Theorem 2.5.2]{Cowling} for the real rank one groups with property $T$.  We shall therefore discuss carefully the rank one groups that fail to have property $T$, namely, 
 $\mathrm{SO}(n,1)$
or $\mathrm{SU}(n,1)$. 

We use Theorem 6 of \cite{KS}, the Howe-Moore theorem,  and
asymptotics for matrix coefficients.   More precisely:

Non-tempered unitary representations may be expressed as Langlands quotients $J(\sigma, z)$, where $\sigma$ is a (finite-dimensional) representation of the 
Levi of the $\R$-parabolic subgroup, and $z \in \R$.
 We normalize
matters so that $\sigma$ trivial, $z=1$ corresponds to the trivial representation.

Suppose we are given a -- not necessary irreducible -- unitary representation $V$
of $G$ that possesses a spectral gap. The definition of ``spectral gap'' easily implies that there exists $\delta  > 0$ so that any irreducible constituent
of $V$ (i.e., any irreducible representation that occurs in the support
of its unitary decomposition) is one of the following types:
\begin{enumerate}
\item Tempered;
\item $J(\sigma, z)$ where $\sigma$ is not the trivial representation;
\item $J(\sigma, z)$ where $\sigma$ is the trivial representation,
and $|z| <  1-\delta$.
\end{enumerate}

By the asymptotic expansion of matrix coefficients (see \cite[Theorem 8.32]{Knapp-red}) we deduce the following: For any $\delta > 0$,
there exists $p$ so that any representation $J(\sigma, z)$ with $0 \leq z \leq 1-\delta$ is $\frac{1}{p}$-tempered. It remains,
then, to show that there exists a uniform $p$ so that:
\begin{equation} \label{desid2}
\mbox{Any unitary $J(\sigma, z)$, where $\sigma$ is nontrivial, is $1/p$-tempered. }
\end{equation}

Fix $\sigma$ nontrivial. 
Set $z_0$ to be the supremum
of those $z \geq 0$ for which $J(\sigma, z)$ is unitary.
Proposition 45 and Theorem 6 of \cite{KS}, together, imply that 
$z_0 = z_c$, the ``critical abscissa'' of \cite{KS} defined
prior to (9.1), 
and moreover $J(\sigma, z_0)$ is itself unitary (i.e, unitarizable). 
We shall prove in a moment that $z_0 < 1$. Assuming this,
the definition of critical abscissa shows
 $z_0$ belongs to a discrete subset of $[0,1)$, independent of $\sigma$.
The asymptotic expansion of matrix coefficients, again,
shows \eqref{desid2}. 

It remains to show that $J(\sigma, z)$ can only be unitary for $z < 1$.
But that is, again, a consequence of the asymptotic expansion
of matrix coefficients -- \cite[Theorem 8.32]{Knapp-red}
-- in combination with the Howe-Moore theorem
that asserts that matrix coefficients decay at $\infty$.  
(One needs to know in addition that the leading exponent
survives in the Langlands quotient, which is \cite[Proposition 8.61]{Knapp-red}).

%%%%%%%%%%%%%%%%%%%%%%%%%%%%%%%%%%%%%%%%%%%%%%%%%%%%%%%%%%%%%%%%%%%%%%%%%%%%%%%

%%%%%%%%%%%%%%%%%%%%%%%%%%%%%%%%%%%%%%%%%%%%%%%%%%%%%%%%%%%%%%%%%%%%%%%%%%%%%%%%

\bibliographystyle{plain}

\printindex

\end{document}